\documentclass[hidelinks,onefignum,onetabnum]{siamart220329}

\usepackage{bm}
\usepackage{bbm}
\usepackage{algorithm}  
\usepackage{algorithmic}
\usepackage{hyperref}
\usepackage{subfigure}
\usepackage{graphicx}

\numberwithin{equation}{section}

\usepackage{cite}
\usepackage{booktabs}
\usepackage{makecell}
\usepackage{multirow}
\numberwithin{figure}{section}
\numberwithin{table}{section}
\usepackage{float}

\newtheorem{assumption}{Assumption}[section]

\usepackage{lipsum}
\usepackage{amsfonts}
\usepackage{graphicx}
\usepackage{epstopdf}
\usepackage{algorithmic}

\ifpdf
  \DeclareGraphicsExtensions{.eps,.pdf,.png,.jpg}
\else
  \DeclareGraphicsExtensions{.eps}
\fi

\newsiamremark{remark}{Remark}
\newsiamremark{hypothesis}{Hypothesis}
\crefname{hypothesis}{Hypothesis}{Hypotheses}
\newsiamthm{claim}{Claim}

\headers{Spectral analysis of a mixed method for linear elasticity}{Xiang Zhong and Weifeng Qiu}

\title{Spectral analysis of a mixed method for linear elasticity\thanks{Submitted to the editors DATE.
\funding{Xiang Zhong and Weifeng Qiu's research was partially supported by the Research
	Grants Council of the Hong Kong Special Administrative Region, China. (Project No. CityU 11302219).}}}

\author{Xiang Zhong\thanks{Department of Mathematics, City University of Hong Kong, Kowloon Tong, Hong Kong SAR, China
  (\email{xianzhong5-c@my.cityu.edu.hk},\email{weifeqiu@cityu.edu.hk}). Weifeng Qiu is the corresponding author.}
\and Weifeng Qiu\footnotemark[2]}

\usepackage{amsopn}

\ifpdf
\hypersetup{
  pdftitle={Spectral analysis of a mixed method for linear elasticity},
  pdfauthor={Xaing Zhong and Weifeng Qiu}
}
\fi

\begin{document}
	
\maketitle
% REQUIRED
\begin{abstract}
The purpose of this paper is to analyze a mixed method for linear elasticity eigenvalue problem, which approximates 
numerically the stress, displacement, and rotation, by piecewise $(k+1)$, $k$ and $(k+1)$-th degree polynomial 
functions ($k\geq 1$), respectively. The numerical eigenfunction of stress is symmetric. 
By the discrete $H^1$-stability of numerical displacement, we prove an $O(h^{k+2})$ approximation to 
the $L^{2}$-orthogonal projection of the eigenspace of exact displacement for the eigenvalue problem, 
with proper regularity assumption. Thus via postprocessing, we obtain a better approximation to the eigenspace of exact 
displacement for the eigenproblem than conventional methods.  We also prove that numerical approximation to 
the eigenfunction of stress is locking free with respect to Poisson ratio. 
We introduce a hybridization to reduce the mixed method to a condensed eigenproblem and 
prove an $O(h^2)$ initial approximation (independent of the inverse of  the elasticity operator)  of the eigenvalue 
for the nonlinear eigenproblem by using the discrete $H^1$-stability of numerical displacement, 
while only an $O(h)$ approximation can be obtained if we use the traditional inf-sup condition. 
Finally, we report some numerical experiments.
\end{abstract}

% REQUIRED


\section{Introduction}
In this paper, we consider a mixed finite element approximation of an eigenvalue problem in linear elasticity for two and 
three dimensional domains.
The mixed method we use is introduced in \cite[Section 6]{JJ2012} on special grids presented in \cite[Assumption 6.2]{JJ2012}. It's also well known that the mixed method 
allows us to deal safely with nearly incompressible materials.

The numerical approximation of eigenvalue problems has attracted extensive research interest, for instance, \cite{S2022,SDR2015,ADR2007,SMe2014,FG2021,DS2016,DF2020,SM2022,SDR2013,SDR2019}. Particular interest in the mixed finite element analysis of elasticity eigenproblem can be found in \cite{S2022,ADR2007,SMe2014,DS2016,SM2022,SDR2013,SDR2019}. In these references, they show that the proposed scheme provides a correct approximation of the spectrum. However, the underlying solution operators corresponding to their mixed formulations are not compact. So the analysis of the resulting discrete eigenproblems doesn't fit in the standard spectral framework(see \cite{PG2013,Osborn1975}). %For convenience, we mainly refer to \cite{SDR2013,SDR2019}. 
We recalled the eigenvalue problem of \cite{SDR2013,SDR2019,DS2016,SM2022}: Find the stress $\underline{\bm{\sigma}}:\Omega\to\mathbb{R}^{n\times n}$ symmetric, $\underline{\bm{\rho}}:\Omega\to\mathbb{R}^{n\times n}$ skew symmetric and the corresponding natural frequencies $\omega\in\mathbb{R}$ such that
\begin{subequations}
	\label{previous pde}
	\begin{align}		
		-\nabla(\rho_S^{-1}\textbf{div}\underline{\bm{\sigma}})=\omega^2(\mathcal{A}\underline{\bm{\sigma}}+\underline{\bm{\rho}})\quad &\rm in\quad\Omega,\label{previous pde_a}\\
		\textbf{div}\underline{\bm{\sigma}}=\mathbf{0}\quad  &\rm on\quad\Gamma_0,\label{previous pde_b}\\
		\underline{\bm{\sigma}}\mathbf{n}=\mathbf{0}\quad &\rm on\quad\Gamma_1, \label{previous pde_c}
	\end{align}	
\end{subequations}		
where $\underline{\bm{\rho}}:=\nabla\mathbf{u}-\underline{\bm{\epsilon}}(\mathbf{u})$ is the rotation and $\mathbf{u}$ denotes the displacement, $\underline{\bm{\epsilon}}(\mathbf{u})=\frac{1}{2}(\nabla\mathbf{u}+(\nabla\mathbf{u})^t)$ the linearized strain tensor. $\Omega\subset \mathbb{R}^n$ 
($n=2,3$) is a bounded and connected Lipschitz polyhedral domain occupied by an isotropic and linearly elastic
solid. $\partial\Omega$ is the boundary of the domain $\Omega$, which admits
a disjoint partition $\partial\Omega=\Gamma_0\cup\Gamma_1$. For  simplicity, we also assume that $\Gamma_0$ has a positive measure. $\mathbf{n}$ is the outward unit normal vector to $\partial\Omega$. $\rho_S$ is the density of the material, which we assume a strictly positive constant. $\mathcal{A}$ is the inverse of the elasticity operator, which is given by
\begin{equation}
	\label{1220a}
	\mathcal{A}\underline{\bm{\tau}}:=\frac{1}{2\mu_S}\underline{\bm{\tau}}^D+\frac{1}{n(n\lambda_S+2\mu_S)}(tr\underline{\bm{\tau}})\bm{\mathit{I}},
\end{equation}
where $\lambda_S$ and $\mu_S$ are the Lam$\rm \acute{e}$ constants. $\bm{\mathit{I}}$ is the identity matrix of $\mathbb{R}^{n\times n}$. And we denote the deviatoric tensor $\underline{\bm{\tau}}^D:=\underline{\bm{\tau}}-\frac{1}{n}(tr\underline{\bm{\tau}})\bm{\mathit{I}}$. For nearly incompressible materials, $\lambda_S$ is large in comparison with $\mu_S$.

Observing (\ref{previous pde}), the underlying solution operator is clearly non-compact since it acts on the stress $\underline{\bm{\sigma}}$, which is defined in $\bm{\mathit{H}}(div;\Omega)$. We recall that \cite{DS2016} utilizes the lowest order mixed finite element of the PEERS; \cite{SDR2013,SM2022} use the Arnold-Falk-Winther element; \cite{S2022,SDR2019} analyze the elasticity eigenproblem by mixed DG methods. In these references, the displacement can be recovered and post-processed at the discrete level by using identity ``$\omega^2\mathbf{u}=-\rho_S^{-1}\textbf{div}\underline{\bm{\sigma}}$''. By utilizing the techniques 
in \cite{SDR2013,SDR2019,S2022,DS2016,SM2022}, it is easy to show 
{\begin{equation}
		\label{previous estimate}
		\widetilde{\delta}(R(\bm{\mathit{E}}),  R(\bm{\mathit{E}}_h))\leq Ch^{\rm min\{t,k+1\}},
	\end{equation}
	where $R(\bm{\mathit{E}})$, $R(\bm{\mathit{E}}_h)$ are the eigenspaces of the exact and numerical displacements. $\widetilde{\delta}(R(\bm{\mathit{E}}),  R(\bm{\mathit{E}}_h))$ is the gap between $R(\bm{\mathit{E}})$ and $R(\bm{\mathit{E}}_h)$ under $L^2$-norm.} $k$ is the polynomial degree of numerical displacement and $t$ is a constant associated with the regularity. Moreover, under discrete $H^1$-norm(see (\ref{0112b})), we can only obtain($\hat{\delta}(\cdot,\cdot)$ is the gap under discrete $H^1$-norm) 
\begin{equation}
	\label{previous estimate_discrete H1}
	\hat{\delta}(\bm{\mathit{P}}(R(\bm{\mathit{E}})),  R(\bm{\mathit{E}}_h))\leq Ch^{\rm min\{t,k\}},
\end{equation}
where $\bm{\mathit{P}}$ is the $\bm{\mathit{L}}^2$-orthogonal projection operator onto the finite element space for $\mathbf{u}_h$(i.e. the numerical eigenfunction of displacement).

Unlike \cite{SDR2013,SDR2019}, our formulation of the eigenvalue problem with natural frequencies $\omega>0$ is given by
\begin{subequations}
	\label{our pde}
	\begin{align}		
		\mathcal{A}\underline{\bm{\sigma}}-\nabla\mathbf{u}+\underline{\bm{\rho}}=\mathbf{0}\quad &\rm in\quad\Omega ,\label{our pde_a}\\
		-\rho_S^{-1}\textbf{div}\underline{\bm{\sigma}}=\omega^2\mathbf{u}\quad &\rm in\quad\Omega,\label{oue pde_b}\\
		\underline{\bm{\sigma}}\mathbf{n}=\mathbf{0}\quad &\rm on\quad\Gamma_1,\label{our pde_c}\\
		\mathbf{u}=\mathbf{0}\quad &\rm on\quad\Gamma_0,\label{our pde_d}
	\end{align}	
\end{subequations}
It's clear that (\ref{our pde}) is equivalent to (\ref{previous pde}) and we can obtain the corresponding source problem by replacing $\omega^2\mathbf{u}$ with $\bm{\mathit{f}}\in\bm{\mathit{L}}^2(\Omega)$. The underlying solution operator $\bm{\mathit{T}}$ is compact from $\bm{\mathit{L}}^2(\Omega)$ into itself: $\mathbf{u}=\bm{\mathit{T}}\bm{\mathit{f}}$, which is analogous to the analyses in different or similar contexts (see, for instance, \cite[Chapter 11.3]{BBF2013} and \cite{BBD2021,FDR2021,JA2018}). Our mixed method approximates numerically the stress, displacement, and rotation, by piecewise 
$(k+1)$, $k$, and $(k+1)$-th degree polynomial functions ($k\geq 1$) on special grids(see section \ref{sec:The discrete problem} for more details). Actually, we select this special grid sequence because the method yields exactly symmetric stress approximations.

More importantly, we utilize discrete $H^1$-stability of numerical displacement to analyze the linear elasticity eigenvalue problem, which doesn't exist before. %The innovation ``discrete $H^1$-stability of numerical displacement'' runs through our entire article.
By this new way of analysis, we provide a better estimate for the eigenvalue problem:
{\begin{equation}
		\label{our estimate}
		\hat{\delta}(\bm{\mathit{P}}(R(\bm{\mathit{E}})),  R(\bm{\mathit{E}}_h))\leq Ch^{\rm min\{t,k+2\}},
	\end{equation}
	where the gap $\hat{\delta}(\cdot,\cdot)$ is defined under discrete $H^1$ norm(i.e.$||\cdot||_{1,h}$).}  $||\mathbf{u}_h||_{1,h}$ is denoted as
\begin{equation}
	\label{0112b}
	||\mathbf{u}_h||_{1,h}^2:=\sum_{K\in\mathcal{T}_h}\int_K|\nabla\mathbf{u}_h|^2dx+\sum_{F\in\varepsilon^{I}\cup\varepsilon^{D}}\frac{1}{h_F}\int_F|[\![\mathbf{u}_h]\!]|^2dx, 	
\end{equation}
where $h_F$ denotes the size of the face $F$. For two adjacent elements $K$ and $K'$ sharing the same face $F$, the jump of $\mathbf{u}_h$ across $F$ is defined by 
$[\![\mathbf{u}_h]\!]:=\mathbf{u}_h|_{\partial K\cap F}-\mathbf{u}_h|_{\partial K'\cap F}.$
In the case of $F\in\partial K$ lying on $\partial\Omega$, we define $[\![\mathbf{u}_h]\!]:=\mathbf{u}_h|_{\partial K\cap F}$. $\varepsilon^{I}$ is the set of all interior edges,  $\varepsilon^{D}:=\{F: F\in\Gamma_0\}$, $\varepsilon^{N}:=\{F: F\in\Gamma_1\}$ and $\varepsilon_h=\varepsilon^{I}\cup\varepsilon^{D}\cup\varepsilon^{N}$.
Then by postprocessing,
\begin{equation}
	\label{0110c}
	\hat{\delta}_1(R(\bm{\mathit{E}}),R(\hat{\bm{\mathit{E}}}_h))\leq Ch^{k+2},\hspace{1em}\widetilde{\delta}(R(\bm{\mathit{E}}),R(\hat{\bm{\mathit{E}}}_h))\leq Ch^{k+3} 
\end{equation}
with proper regularity assumption. Here $\hat{\delta}_1(\cdot,\cdot)$ is the gap under the norm $(\sum_{K\in\mathcal{T}_h}$ $||\nabla(\cdot)||_{0,K}^2)^{\frac{1}{2}}$ and $R(\hat{\bm{\mathit{E}}}_h)$ is the postprocessed eigenspace(see section \ref{sec: Postprocessing} for more details). Comparing (\ref{previous estimate}), (\ref{previous estimate_discrete H1}) and (\ref{our estimate}), (\ref{0110c}), our analysis based on discrete $H^{1}$-stability of numerical displacement provides a better convergence for the eigenfunction of displacement. %We point out that we could as well have chosen other finite element methods, such as the Arnold-Falk-Winther element, to obtain the same stability properties and error estimates.

We also point out that our mixed formulation is equivalent to that of \cite{SDR2013,SDR2019}(see Remark \ref{added_a} for more details). Therefore, utilizing our analysis to the system of \cite{SDR2013,SDR2019}, we can obtain all results of \cite{SDR2013,SDR2019} and get some better estimates than \cite{SDR2013,SDR2019}: discrete $H^1$ stability properties and error estimates as (\ref{our estimate}), (\ref{0110c}).

Another contribution of this paper is that we introduce a dimension-reduced
implementation by hybridization and give a good initial approximation of the eigenvalue for the nonlinear eigenvalue problem. The mixed approximation in \cite{SDR2013,SDR2019} is formulated by the stress tensor and the rotation. Our mixed method approximates numerically the stress, displacement, and rotation, by piecewise 
$(k+1)$, $k$, and $(k+1)$-th degree polynomial functions ($k\geq 1$), respectively. Though the system looks more complicated,  
hybridization reduces the approximation to a condensed eigenproblem. The condensed eigenproblem is nonlinear, but much smaller than the original mixed approximation. A related hybridization was introduced in \cite{CJL2010} for Poisson equations.  We take advantage of the
discrete $H^1$-stability of numerical displacement to obtain an $O(h^2)$ initial approximation of the  eigenvalue for the nonlinear eigenproblem. All analyses here don't depend on the inverse of $\mathcal{A}$. 
We point out that if we use the analysis in \cite{CJL2010}, we can not obtain the approximation of the rotation because the rotation will be eliminated during the energy argument; Then by the traditional inf-sup condition, we can only obtain an $O(h)$ estimate. By contrast, ``discrete $H^1$-stability of numerical displacement'' helps to give a more detailed analysis(see Remarks \ref{1222l} and \ref{novelty remark} for more details).

The paper is organized as follows. In section \ref{sec: The spectral problem}, we introduce a mixed formulation with reduced symmetry of the eigenvalue elasticity problem and the corresponding solution
operator. Section \ref{sec: Spectral characterization} is devoted to giving the spectrum of the compact solution operator. In section \ref{sec:The discrete problem}, we introduce the discrete eigenvalue problem, describe 
the spectrum of the discrete solution operator, and prove that the discrete solution operators converge to the compact solution operator in discrete $H^1$-norm by using the discrete $H^1$-stability of numerical displacement. In section \ref{sec: Spectral approximation}, we approximate the $L^2$-orthogonal projection of the
eigenspace of exact displacement in discrete $H^1$-norm. Asymptotic error estimates for the eigenvalues and the stress of the eigenvalue elasticity problem will be also established. In section \ref{sec: Hybridization}, we present a hybridization for the eigenvalue problem and give an $O(h^2)$ initial approximation of the eigenvalue for the nonlinear eigenproblem. In section \ref{sec: Postprocessing},
we provide a post-processing that gives an improved approximation to the eigenspace of exact displacement. Finally, we present
in section \ref{sec: Numerical results} numerical experiments to confirm the
theoretical results proved above.

We end this section with some notations. Given any Hilbert space $\mathit{V}$, let $\bm{\mathit{V}}$ and $\underline{\bm{\mathit{V}}}$ denote, respectively, the space of vector-valued and matrix-valued functions with entries in $\mathit{V}$.   
$\mathbb{C}$ represents the set of all complex numbers. In particular, $\underline{\bm{\mathit{H}}}(\textbf{div};\Omega)$ is the space of matrix-valued functions such that each row belongs to $\bm{\mathit{H}}(\textbf{div};\Omega)$. Given $\underline{\bm{\tau}}:=(\tau_{ij})$, we define as usual the transpose matrix $\underline{\bm{\tau}}^t:=(\tau_{ji})$, the trace $tr\underline{\bm{\tau}}:=\sum_{i=1}^n\tau_{ii}$. $(\bm{\zeta},\bm{\theta})_D,(\underline{\bm{\zeta}}$, $\underline{\bm{\theta}})_D$ denote the integral over $D$ of $\bm{\zeta}\cdot\bm{\theta}$, $\underline{\bm{\zeta}}:\underline{\bm{\theta}}$ ( ``$\cdot$'' and ``:'' denote the dot product and the Frobenius inner-product respectively). For $s>0$, $||\cdot||_{s,\Omega}$ stands for the norm of the Hilbertian Sobolev spaces $\mathit{H}^s(\Omega)$, $\bm{\mathit{H}}^s(\Omega)$, $\underline{\bm{\mathit{H}}}^s(\Omega)$, with the convention $\mathit{H}^0(\Omega):=\mathit{L}^2(\Omega)$. We also define for $s\geq0$ the Hilbert space $\underline{\bm{\mathit{H}}}^s(\textbf{div};\Omega):=\{\underline{\bm{\tau}}\in\underline{\bm{\mathit{H}}}^s(\Omega): \textbf{div}\underline{\bm{\tau}}\in\bm{\mathit{H}}^s(\Omega)\}$, whose norm is given by $||\underline{\bm{\tau}}||_{\underline{\bm{\mathit{H}}}^s(\textbf{div};\Omega)}^2:=||\underline{\bm{\tau}}||_{s,\Omega}^2+||\textbf{div}\underline{\bm{\tau}}||_{s,\Omega}^2$ and let $\underline{\bm{\mathit{H}}}^0(\textbf{div};\Omega):=\underline{\bm{\mathit{H}}}(\textbf{div};\Omega)$. 
%We also define for $s>0$, the Hilbert space $\underline{\bm{\mathit{H}}}^s(\textbf{div};\Omega):=\{\underline{\bm{\tau}}\in\underline{\bm{\mathit{H}}}^s(\Omega):\textbf{div}\underline{\bm{\tau}}\in\bm{\mathit{H}}^s(\Omega)\}$, whose norm is given by $||\underline{\bm{\tau}}||_{\underline{\bm{\mathit{H}}}^s(\textbf{div};\Omega)}^2:=||\underline{\bm{\tau}}||_{s,\Omega}^2+||\textbf{div}\underline{\bm{\tau}}||_{s,\Omega}^2$ and 
In addition, we denote generic constants $C$, taking different values at different places.

%\lipsum[2-3]

% The outline is not required, but we show an example here.
%The paper is organized as follows. Our main results are in
%\cref{sec:main}, our new algorithm is in \cref{sec:alg}, experimental
%results are in \cref{sec:experiments}, and the conclusions follow in
%\cref{sec:conclusions}.

\section{The spectral problem}
\label{sec: The spectral problem}

In this section, we will introduce a mixed formulation
with reduced symmetry of the eigenvalue elasticity problem and define the corresponding solution
operator.
We aim to employ a mixed approach to derive a variational formulation of this problem. The stress tensor $\underline{\bm{\sigma}}$ will be sought in 
$\underline{\bm{\mathit{Y}}}:=\{\underline{\bm{\tau}}\in\underline{\bm{\mathit{H}}}(\textbf{div};\Omega):\underline{\bm{\tau}}\mathbf{n}=\mathbf{0}\hspace{0.5em} {\rm on\hspace{0.5em}\Gamma_1}\},$ a closed subspace of $\underline{\bm{\mathit{H}}}(\textbf{div};\Omega)$.
We also introduce the space of skew symmetric tensors
$\underline{\bm{\mathit{Q}}}:=\{\bm{\mathit{s}}\in\underline{\bm{\mathit{L}}}^2(\Omega): \bm{\mathit{s}}=-\bm{\mathit{s}}^t\}$ and the rotation $\underline{\bm{\rho}}:=\frac{1}{2}(\nabla\mathbf{u}-(\nabla\mathbf{u})^t)$. 
In order to write the variational formulation of the spectral problem, we introduce the following bounded bilinear forms in $\underline{\bm{\mathit{Y}}}\times\bm{\mathit{L}}^2(\Omega)\times\underline{\bm{\mathit{Q}}}$:
$A((\underline{\bm{\sigma}},\mathbf{u},\underline{\bm{\rho}}),(\underline{\bm{\upsilon}},\bm{\omega},\underline{\bm{\eta}}))=(\mathcal{A}\underline{\bm{\sigma}},\underline{\bm{\upsilon}})_{\Omega}+(\mathbf{u},\textbf{div}\underline{\bm{\upsilon}})_{\Omega}+(\underline{\bm{\rho}},\underline{\bm{\upsilon}})_{\Omega}+(\rho_S^{-1}\textbf{div}\underline{\bm{\sigma}},\bm{\omega})_{\Omega}+(\underline{\bm{\sigma}},\underline{\bm{\eta}})_{\Omega},$
$B((\underline{\bm{\sigma}},\mathbf{u},\underline{\bm{\rho}}),(\underline{\bm{\upsilon}},\bm{\omega},\underline{\bm{\eta}}))=-(\mathbf{u},\bm{\omega})_{\Omega}.$
For simplicity, we also write $B((\underline{\bm{\sigma}},\mathbf{u},\underline{\bm{\rho}}),(\underline{\bm{\upsilon}},\bm{\omega},\underline{\bm{\eta}}))$ as $B(\mathbf{u},\bm{\omega})$. we can write the variational eigenvalue
problem in which $\lambda=\omega^2$ as
\begin{equation}
	\label{0112a}
	A((\underline{\bm{\sigma}},\mathbf{u},\underline{\bm{\rho}}),(\underline{\bm{\upsilon}},\bm{\omega},\underline{\bm{\eta}}))=\lambda B(\mathbf{u},\bm{\omega}),\quad \forall(\underline{\bm{\upsilon}},\bm{\omega},\underline{\bm{\eta}})\in\underline{\bm{\mathit{Y}}}\times\bm{\mathit{L}}^2(\Omega)\times\underline{\bm{\mathit{Q}}}
\end{equation} 
For problem (\ref{0112a}), we introduce the corresponding solution operator 
\begin{equation*}
	\begin{aligned}
		\bm{\mathit{T}}: \bm{\mathit{L}}^2(\Omega)& \rightarrow \bm{\mathit{L}}^2(\Omega) \\
		\bm{\mathit{f}}&\mapsto\mathbf{u}^{\mathit{f}}
	\end{aligned}	
\end{equation*}
where $\mathbf{u}^{\mathit{f}}$, also with $(\underline{\bm{\sigma}}^{\mathit{f}},\underline{\bm{\rho}}^{\mathit{f}})\in\underline{\bm{\mathit{Y}}}\times\underline{\bm{\mathit{Q}}}$ is the solution of the following source problem:
\begin{subequations}
	\label{c}
	\begin{align}		
		(\mathcal{A}\underline{\bm{\sigma}}^{\mathit{f}},\underline{\bm{\upsilon}})_{\Omega}+(\mathbf{u}^{\mathit{f}},\textbf{div}\underline{\bm{\upsilon}})_{\Omega}+(\underline{\bm{\rho}}^{\mathit{f}},\underline{\bm{\upsilon}})_{\Omega}&=0 \label{c_a}\\
		-(\rho_S^{-1}\textbf{div}\underline{\bm{\sigma}}^{\mathit{f}},\bm{\omega})_{\Omega}&=(\bm{\mathit{f}},\bm{\omega})_{\Omega} \label{c_b}\\
		(\underline{\bm{\sigma}}^{\mathit{f}},\underline{\bm{\eta}})_{\Omega}&=0  \label{c_c}
	\end{align}	
\end{subequations}
for all $(\underline{\bm{\upsilon}},\bm{\omega},\underline{\bm{\eta}})\in\underline{\bm{\mathit{Y}}}\times\bm{\mathit{L}}^2(\Omega)\times\underline{\bm{\mathit{Q}}}$. 
Notice that $(\lambda,\mathbf{u})\in \mathbb{R}\times\bm{\mathit{L}}^2(\Omega)$ solves ($\ref{0112a}$) if and only if $(\mu,\mathbf{u})$, with $\mu=\lambda^{-1}$ is an eigenpair of $\bm{\mathit{T}}$, i.e. if and only if $\bm{\mathit{T}}\mathbf{u}=\lambda^{-1}\mathbf{u} \quad (\lambda> 0).$
In this case, $(\underline{\bm{\sigma}}^{\mathit{f}}, \underline{\bm{\rho}}^{\mathit{f}})=\lambda^{-1}(\underline{\bm{\sigma}}, \underline{\bm{\rho}})$.
Clearly the linear operator $\bm{\mathit{T}}$ is well defined and bounded.

%\lipsum[6]
\section{Spectral characterization}
\label{sec: Spectral characterization}
In this section, the characterizations of the spectrum of $\bm{\mathit{T}}$ will be given by Theorem \ref{theorem3.1} and proper regularity results will also be assumed in Assumption \ref{0128b}.
%\subsection{The compactness of the solution operator $\bm{\mathit{T}}$}

Given any $\bm{\mathit{f}}\in\bm{\mathit{L}}^2(\Omega)$, we have $\bm{\mathit{T}}\bm{\mathit{f}}=\mathbf{u}^{\mathit{f}}$; By testing (\ref{c_a}) with $\underline{\bm{\upsilon}}\in\underline{\bm{\mathcal{D}}}(\Omega)\cap\underline{\bm{\mathit{Y}}}$, where $\underline{\bm{\mathcal{D}}}(\Omega)=\{\underline{\bm{\phi}}\in\underline{\bm{\mathcal{C}}}^{\infty}(\Omega)$, supp$\underline{\bm{\phi}}$ is  a  compact subset of $\Omega$\}, we have that
$\mathcal{A}\underline{\bm{\sigma}}^{\mathit{f}}-\nabla\mathbf{u}^{\mathit{f}}+\underline{\bm{\rho}}^{\mathit{f}}=\mathbf{0}$.
Then we have $\mathbf{u}^{\mathit{f}}\in\bm{\mathit{H}}^1_{0,\Gamma_0}(\Omega):=\{\mathbf{u}\in\bm{\mathit{H}}^1(\Omega):\mathbf{u}=\mathbf{0}\hspace{0.5em} \rm on\hspace{0.5em}\Gamma_0\}$, which is a closed subspace of $\bm{\mathit{H}}^1(\Omega)$. 
Problem (\ref{c}) is the well-known dual-mixed formulation with weakly imposed symmetry of the classical elasticity source problem (i.e. (\ref{our pde}) with $(\omega^2\mathbf{u},\underline{\bm{\sigma}},\mathbf{u},\underline{\bm{\rho}})$ replaced by $(\bm{\mathit{f}},\underline{\bm{\sigma}}^{\mathit{f}},\mathbf{u}^{\mathit{f}},\underline{\bm{\rho}}^{\mathit{f}})$).
Clearly, $\bm{\mathit{T}}: \bm{\mathit{L}}^2(\Omega)\rightarrow \bm{\mathit{H}}^1_{0,\Gamma_0}(\Omega)\subset\bm{\mathit{L}}^2(\Omega)$ is compact.

%\subsection{Characterizations of spectrum of the solution operator \texorpdfstring{{\boldmath$\bm{\mathit{T}}$}}{T}}
%\label{sec: Characterizations}
\begin{theorem}
	\label{theorem3.1}
	The spectrum of $\bm{\mathit{T}}$ decomposes as follows: sp$(\bm{\mathit{T}})$=$\{0\}\cup\{\mu_k\}_{k=1}^{\infty}$, where
	
	{\rm(i)} $\{\mu_k\}_{k=1}^{\infty}\subset(0,\infty)$ is an infinite sequence of eigenvalues of $\bm{\mathit{T}}$ and there exists an infinite sequence $\{\bm{\mathit{p}}_n\}_{n=1}^{\infty}$ of corresponding eigenvectors in $\bm{\mathit{H}}^1_{0,\Gamma_0}(\Omega)$ that satisfy 
	$\mu_1=||\bm{\mathit{T}}||,\hspace{0.5em} \mu_1\geq\mu_2\geq\cdot\cdot\cdot\geq\mu_n\geq\cdot\cdot\cdot,\hspace{0.5em} \mu_n\neq0\hspace{0.5em} \rm for\hspace{0.5em} \rm all\hspace{0.5em} \rm n\geq 1,\hspace{0.5em} \rm\lim\limits_{n\to\infty}\mu_n=0,$$
	$$\bm{\mathit{T}}\bm{\mathit{p}}_n=\mu_n\bm{\mathit{p}}_n\hspace{0.5em}  for\hspace{0.5em}  all\hspace{0.5em}  n\geq 1\hspace{0.5em}  and \hspace{0.5em} (\bm{\mathit{p}}_{\mathit{k}},\bm{\mathit{p}}_{\mathit{l}})_{\Omega}=\delta_{\mathit{k}\mathit{l}}\hspace{0.5em}  for\hspace{0.5em}  all\hspace{0.5em}  \mathit{k},\mathit{l}\geq 1.$
	Moreover, the ascent of each eigenvalue is 1.
	{\rm(ii)} $\mu=0$ is not an eigenvalue of $\bm{\mathit{T}}$.	
\end{theorem}
\begin{proof}
	Since $\bm{\mathit{T}}:\bm{\mathit{L}}^2(\Omega)\to\bm{\mathit{L}}^2(\Omega)$ is compact and self-adjoint(with respect to $\mathit{L}^2$-norm), then by \cite[Theorem 4.11-1]{PG2013}, we obtain (i). Moreover, the ascent of each eigenvalue is 1 (see, for instance \cite{Osborn1975}). For (ii), since $\bm{\mathit{L}}^2(\Omega)$ is infinite-dimensional, then $\mu=0$ is a spectrum of $\bm{\mathit{T}}$ (If $0\notin sp(\bm{\mathit{T}})$, then 0 belongs to the resolvent set of $\bm{\mathit{T}}.$ By the definition of resolvent set(see \cite[p.90]{App}), $\bm{\mathit{T}}^{-1}$ exists and is bounded from $R(\bm{\mathit{T}})$(i.e. the range of $\bm{\mathit{T}}$) to $\bm{\mathit{L}}^2(\Omega)$. So $\bm{\mathit{T}}\bm{\mathit{T}}^{-1}=\bm{\mathit{I}}$ is compact. It's a contradiction because $\bm{\mathit{L}}^2(\Omega)$ is infinite dimensional). Suppose $\exists\mu=0$ and $\bm{\mathit{f}}\neq\mathbf{0}$ such that $\mathbf{u}^{\mathit{f}}=\bm{\mathit{T}}\bm{\mathit{f}}=\mu\bm{\mathit{f}}=\mathbf{0}$. Then $(\underline{\bm{\sigma}}^{\mathit{f}},\underline{\bm{\rho}}^{\mathit{f}})=\mu(\underline{\bm{\sigma}},\underline{\bm{\rho}})=\mathbf{0}.$ Thus $\bm{\mathit{f}}=-\rho_S^{-1}\rm \textbf{div}\underline{\bm{\sigma}}^{\mathit{f}}=\mathbf{0} $ by (\ref{c_b}), which is a contradiction.
\end{proof}
%\subsection{Regularity assumption}
\begin{assumption}
	\label{0128b}
	(see\cite{Dau1988,Gri1986,SDR2019})$\exists$ $s_0\in(0,1]$ such that for all $s\in(0,s_0]$: There exists a constant $C^{reg}>0$, independent of $\lambda_S$, such that for all $\bm{\mathit{f}}\in\bm{\mathit{L}}^2(\Omega)$, if $(\underline{\bm{\sigma}}^{\mathit{f}},\mathbf{u}^{\mathit{f}},\underline{\bm{\rho}}^{\mathit{f}})\in\underline{\bm{\mathit{Y}}}\times\bm{\mathit{L}}^2(\Omega)\times\underline{\bm{\mathit{Q}}}$	is the solution to (\ref{c}), then
	$||\underline{\bm{\sigma}}^{\mathit{f}}||_{s,\Omega}+||\mathbf{u}^{\mathit{f}}||_{1+s,\Omega}+||\underline{\bm{\rho}}^{\mathit{f}}||_{s,\Omega}\leq C^{reg}||\bm{\mathit{f}}||_{0,\Omega}.$
	Consequently, $(\underline{\bm{\sigma}}^{\mathit{f}},\underline{\bm{\rho}}^{\mathit{f}})\in\underline{\bm{\mathit{H}}}^s(\Omega)\times\underline{\bm{\mathit{H}}}^s(\Omega),$ $\bm{\mathit{T}}(\bm{\mathit{L}}^2(\Omega))\subset{\bm{\mathit{H}}}^{1+s}(\Omega)$.
\end{assumption}

\section{The discrete problem}
\label{sec:The discrete problem}

%\lipsum[40]
In this section, we will introduce the discrete eigenvalue problem. Based on Theorem \ref{1220gg}(discrete $H^1$-stability of numerical displacement), 
we provide a better convergence of the discrete solution operators in Theorem \ref{1221a}. The spectrum of the discrete solution operator is presented in Theorem \ref{spectrum of T_h}.
\subsection{The discrete eigenvalue problem}
\label{sec: The discrete eigenvalue problem}
We introduce the discrete counterpart of problem ($\ref{0112a}$): Find $\lambda_h\in\mathbb{R}$ and $\mathbf{0}\neq(\underline{\bm{\sigma}}_h,\mathbf{u}_h,\underline{\bm{\rho}}_h)\in\underline{\bm{\mathit{V}}}^h\times\bm{\mathit{W}}^h\times\underline{\bm{\mathit{A}}}^h$ such that
\begin{subequations}
	\label{aa}
	\begin{align}		
		(\mathcal{A}\underline{\bm{\sigma}}_h,\underline{\bm{\upsilon}}_h)_{\Omega}+(\mathbf{u}_h,\textbf{div}\underline{\bm{\upsilon}}_h)_{\Omega}+(\underline{\bm{\rho}}_h,\underline{\bm{\upsilon}}_h)_{\Omega}&=0 \\
		-(\rho_S^{-1}\textbf{div}\underline{\bm{\sigma}}_h,\bm{\omega}_h)_{\Omega}&=\lambda_h(\mathbf{u}_h,\bm{\omega}_h)_{\Omega} \\
		(\underline{\bm{\sigma}}_h,\underline{\bm{\eta}}_h)_{\Omega}&=0  
	\end{align}	
\end{subequations}
for all $(\underline{\bm{\upsilon}}_h,\bm{\omega}_h,\underline{\bm{\eta}}_h)\in\underline{\bm{\mathit{V}}}^h\times\bm{\mathit{W}}^h\times\underline{\bm{\mathit{A}}}^h$. Here $\underline{\bm{\mathit{V}}}^h:=\{\underline{\bm{\upsilon}}\in\underline{\bm{\mathit{Y}}}:\underline{\bm{\upsilon}}|_K\in\underline{\bm{\mathit{V}}}(K)$ for all $K\in\mathcal{T}_h\}$,
$\bm{\mathit{W}}^h:=\{\bm{\omega}\in\bm{\mathit{L}}^2(\Omega):\bm{\omega}|_K\in\bm{\mathit{W}}(K)\hspace{0.5em} {\rm for \hspace{0.5em} all}\hspace{0.5em} K\in\mathcal{T}_h\},
\underline{\bm{\mathit{A}}}^h:=\{\underline{\bm{\upsilon}}\in\underline{\bm{\mathit{L}}}^2(\Omega):\underline{\bm{\upsilon}}|_K\in\underline{\bm{\mathit{A}}}(K)\hspace{0.5em} {\rm for\hspace{0.5em} all}\hspace{0.5em} K\in\mathcal{T}_h\}.$
The local spaces are
\begin{equation*}
	\begin{aligned}
		\underline{\bm{\mathit{V}}}(K)&:=\underline{\bm{\mathit{V}}}^k(K):=\underline{\bm{\mathit{p}}}^{k+1}(K)\rm\hspace{0.5em}on\hspace{0.5em}special\hspace{0.5em}grids, \\
		\bm{\mathit{W}}(K)&:=\bm{\mathit{W}}^k(K):=\bm{\mathit{p}}^k(K),\quad
		\underline{\bm{\mathit{A}}}(K):=\underline{\bm{\mathit{A}}}^k(K):=\{\underline{\bm{\eta}}\in\underline{\bm{\mathit{p}}}^{k+1}(K):\underline{\bm{\eta}}+\underline{\bm{\eta}}^t=\mathbf{0}\},
	\end{aligned}
\end{equation*}
where the special grids on $\mathcal{T}_h$ are defined by \cite[Assumption 6.2]{JJ2012}: In either the
two-dimensional(triangular) or the three-dimensional(tetrahedral) case, assume the following.

(1) The grid is obtained from a quasi-uniform grid after splitting each of its elements into $n+1$
elements (Here $n$ is the dimension) by connecting the vertices of the element to its barycentre.

(2) In the two-dimensional case assume $k\geq0$; In the three-dimensional case assume $k\geq1$.

Notice that $\underline{\bm{\mathit{V}}}^h\times\bm{\mathit{W}}^h\times\underline{\bm{\mathit{A}}}^h$ is the mixed finite element of the family introduced for linear elasticity by J.Gopalakrishnan and J.Guzm$\rm \acute{a}$n (see \cite[section 6]{JJ2012}).
\begin{remark}
	\label{added_a}
	We point out that we can obtain the same formulations as \cite[Problem 4.1]{SDR2013} by a simple algebraic operation in (\ref{aa}):
	\begin{equation*}
		\label{aa_1}
		\begin{aligned}	
			(\rho_S^{-1}\textbf{div}\underline{\bm{\sigma}}_h,\textbf{div}\underline{\bm{\upsilon}}_h)_{\Omega}&=\lambda_h((\mathcal{A}\underline{\bm{\sigma}}_h,\underline{\bm{\upsilon}}_h)_{\Omega}+(\underline{\bm{\rho}}_h,\underline{\bm{\upsilon}}_h)_{\Omega}) \\
			\lambda_h(\underline{\bm{\sigma}}_h,\underline{\bm{\eta}}_h)_{\Omega}&=0  
		\end{aligned}	
	\end{equation*}
	for all $(\underline{\bm{\upsilon}}_h,\underline{\bm{\eta}}_h)\in\underline{\bm{\mathit{V}}}^h\times\underline{\bm{\mathit{A}}}^h$.	We use the finite element on special grids described as above while \cite{SDR2013} utilizes the lowest order AFW element on arbitrary quasi-uniform grids. However, based on the equivalence of the mixed formulations, the discrete $H^1$ stability properties and error estimates obtained by our mixed method can also be applied to \cite{SDR2013}.
\end{remark}

Let us now recall some well-known approximation properties of the finite element
spaces introduced above. Given $t>0$, the tensorial version of the BDM-interpolation operator (see \cite{BF1991}) $\mathbf{\Pi}_h: \underline{\bm{\mathit{H}}}^t(\Omega)\cap\underline{\bm{\mathit{Y}}}\to \underline{\bm{\mathit{V}}}^h$ is characterized by the following identities:
\begin{equation}
	\label{BDM}
	\begin{aligned} \int_F(\mathbf{\Pi}_h\underline{\bm{\mathit{\tau}}})\mathbf{n}_F\cdot\bm{\mathit{p}}&=\int_F\underline{\bm{\mathit{\tau}}}\mathbf{n}_F\cdot\bm{\mathit{p}},\hspace{0.5em}\forall\bm{\mathit{p}}\in\bm{\mathit{p}}^{k+1}(F),\hspace{0.5em} {\rm for \hspace{0.5em} all\hspace{0.5em} faces}\hspace{0.5em} F \hspace{0.5em} {\rm of} \hspace{0.5em} K,\\
		\int_K \mathbf{\Pi}_h\underline{\bm{\mathit{\tau}}}\cdot\underline{\bm{\upsilon}}&=\int_K\underline{\bm{\mathit{\tau}}}\cdot\underline{\bm{\upsilon}},\quad\forall\underline{\bm{\mathit{\upsilon}}}\in\underline{\bm{\mathit{N}}}^k(K).
	\end{aligned}
\end{equation}
Here $\underline{\bm{\mathit{N}}}^k(K)$ is defined by
$\underline{\bm{\mathit{N}}}^k(K):=\underline{\bm{\mathit{p}}}^{k-1}(K)+\underline{\bm{\mathit{S}}}^{k}(K),$
where $\underline{\bm{\mathit{S}}}^{k}(K):=\{\underline{\bm{\upsilon}}\in \underline{\widetilde{{\bm{\mathit{p}}}}}^k(K):\underline{\bm{\upsilon}}\mathbf{x}=\mathbf{0}\}$
for $k\geq1$. The entries of  $\underline{\widetilde{{\bm{\mathit{p}}}}}^k(K)$ are the homogeneous polynomials of degree $k$.
The following commuting diagram property holds(see \cite{BF1991}):
\begin{equation}
	\label{hh}
	\textbf{div}\mathbf{\Pi}_h\underline{\bm{\tau}}=\bm{\mathit{P}}\textbf{div}\underline{\bm{\tau}},\quad\forall \underline{\bm{\tau}}\in\underline{\bm{\mathit{H}}}^t(\Omega)\cap\underline{\bm{\mathit{H}}}(\textbf{div};\Omega)
\end{equation}
where $\bm{\mathit{P}}: \bm{\mathit{L}}^2(\Omega)\to \bm{\mathit{W}}^h$ is the $\bm{\mathit{L}}^2$-orthogonal projector. In addition, it's well known that there exists $C>0$, independent of $h$, such that for each $\underline{\bm{\tau}}\in\underline{\bm{\mathit{H}}}^t(\Omega)\cap\underline{\bm{\mathit{H}}}(\textbf{div};\Omega)$ (see \cite{BBF2013}): 
\begin{equation}
	\label{ii}
	||\underline{\bm{\tau}}-\mathbf{\Pi}_h\underline{\bm{\tau}}||_{0,\Omega}\leq Ch^{\rm min\{t, k+2\}}||\underline{\bm{\tau}}||_{t,\Omega},\quad \forall t>1/2.
\end{equation}
For less regular tensorial fields, we have the following error estimate (see \cite[Theorem 3.16]{H2002})
\begin{equation}
	\label{jj}
	||\underline{\bm{\tau}}-\mathbf{\Pi}_h\underline{\bm{\tau}}||_{0,\Omega}\leq Ch^t(||\underline{\bm{\tau}}||_{t,\Omega}+||\textbf{div}\underline{\bm{\tau}}||_{0,\Omega}),\quad \forall t\in(0,1/2].
\end{equation}
And we denote $\underline{\bm{\mathit{P}}}: \underline{\bm{\mathit{Q}}}\to \underline{\bm{\mathit{A}}}^h$ the orthogonal projector with respect to the $\underline{\bm{\mathit{L}}}^2(\Omega)$-norm. Then for any $t>0$, we have
\begin{equation}
	\label{kk}
	||\underline{\bm{\rho}}-\underline{\bm{\mathit{P}}}\underline{\bm{\rho}}||_{0,\Omega}\leq Ch^{\rm min\{t,k+2\}}||\underline{\bm{\rho}}||_{t,\Omega} \quad \forall\underline{\bm{\rho}}\in\underline{\bm{\mathit{H}}}^t(\Omega)\cap\underline{\bm{\mathit{Q}}},
\end{equation}
\begin{equation}
	\label{ll}
	||\mathbf{u}-\bm{\mathit{P}}\mathbf{u}||_{0,\Omega}\leq Ch^{\rm min\{t,k+1\}}||\mathbf{u}||_{t,\Omega}\quad\forall \mathbf{u}\in{\bm{\mathit{H}}}^t(\Omega).
\end{equation}

For problem (\ref{aa}), the discrete version of the operator $\bm{\mathit{T}}$ is denoted as 
\begin{equation*}
	\begin{aligned}
		\bm{\mathit{T}}_h: \bm{\mathit{L}}^2(\Omega)&\to \bm{\mathit{W}}^h\\
		\bm{\mathit{f}}&\mapsto\mathbf{u}^{\mathit{f}}_h
	\end{aligned}	
\end{equation*}
where $\mathbf{u}^{\mathit{f}}_h$, also with $(\underline{\bm{\sigma}}^{\mathit{f}}_h,\underline{\bm{\rho}}^{\mathit{f}}_h)$ is the solution of the following discrete source problem:
\begin{subequations}
	\label{mm}
	\begin{align}		
		(\mathcal{A}\underline{\bm{\sigma}}^{\mathit{f}}_h,\underline{\bm{\upsilon}})_{\Omega}+(\mathbf{u}^{\mathit{f}}_h,\textbf{div}\underline{\bm{\upsilon}})_{\Omega}+(\underline{\bm{\rho}}^{\mathit{f}}_h,\underline{\bm{\upsilon}})_{\Omega}&=0 \label{mm_a}\\
		-(\rho_S^{-1}\textbf{div}\underline{\bm{\sigma}}^{\mathit{f}}_h,\bm{\omega})_{\Omega}&=(\bm{\mathit{f}},\bm{\omega})_{\Omega} \label{mm_b} \\
		(\underline{\bm{\sigma}}^{\mathit{f}}_h,\underline{\bm{\eta}})_{\Omega}&=0  
	\end{align}	
\end{subequations}
for all $(\underline{\bm{\upsilon}},\bm{\omega},\underline{\bm{\eta}})\in\underline{\bm{\mathit{V}}}^h\times\bm{\mathit{W}}^h\times\underline{\bm{\mathit{A}}}^h$.

By \cite[Theorem 6.3]{JJ2012}, on the special grids, (Note that \cite{JJ2012} gives the stability with a zero boundary condition, but combining the commuting diagram in \cite{JW2012}, we can also obtain the stability with the mixed boundary conditions in this paper) we have
\begin{equation}
	\label{pp}
	\sup\limits_{\substack{(\underline{\bm{\upsilon}},\bm{\omega},\underline{\bm{\eta}})\in\\\underline{\bm{\mathit{V}}}^h\times\bm{\mathit{W}}^h\times\underline{\bm{\mathit{A}}}^h}}\frac{A((\underline{\bm{\sigma}},\mathbf{u},\underline{\bm{\rho}}),(\underline{\bm{\upsilon}},\bm{\omega},\underline{\bm{\eta}}))}{||(\underline{\bm{\upsilon}},\bm{\omega},\underline{\bm{\eta}})||}\geq\beta||(\underline{\bm{\sigma}},\mathbf{u},\underline{\bm{\rho}})||,\quad\forall (\underline{\bm{\sigma}},\mathbf{u},\underline{\bm{\rho}})\in\underline{\bm{\mathit{V}}}^h\times\bm{\mathit{W}}^h\times\underline{\bm{\mathit{A}}}^h
\end{equation}
with $\beta>0$ independent of $h$ and $\lambda_S$, $||(\underline{\bm{\upsilon}},\bm{\omega},\underline{\bm{\eta}})||^2=||\underline{\bm{\upsilon}}||^2_{\underline{\bm{\mathit{H}}}(\textbf{div};\Omega)}+||\bm{\omega}||_{0,\Omega}^2+||\underline{\bm{\eta}}||_{0,\Omega}^2$. Indeed, 

(i) the discrete inf-sup condition: %which can be proved %by combining the commuting diagram in \cite{JW2012} and \cite[Theorem 7.1]{AFW2007}: %the cases of the pure Dirichlet boundary condition and Theorem 11.9 (the traction boundary condition) of \cite{AFW2006}:
$\exists\beta^{\star}>0$, independent of $h$, such that
\begin{equation}
	\label{nnn}
	\sup\limits_{\underline{\bm{\tau}}\in\underline{\bm{\mathit{V}}}^h}\frac{({\bm{\omega}}_h,\textbf{div}\underline{\bm{\tau}})_{\Omega}+(\underline{\bm{\tau}},\underline{\mathbf{s}}_h)_{\Omega}}{||\underline{\bm{\tau}}||_{\underline{\bm{\mathit{H}}}(\textbf{div};\Omega)}}\geq\beta^{\star}(||\underline{\mathbf{s}}_h||_{0,\Omega}+||\bm{\omega}_h||_{0,\Omega}),\quad\forall (\bm{\omega}_h,\underline{\mathbf{s}}_h)\in\bm{\mathit{W}}^h\times\underline{\bm{\mathit{A}}}^h	
\end{equation}

(ii) the coercivity in the kernel condition
$(\mathcal{A}\underline{\bm{\tau}},\underline{\bm{\tau}})\geq C||\underline{\bm{\tau}}||^2_{\underline{\bm{\mathit{H}}}(\textbf{div};\Omega)}$,
for all $\underline{\bm{\tau}}$ in the kernel given by $\bm{\mathit{K}}$:= $\{\underline{\bm{\tau}}\in\underline{\bm{\mathit{V}}}^h: (\bm{\omega}_h,\textbf{div}\underline{\bm{\tau}})_{\Omega}+(\underline{\bm{\eta}}_h,\underline{\bm{\tau}})_{\Omega}=0\hspace{0.5em} {\rm for\hspace{0.5em} all\hspace{0.5em}} \bm{\omega}_h\in\bm{\mathit{W}}^h\hspace{0.5em} {\rm and\hspace{0.5em} all\hspace{0.5em}} \underline{\bm{\eta}}_h\in\underline{\bm{\mathit{A}}}^h\}.$ Here the constant $C$ doesn't depend on $\lambda_S$. The lemma below explains the independence.
\begin{lemma}
	\label{lemma2.2}
	There exists a constant $C>0$, depending on $\mu_s$ and $\Omega$ (but not on $\lambda_S$), such that 
	$
	(\mathcal{A}\underline{\bm{\tau}},\underline{\bm{\tau}})_{\Omega}\geq C||\underline{\bm{\tau}}||^2_{\underline{\bm{\mathit{H}}}(\textbf{div};\Omega)}$, for all $\underline{\bm{\tau}}\in\bm{\mathit{K}}$.	
\end{lemma}
\begin{proof}
	For any $\underline{\bm{\tau}}\in\underline{\bm{\mathit{V}}}^h$, by (\ref{1220a}), we have $\mathcal{A}\underline{\bm{\tau}}:\underline{\bm{\tau}}=\frac{1}{2\mu_S}\underline{\bm{\tau}}^D:\underline{\bm{\tau}}^D+\frac{(tr\underline{\bm{\tau}})^2}{n(n\lambda_S+2\mu_S)}$. Then clearly 
	$
	(\mathcal{A}\underline{\bm{\tau}},\underline{\bm{\tau}})\geq\frac{1}{2\mu_S}||\underline{\bm{\tau}}^D||_{0,\Omega}^2 
	$. Let $\underline{\bm{\tau}}_0:=\underline{\bm{\tau}}-\frac{1}{n|\Omega|}(\int_{\Omega}tr\underline{\bm{\tau}})\bm{\mathit{I}}$. For $n=2$, it is proved in \cite[Proposition IV.3.1]{BF1991} that there exists a constant $C>0$, such that
	\begin{equation}
		\label{xb}
		||\underline{\bm{\tau}}_0||_{0,\Omega}^2\leq C(||\underline{\bm{\tau}}^D||_{0,\Omega}^2+||\textbf{div}\underline{\bm{\tau}}||_{0,\Omega}^2)	 \quad \forall \underline{\bm{\tau}}\in\underline{\bm{\mathit{V}}}^h
	\end{equation}
	The same proof also runs for $n=3$. On the other hand, in terms of \cite[Lemma 2.2]{Pro}, we directly know that for a bounded and simply connected domain in $\mathbb{R}^2$ with Lipschitz-continuous boundary, there exists $C>0$, only depending on $\Omega$ and $\Gamma_1$, such that
	\begin{equation}
		\label{xc}
		||\underline{\bm{\tau}}||_{\underline{\bm{\mathit{H}}}(\textbf{div};\Omega)}^2\leq C||\underline{\bm{\tau}}_0||_{\underline{\bm{\mathit{H}}}(\textbf{div};\Omega)}^2	 \quad \forall \underline{\bm{\tau}}\in\underline{\bm{\mathit{V}}}^h\subset\underline{\bm{\mathit{Y}}}		
	\end{equation}
	The same proof also runs for a bounded and connected Lipschitz-domain in $\mathbb{R}^2$ and $\mathbb{R}^3$.
	Then by (\ref{xb}), (\ref{xc}) and the fact that $\textbf{div}\underline{\bm{\tau}}_0=\textbf{div}\underline{\bm{\tau}}$, we obtain
	$||\underline{\bm{\tau}}||_{\underline{\bm{\mathit{H}}}(\textbf{div};\Omega)}^2\leq C(||\underline{\bm{\tau}}^D||_{0,\Omega}^2+||\textbf{div}\underline{\bm{\tau}}||_{0,\Omega}^2)$.
	Thus, for all $\underline{\bm{\tau}}\in\bm{\mathit{K}}$(i.e. $\textbf{div}\underline{\bm{\tau}}=0$, $\underline{\bm{\tau}}$ is symmetric), we have
	$(\mathcal{A}\underline{\bm{\tau}},\underline{\bm{\tau}})_{\Omega}\geq\frac{1}{2\mu_S}||\underline{\bm{\tau}}^D||_{0,\Omega}^2\geq C||\underline{\bm{\tau}}||_{\underline{\bm{\mathit{H}}}(\textbf{div};\Omega)}^2,$
	where $C$ depends on $\mu_S$ and $\Omega$, but not on $\lambda_S$.
\end{proof} Therefore, (\ref{pp}) holds.
As the continuous case, $(\lambda_h,\underline{\bm{\sigma}}_h,\mathbf{u}_h,\underline{\bm{\rho}}_h)\in \mathbb{R}\times\underline{\bm{\mathit{V}}}^h\times\bm{\mathit{W}}^h\times\underline{\bm{\mathit{A}}}^h$ solves (\ref{aa}) if and only if $(\mu_h,\mathbf{u}_h)$ with $\mu_h=\frac{1}{\lambda_h}$ is an eigenpair of $\bm{\mathit{T}}_h$ and  $(\underline{\bm{\sigma}}_h^{\mathit{f}},\underline{\bm{\rho}}_h^{\mathit{f}})=\frac{1}{\lambda_h}(\underline{\bm{\sigma}}_h, \underline{\bm{\rho}}_h)$. We define the operators $\bm{\mathit{T}}_M$ and $\bm{\mathit{T}}_{M,h}$ corresponding to $\bm{\mathit{T}}$ and $\bm{\mathit{T}}_h$ as follows:
\begin{equation*}
	\begin{aligned}
		\bm{\mathit{T}}_M: \bm{\mathit{L}}^2(\Omega)&\rightarrow \underline{\bm{\mathit{L}}}^2(\Omega)\times\bm{\mathit{L}}^2(\Omega)\times\underline{\bm{\mathit{Q}}}\\
		\bm{\mathit{f}}&\mapsto \begin{Bmatrix}
			&\underline{\bm{\sigma}}^{\mathit{f}}
			&\mathbf{u}^{\mathit{f}} 
			&\underline{\bm{\rho}}^{\mathit{f}}
		\end{Bmatrix}
	\end{aligned}
\end{equation*}
Obviously, under proper regularity assumptions, $\bm{\mathit{T}}_M$ is compact. $\bm{\mathit{T}}_{M,h}$ is defined by $
		\bm{\mathit{T}}_{M,h}: \bm{\mathit{L}}^2(\Omega)\rightarrow \underline{\bm{\mathit{V}}}^h\times\bm{\mathit{W}}^h\times\underline{\bm{\mathit{A}}}^h(i.e. 
		\bm{\mathit{T}}_{M,h}\bm{\mathit{f}}= 
			(\underline{\bm{\sigma}}^{\mathit{f}}_h,
			\mathbf{u}^{\mathit{f}}_h, 
			\underline{\bm{\rho}}^{\mathit{f}}_h)\in\underline{\bm{\mathit{V}}}^h\times\bm{\mathit{W}}^h\times\underline{\bm{\mathit{A}}}^h
		)$.

\subsection{Discrete $H^1$-stability of numerical displacement} \label{sec: Discrete $H^1$-stability of numerical displacement} We recalled the definition of $||\mathbf{u}_h||_{1,h}$ in (\ref{0112b}), which is
different from that in \cite{GQ2018} since we discuss the mixed boundary conditions here.
We will show in Theorem \ref{1220gg} that $||\mathbf{u}_h^{\mathit{f}}||_{1,h}$ can be bounded by $||\mathcal{A}\underline{\bm{\sigma}}^{\mathit{f}}_h||_{0,\Omega}$. Before going further, we provide a lemma for the BDM element, which can be used in the later proof. A similar lemma is \cite[Lemma 3.1]{GQ2018} and hence we omit the proofs.
\begin{lemma}
	\label{1220mm}
	For each element $K\in\mathcal{T}_h$, given $\underline{\bm{\sigma}}\in\underline{\bm{\mathit{L}}}^2(K)$, $\mathbf{q}_i\in\bm{\mathit{L}}^2(F_i)$, where $\{F_i\}\in\partial K$, there exists a unique $\underline{\bm{\upsilon}}\in \underline{\bm{\mathit{V}}}^k(K)$ such that
	\begin{subequations}
		\label{1220r}
		\begin{align}
		\int_{F_i}\underline{\bm{\upsilon}}\mathbf{n}_{F_i}\cdot\bm{\mu}ds&=\int_{F_i}\mathbf{q}_i\cdot\bm{\mu}ds\quad\quad\quad \forall \bm{\mu}\in{\bm{\mathit{p}}}^{k+1}(F_i), F_i\in\partial K\\
		\int_K\underline{\bm{\upsilon}}\cdot\underline{\bm{\omega}}dx&=\int_K\underline{\bm{\sigma}}\cdot\underline{\bm{\omega}}dx\quad\quad\quad \forall \underline{\bm{\omega}}\in\underline{\bm{\mathit{N}}}^{k}(K)
		\end{align}
	\end{subequations}
	Here $\underline{\bm{\mathit{N}}}^{k}(K)$ is defined under (\ref{BDM}). More importantly,
	$||\underline{\bm{\upsilon}}||^2_{0,K}\leq C(||\underline{\bm{\sigma}}||^2_{0,K}+h\sum_{F_i\in\partial K}$ $||\mathbf{q}_i||^2_{L^2(F_i)}),$
	where C is independent of $h$ and $K$.
\end{lemma}
Next, we prove that $||\mathbf{u}_h^{\mathit{f}}||_{1,h}$  can be controlled by $||\mathcal{A}\underline{\bm{\sigma}}^{\mathit{f}}_h||_{0,\Omega}$, which plays a key role in the proof of Theorem \ref{1221a}. Actually, a very simple example similar to this type of control relationship can be found in  \cite[Theorem 3.2]{GQ2018} for the Poisson problem. However, the control relationship is not obvious in our case. 
\begin{theorem}
	\label{1220gg}
	For (\ref{mm}), there exists a constant $C$ independent of $\mathcal{A}^{-1}$ and $h$ such that
	$
		||\mathbf{u}_h^{\mathit{f}}||_{1,h}^2\leq C||\mathcal{A}\underline{\bm{\sigma}}^{\mathit{f}}_h||_{0,\Omega}^2.
	$
\end{theorem}

\begin{proof}
	By using integration by parts for equation (\ref{mm_a}), we have
	\begin{equation}
		\label{1220ii}
		\begin{aligned}
			0&=(\mathcal{A}\underline{\bm{\sigma}}_h^{\mathit{f}},\underline{\bm{\upsilon}})_{\Omega}+(\mathbf{u}_h^{\mathit{f}},\textbf{div}\underline{\bm{\upsilon}})_{\Omega}+(\underline{\bm{\rho}}_h^{\mathit{f}},\underline{\bm{\upsilon}})_{\Omega}\\
			&=(\mathcal{A}\underline{\bm{\sigma}}_h^{\mathit{f}},\underline{\bm{\upsilon}})_{\Omega}-\sum_{K\in\mathcal{T}_h}(\nabla\mathbf{u}_h^{\mathit{f}},\underline{\bm{\upsilon}})_K+\sum_{F\in\varepsilon^{I}\cup\varepsilon^{D}}\left \langle[\![\mathbf{u}_h^{\mathit{f}}]\!],\underline{\bm{\upsilon}}\mathbf{n}\right \rangle_F+(\underline{\bm{\rho}}_h^{\mathit{f}},\underline{\bm{\upsilon}})_{\Omega}
		\end{aligned}
	\end{equation}
	On each element $K$, let $\underline{\bm{\upsilon}}_1\in\underline{\bm{\mathit{V}}}^h$ (Notice that $\underline{\bm{\upsilon}}_1\mathbf{n}|_{\Gamma_1}=\mathbf{0}$) be the projection such that
	\begin{equation}
		\label{1220jj}
		\int_{F_i}\underline{\bm{\upsilon}}_1\mathbf{n}_{F_i}\cdot\bm{\mu}ds=\int_{F_i}(-\frac{1}{h_{F_i}}[\![\mathbf{u}_h^{\mathit{f}}]\!])\cdot\bm{\mu}ds\quad\quad \forall \bm{\mu}\in{\bm{\mathit{p}}}^{k+1}(F_i), F_i\in\partial K\cap(\varepsilon^{I}\cup\varepsilon^{D})
	\end{equation}
	\begin{equation}
		\label{1220kk}
		\int_K\underline{\bm{\upsilon}}_1\cdot\underline{\bm{\omega}}dx=\int_K\underline{\bm{\epsilon}}(\mathbf{u}_h^{\mathit{f}})\cdot\underline{\bm{\omega}}dx\quad\quad\forall \underline{\bm{\omega}}\in\underline{\bm{\mathit{N}}}^{k}(K)\supseteq\underline{\bm{\mathit{p}}}^{k-1}(K)
	\end{equation}
	Setting $\bm{\mu}=[\![\mathbf{u}_h^{\mathit{f}}]\!]\in\bm{\mathit{p}}^{k+1}(F_i)$, $\underline{\bm{\omega}}=\nabla\mathbf{u}_h^{\mathit{f}}\in\underline{\bm{\mathit{p}}}^{k-1}(K)$, and substituting the above two equalities into (\ref{1220ii}), we get
	\begin{equation}
		\label{1220ll}
		(\mathcal{A}\underline{\bm{\sigma}}_h^{\mathit{f}},\underline{\bm{\upsilon}}_1)_{\Omega}=\sum_{K\in\mathcal{T}_h}||\underline{\bm{\epsilon}}(\mathbf{u}_h^{\mathit{f}})||^2_{0,K}+\sum_{F\in\varepsilon^{I}\cup\varepsilon^{D}}\frac{1}{h_F}||[\![\mathbf{u}_h^{\mathit{f}}]\!]||^2_{L^2(F)}-(\underline{\bm{\rho}}_h^{\mathit{f}},\underline{\bm{\upsilon}}_1)_{\Omega}	
	\end{equation}
	Then in terms of Lemma \ref{1220mm}, we have
	\begin{equation}
		\label{nn}
		||\underline{\bm{\upsilon}}_1||^2_{0,K}\leq C(||\underline{\bm{\epsilon}}(\mathbf{u}_h^{\mathit{f}})||^2_{0,K}+\sum_{F\in\partial K\cap(\varepsilon^{I}\cup\varepsilon^{D})}\frac{1}{h_F}||[\![\mathbf{u}_h^{\mathit{f}}]\!]||^2_{L^2(F)}).	
	\end{equation}
	Recall that the linear elasticity problem can be associated with the Stokes problem with the Stokes pair of spaces $\bm{\mathit{S}}^h\times R^h\subseteq\bm{\mathit{H}}^1(\Omega)\times\mathit{L}^2(\Omega)$(see \cite{JJ2012}):
	$R^h=\{r\in\mathit{L}^2(\Omega):r|_K\in\mathit{p}^{k+1}(K),\hspace{0.5em} {\rm for\hspace{0.5em} all\hspace{0.5em} elements}\hspace{0.5em} K\},\bm{\mathit{S}}^h= \{\bm{\mathit{S}}\in\bm{\mathit{H}}^1(\Omega):\bm{\mathit{S}}|_K\in\bm{\mathit{p}}^{k+2}(K),\hspace{0.5em} {\rm for\hspace{0.5em} all\hspace{0.5em} elements}\hspace{0.5em} K\}.$
	In the following we will directly use \cite[(5.1); (5.2)]{JJ2012} and notations in \cite{JJ2012}. We recalled \cite[(5.1)]{JJ2012} which is shown in Figure \ref{diagram commutes} (Here {\rm skw}($\cdot$) denotes the operator mapping matrices to their skew-symmetric parts. $\mathit{P}$ denotes the
	$L^2$-orthogonal projection onto $R^h$. $\underline{\bm{\mathit{P}}}$ denotes the $\underline{\bm{\mathit{L}}}^2$-orthogonal projection onto $\underline{\bm{\mathit{A}}}^h$. $\underline{\bm{\mathit{X}}}_2$ is an isomorphic map between $R^h$ and $\underline{\bm{\mathit{A}}}^h$). Notice that the diagram commutes generally hold for the space $\bm{\mathit{S}}^h$ with mixed boundary conditions (i.e. our case: $\underline{\hat{\bm{\mathit{V}}}}^h=\underline{\bm{\mathit{V}}}^h$). And recall \cite[(5.2)]{JJ2012}: for all $\bm{\mathit{s}}^h\in\bm{\mathit{S}}^h$.
	\begin{equation}
		\label{0110g}
		\frac{1}{2}\underline{\bm{\mathit{X}}}_2\underline{\bm{\mathit{P}}}{\rm div}\bm{\mathit{s}}^h=\underline{\bm{\mathit{P}}}{\rm skw}{\rm curl}\bm{\mathit{s}}^h,
	\end{equation}
	\begin{figure}[tbph] 
		\centering 
		\includegraphics[height=2.5cm,width=8cm]{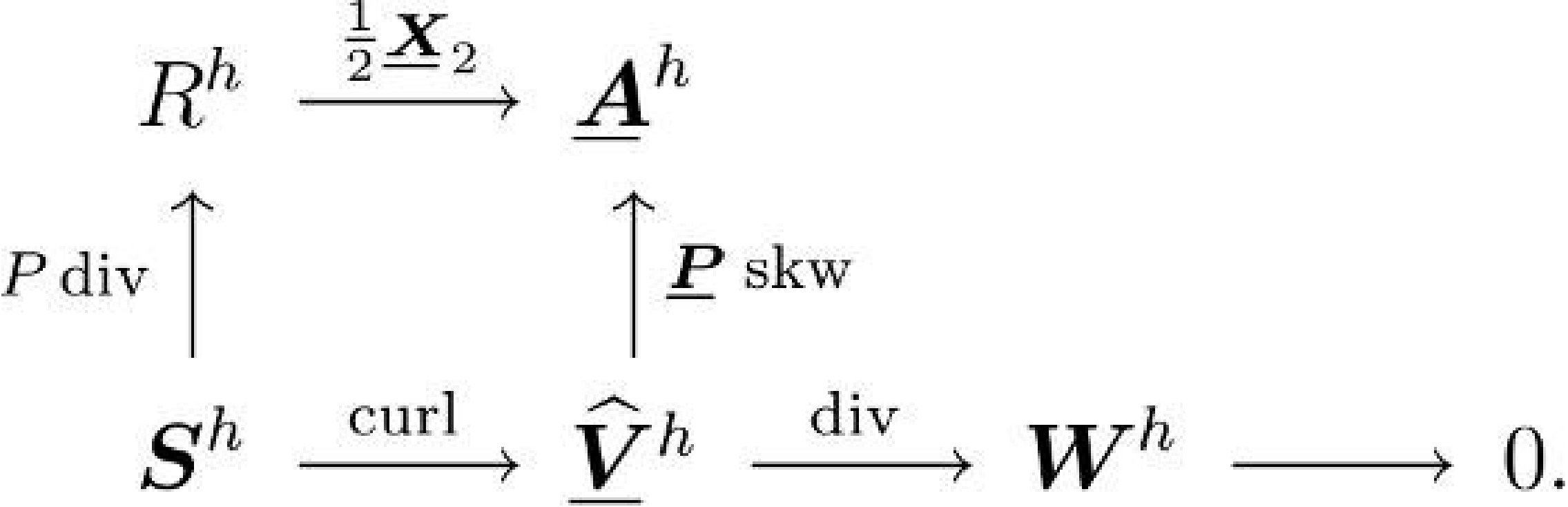} 
		\caption{ diagram commutes}
		\label{diagram commutes}
	\end{figure}
	By \cite{AQ1992}, on the special grids, %the associated Stokes problem is stable. By virtue of the stability of continuous formulation, we know that
	for $r^h=2\underline{\bm{\mathit{X}}}_2^{-1}\underline{\bm{\mathit{P}}}(-{\rm skw}\underline{\bm{\upsilon}}_1)$, we can find  $\bm{\mathit{s}}^h\in\bm{\mathit{S}}^h$ satisfying
	%there exists $\bm{\mathit{s}}\in\bm{\mathit{H}}^1(\Omega)$ such that ${\rm div}\bm{\mathit{s}}=r^h$, and $||\bm{\mathit{s}}||_{\mathit{H}^1}\leq C||r^h||_{\mathit{L}^2}$. In terms of the stability of the finite element spaces $\bm{\mathit{S}}^h\times R^h\subseteq\bm{\mathit{H}}^1(\Omega)\times\mathit{L}^2(\Omega)$, also combining the converse of Fortin's criterion, we know there exists a bounded linear projector: $\bm{\mathit{\Pi}}_h^1: \bm{\mathit{H}}^1(\Omega)\mapsto\bm{\mathit{S}}_h$ such that 
	%$$\int_{\Omega}{\rm div}(\bm{\mathit{s}}-\bm{\mathit{\Pi}}_h^1\bm{\mathit{s}})\cdot\mu_h=0,\quad\forall \mu_h\in R^h$$
	%That's, $r^h={\rm div}\bm{\mathit{s}}=\mathit{P}{\rm div}\bm{\mathit{\Pi}}_h^1\bm{\mathit{s}}$. We denote $\bm{\mathit{s}}^h=\bm{\mathit{\Pi}}_h^1\bm{\mathit{s}}\in\bm{\mathit{S}}^h$. Then we have
	%	\begin{equation}
	%		\label{1220oo}
	%		\mathit{P}{\rm div}\bm{\mathit{s}}^h=r^h, \quad ||\bm{\mathit{s}}^h||_{\mathit{H}^1}\leq C||r^h||_{\mathit{L}^2}.	
	%	\end{equation}
	%	We choose $r^h=2\underline{\bm{\mathit{X}}}_2^{-1}\underline{\bm{\mathit{P}}}(-{\rm skw}\underline{\bm{\upsilon}}_1)$, by (\ref{1220oo}), we can find $\bm{\mathit{s}}^h\in\bm{\mathit{S}}^h$ such that
	\begin{equation}
		\label{1220pp}
		\mathit{P}{\rm div}\bm{\mathit{s}}^h=2\underline{\bm{\mathit{X}}}_2^{-1}\underline{\bm{\mathit{P}}}(-{\rm skw}\underline{\bm{\upsilon}}_1), \quad ||\bm{\mathit{s}}^h||_{\mathit{H}^1}\leq C||{\rm skw}\underline{\bm{\upsilon}}_1||_{\mathit{L}^2}.	
	\end{equation}
	Setting $\underline{\bm{\upsilon}}_2={\rm curl}\bm{\mathit{s}}^h$, and by (\ref{0110g}), this implies that
	\begin{equation}
		\label{1220rr}
		\underline{\bm{\mathit{P}}}{\rm skw}\underline{\bm{\upsilon}}_2=\underline{\bm{\mathit{P}}}(-{\rm skw}\underline{\bm{\upsilon}}_1), \quad ||\underline{\bm{\upsilon}}_2||_{\mathit{L}^2}\leq C||\underline{\bm{\upsilon}}_1||_{\mathit{L}^2}			
	\end{equation}
	Above is the two-dimensional case. For the three-dimensional case, in terms of \cite{Z2005}, we can obtain a similar result to (\ref{1220pp}). %(We note that although the result in \cite{Z2005} is stated only for spaces with a zero boundary condition, the analysis there holds more generally even for the space $\bm{\mathit{S}}^h$ without the boundary condition.) 
	By analogous analysis, we can obtain the same results as the two dimension. So we omit the details of the three-dimensional case here.
	Finally, setting $\underline{\bm{\upsilon}}=\underline{\bm{\upsilon}}_1+\underline{\bm{\upsilon}}_2$. By using (\ref{1220rr}), (\ref{1220ii}), (\ref{1220ll}) and (\ref{nn}), we get
	\begin{equation}
		\label{1220aaa}
		\sum_{K\in\mathcal{T}_h}||\underline{\bm{\epsilon}}(\mathbf{u}_h^{\mathit{f}})||^2_{0,K}+\sum_{F\in\varepsilon^{I}\cup\varepsilon^{D}}h^{-1}||[\![\mathbf{u}_h^{\mathit{f}}]\!]||^2_{L^2(F)}\leq C||\mathcal{A}\underline{\bm{\sigma}}_h^{\mathit{f}}||_{0,\Omega}^2,
	\end{equation}
	where $C$ doesn't depend on $h$ and $\lambda_S$. Then in terms of the well-known $H^1$ interpolation, we know that there exists $\tilde{\mathbf{u}}_h^{\mathit{f}}\in \{\mathbf{u}\in\bm{\mathit{H}}^1(\Omega): \mathbf{u}|_{\Gamma_0}=\mathbf{0},|\Gamma_0|>0\}\cap \bm{\mathit{p}}^k(\mathcal{T}_h)$ such that
	\begin{equation}
		\label{1220ww}
		\sum_{K\in\mathcal{T}_h}||\nabla\mathbf{u}_h^{\mathit{f}}-\nabla\tilde{\mathbf{u}}_h^{\mathit{f}}||^2_{0,K}\le C\sum_{F\in\varepsilon^{I}\cup\varepsilon^{D}}h_F^{-1}||[\![\mathbf{u}_h^{\mathit{f}}]\!]||^2_{L^2(F)}
	\end{equation}
	Then, combining the triangle inequality, we immediately obtain
	\begin{equation}
		\label{1220yy}			
		||\underline{\bm{\epsilon}}(\tilde{\mathbf{u}}_h^{\mathit{f}})||_{0,\Omega}^2\leq C(\sum_{K\in\mathcal{T}_h}||\underline{\bm{\epsilon}}(\mathbf{u}_h^{\mathit{f}})||^2_{0,K}+\sum_{F\in\varepsilon^{I}\cup\varepsilon^{D}}h_F^{-1}||[\![\mathbf{u}_h^{\mathit{f}}]\!]||^2_{L^2(F)})
	\end{equation}
	By (\ref{1220ww}), (\ref{1220yy}), the triangle inequality and Korn's inequality (see \cite{PG2013}), we obtain
	\begin{equation}
		\label{1220zz}	
		\sum_{K\in\mathcal{T}_h}||\nabla\mathbf{u}_h^{\mathit{f}}||^2_{0,K}
		\le C(\sum_{K\in\mathcal{T}_h}||\underline{\bm{\epsilon}}(\mathbf{u}_h^{\mathit{f}})||^2_{0,K}+\sum_{F\in\varepsilon^{I}\cup\varepsilon^{D}}h_F^{-1}||[\![\mathbf{u}_h^{\mathit{f}}]\!]||^2_{L^2(F)})			
	\end{equation}
	Therefore, in terms of (\ref{1220aaa}) and (\ref{1220zz}), we get the desired assertion.	
\end{proof}	

\begin{remark}
	In fact, in terms of the definition of $||\mathbf{u}_h||_{1,h}$, we have the following discrete Sobolev embedding inequalities
	\begin{equation}
		\label{0109k}
		\left\{
		\begin{aligned}
			||\mathbf{u}_h||_{\bm{\mathit{L}}^p}& \leq C||\mathbf{u}_h||_{1,h} , \hspace{0.5em}{\rm for}\hspace{0.5em} 1\leq p<\infty,\hspace{0.5em} {\rm in\hspace{0.5em} two\hspace{0.5em}dimensional\hspace{0.5em} space,}\\
			||\mathbf{u}_h||_{\bm{\mathit{L}}^p}& \leq C||\mathbf{u}_h||_{1,h} , \hspace{0.5em}{\rm for}\hspace{0.5em} 1\leq p\leq 6,\hspace{0.5em} {\rm in\hspace{0.5em} three\hspace{0.5em}dimensional\hspace{0.5em} space,}
		\end{aligned}
		\right.
	\end{equation}
	where $C$ depends upon the domain $\Omega$, the polynomial degree $k$ and $p$ only.
	Indeed, for $\mathbf{u}_h\in\underline{\bm{\mathit{V}}}^h$, by utilizing the general $H^1$-interpolation, there exists $\tilde{\mathbf{u}}_h\in \{\mathbf{u}\in\bm{\mathit{H}}^1(\Omega): \mathbf{u}|_{\Gamma_0}=\mathbf{0},d\Gamma-{\rm meas}\Gamma_0>0\}\cap \bm{\mathit{p}}^k(\mathcal{T}_h)$ such that$
		\label{0109m}
		\sum_{K\in\mathcal{T}_h}||\nabla\mathbf{u}_h-\nabla\tilde{\mathbf{u}}_h||^2_{0,K}\le C\sum_{F\in\varepsilon^{I}\cup\varepsilon^{D}}\frac{1}{h_F}||[\![\mathbf{u}_h]\!]||^2_{L^2(F)}$.
	Going through the same steps as (\ref{1220ww})-(\ref{1220zz}), we have
	$||\nabla\tilde{\mathbf{u}}_h||_{0,\Omega}^2$ $\leq C(\sum_{K\in\mathcal{T}_h}||\underline{\bm{\epsilon}}(\mathbf{u}_h)||^2_{0,K}$ $+$ $\sum_{F\in\varepsilon^{I}\cup\varepsilon^{D}}$ $h_F^{-1}||[\![\mathbf{u}_h]\!]||^2_{L^2(F)})$.
	Then by the triangle inequality, conventional Sobolev inequalities and the two inequalities above
	\begin{equation*}
		\label{0109l}
		\begin{aligned}	
			||\mathbf{u}_h||_{\bm{\mathit{L}}^p}^2&\leq C(||\mathbf{u}_h-\tilde{\mathbf{u}}_h||_{\bm{\mathit{L}}^p}^2+||\tilde{\mathbf{u}}_h||_{\bm{\mathit{L}}^p}^2)\leq C(\sum_{K\in\mathcal{T}_h}||\nabla(\mathbf{u}_h-\tilde{\mathbf{u}}_h)||_{0,K}^2)+C||\nabla\tilde{\mathbf{u}}_h||_{0,\Omega}^2\\
			&\leq C(\sum_{K\in\mathcal{T}_h}||\underline{\bm{\epsilon}}(\mathbf{u}_h)||^2_{0,K}+\sum_{F\in\varepsilon^{I}\cup\varepsilon^{D}}h_F^{-1}||[\![\mathbf{u}_h]\!]||^2_{L^2(F)})=C ||\mathbf{u}_h||_{1,h}^2
		\end{aligned}	
	\end{equation*}
	Therefore, (\ref{0109k}) is proved. Now in terms of Theorem \ref{1220gg}, we immediately have the following discrete Sobolev embedding inequality for $(\underline{\bm{\sigma}}^{\mathit{f}}_h,\mathbf{u}_h^{\mathit{f}})$
	\begin{equation}
		\label{0109o}
		\left\{
		\begin{aligned}
			||\mathbf{u}_h^{\mathit{f}}||_{\bm{\mathit{L}}^p}& \leq C||\mathcal{A}\underline{\bm{\sigma}}^{\mathit{f}}_h||_{0,\Omega} , \hspace{0.5em}{\rm for}\hspace{0.5em} 1\leq p<\infty,\hspace{0.5em} {\rm in\hspace{0.5em} two\hspace{0.5em}dimensional\hspace{0.5em} space,}\\
			||\mathbf{u}_h^{\mathit{f}}||_{\bm{\mathit{L}}^p}& \leq C||\mathcal{A}\underline{\bm{\sigma}}^{\mathit{f}}_h||_{0,\Omega} , \hspace{0.5em}{\rm for}\hspace{0.5em} 1\leq p\leq 6,\hspace{0.5em} {\rm in\hspace{0.5em} three\hspace{0.5em}dimensional\hspace{0.5em} space,}
		\end{aligned}
		\right.
	\end{equation}
	where $C$ depends upon the domain $\Omega$, the polynomial degree $k$ and $p$.
\end{remark}

\subsection{Convergence of \texorpdfstring{{\boldmath$\{\bm{\mathit{T}}_{M,h}\}$}}{T} to \texorpdfstring{{\boldmath$\bm{\mathit{T}}_M$}}{T} in discrete $H^1$-norm}
\label{sec: Convergence of TM operators}
\begin{theorem}
	\label{1221a}
	Under Assumption \ref{0128b}, for all $s\in(0,s_0]$, there exists a constant $C>0$ independent of $h$ and $\lambda_S$ such that
	\begin{equation}
		\label{1221l}
		||(\bm{\mathit{T}}_{M,h}-\bm{\mathit{T}}_M)\bm{\mathit{f}}||_{1,h}\leq Ch^s||\bm{\mathit{f}}||_{0,\Omega}\quad \forall \bm{\mathit{f}}\in\bm{\mathit{L}}^2(\Omega),
	\end{equation}
	where  $||(\bm{\mathit{T}}_{M,h}-\bm{\mathit{T}}_M)\bm{\mathit{f}}||_{1,h}^2=||\mathcal{A}^{\frac{1}{2}}(\underline{\bm{\sigma}}^{\mathit{f}}-\underline{\bm{\sigma}}^{\mathit{f}}_h)||_{0,\Omega}^2+||\bm{\mathit{P}}\mathbf{u}^{\mathit{f}}-\mathbf{u}^{\mathit{f}}_h||_{1,h}^2+||\underline{\bm{\rho}}^{\mathit{f}}-\underline{\bm{\rho}}^{\mathit{f}}_h||_{0,\Omega}^2$ and $\bm{\mathit{P}}$ is the $\bm{\mathit{L}}^2$-orthogonal projection onto $\bm{\mathit{W}}^h$, which immediately implies $||(\bm{\mathit{T}}_h-\bm{\mathit{T}})\bm{\mathit{f}}||_{0,\Omega}\leq Ch^s||\bm{\mathit{f}}||_{0,\Omega}$. 
	%More generally, we have
	%\begin{equation}
	%	\label{1228d}
	% ||\underline{\bm{\sigma}}^{\mathit{f}}-\underline{\bm{\sigma}}^{\mathit{f}}_h||_{0,\Omega}^2+||\bm{\mathit{P}}\mathbf{u}^{\mathit{f}}-\mathbf{u}^{\mathit{f}}_h||_{1,h}^2+||\underline{\bm{\rho}}^{\mathit{f}}-\underline{\bm{\rho}}^{\mathit{f}}_h||_{0,\Omega}^2\leq Ch^{\rm 2min\{t,k+2\}}.	
	% \end{equation}
\end{theorem}
\begin{proof}
	In terms of (\ref{c}) and (\ref{hh}), we can obtain that 
	\begin{subequations}
		\label{1221b}
		\begin{align}		
			(\mathcal{A}(\mathbf{\Pi}_h\underline{\bm{\sigma}}^{\mathit{f}}),\underline{\bm{\tau}}_h)_{\Omega}+(\bm{\mathit{P}}\mathbf{u}^{\mathit{f}},\textbf{div}\underline{\bm{\tau}}_h)_{\Omega}+(\underline{\bm{\mathit{P}}}\underline{\bm{\rho}}^{\mathit{f}},\underline{\bm{\tau}}_h)_{\Omega}&=(\mathcal{A}(\mathbf{\Pi}_h\underline{\bm{\sigma}}^{\mathit{f}}-\underline{\bm{\sigma}}^{\mathit{f}}),\underline{\bm{\tau}}_h)_{\Omega} \\
			-(\rho_S^{-1}\textbf{div}(\mathbf{\Pi}_h\underline{\bm{\sigma}}^{\mathit{f}}),\bm{\omega}_h)_{\Omega}=-(\rho_S^{-1}\bm{\mathit{P}}\textbf{div}\underline{\bm{\sigma}}^{\mathit{f}},\bm{\omega}_h)_{\Omega}&=(\bm{\mathit{f}},\bm{\omega}_h)_{\Omega} \\
			(\mathbf{\Pi}_h\underline{\bm{\sigma}}^{\mathit{f}},\underline{\bm{\eta}}_h)_{\Omega}&=(\mathbf{\Pi}_h\underline{\bm{\sigma}}^{\mathit{f}}-\underline{\bm{\sigma}}^{\mathit{f}},\underline{\bm{\eta}}_h)_{\Omega}  
		\end{align}	
	\end{subequations}
	for all $(\underline{\bm{\tau}}_h,\bm{\omega}_h,\underline{\bm{\eta}}_h)\in\underline{\bm{\mathit{V}}}^h\times\bm{\mathit{W}}^h\times\underline{\bm{\mathit{A}}}^h$.
	Recalling (\ref{mm}), we have that 
	\begin{equation}
		\label{1221c}
		A((\underline{\bm{\sigma}}^{\mathit{f}}_h,\mathbf{u}^{\mathit{f}}_h,\underline{\bm{\rho}}^{\mathit{f}}_h),(\underline{\bm{\tau}}_h,\bm{\omega}_h,\underline{\bm{\eta}}_h))= B(\bm{\mathit{f}},\bm{\omega}_h),\quad\forall(\underline{\bm{\tau}}_h,\bm{\omega}_h,\underline{\bm{\eta}}_h)\in\underline{\bm{\mathit{V}}}^h\times\bm{\mathit{W}}^h\times\underline{\bm{\mathit{A}}}^h.
	\end{equation}
	We define $\mathbf{e}_{\underline{\bm{\sigma}}^{\mathit{f}}}=\underline{\bm{\sigma}}^{\mathit{f}}_h-\mathbf{\Pi}_h\underline{\bm{\sigma}}^{\mathit{f}}$, $\mathbf{e}_{\mathbf{u}^{\mathit{f}}}=\mathbf{u}^{\mathit{f}}_h-\bm{\mathit{P}}\mathbf{u}^{\mathit{f}}$, $\mathbf{e}_{\underline{\bm{\rho}}^{\mathit{f}}}=\underline{\bm{\rho}}^{\mathit{f}}_h-\underline{\bm{\mathit{P}}}\underline{\bm{\rho}}^{\mathit{f}}$. Then by subtracting (\ref{1221b}) from (\ref{1221c}), we obtain that
	\begin{subequations}
		\label{1221d}
		\begin{align}		
			(\mathcal{A}\mathbf{e}_{\underline{\bm{\sigma}}^{\mathit{f}}},\underline{\bm{\tau}}_h)_{\Omega}+(\mathbf{e}_{\mathbf{u}^{\mathit{f}}},\textbf{div}\underline{\bm{\tau}}_h)_{\Omega}+(\mathbf{e}_{\underline{\bm{\rho}}^{\mathit{f}}},\underline{\bm{\tau}}_h)_{\Omega}&=(\mathcal{A}(\underline{\bm{\sigma}}^{\mathit{f}}-\mathbf{\Pi}_h\underline{\bm{\sigma}}^{\mathit{f}}),\underline{\bm{\tau}}_h)_{\Omega}\label{1221d_a}\\
			(\textbf{div}\mathbf{e}_{\underline{\bm{\sigma}}^{\mathit{f}}},\bm{\omega}_h)_{\Omega}&=0 \label{1221d_b}\\
			(\mathbf{e}_{\underline{\bm{\sigma}}^{\mathit{f}}},\underline{\bm{\eta}}_h)_{\Omega}&=(\underline{\bm{\sigma}}^{\mathit{f}}-\mathbf{\Pi}_h\underline{\bm{\sigma}}^{\mathit{f}},\underline{\bm{\eta}}_h)_{\Omega}\label{1221d_c}  
		\end{align}	
	\end{subequations}
	for all $(\underline{\bm{\tau}}_h,\bm{\omega}_h,\underline{\bm{\eta}}_h)\in\underline{\bm{\mathit{V}}}^h\times\bm{\mathit{W}}^h\times\underline{\bm{\mathit{A}}}^h$. 
	We choose $\underline{\bm{\tau}}_h=\mathbf{e}_{\underline{\bm{\sigma}}^{\mathit{f}}},\bm{\omega}_h=\mathbf{e}_{\mathbf{u}^{\mathit{f}}},\underline{\bm{\eta}}_h=\mathbf{e}_{\underline{\bm{\rho}}^{\mathit{f}}}$, then
	\begin{equation}
		\label{1221e}
		(\mathcal{A}\mathbf{e}_{\underline{\bm{\sigma}}^{\mathit{f}}},\mathbf{e}_{\underline{\bm{\sigma}}^{\mathit{f}}})_{\Omega}=(\mathcal{A}(\underline{\bm{\sigma}}^{\mathit{f}}-\mathbf{\Pi}_h\underline{\bm{\sigma}}^{\mathit{f}}),\mathbf{e}_{\underline{\bm{\sigma}}^{\mathit{f}}})_{\Omega}-(\underline{\bm{\sigma}}^{\mathit{f}}-\mathbf{\Pi}_h\underline{\bm{\sigma}}^{\mathit{f}},\mathbf{e}_{\underline{\bm{\rho}}^{\mathit{f}}})_{\Omega}	
	\end{equation}
	By virtue of Theorem \ref{1220gg} and the system (\ref{1221d}), we can directly have
	\begin{equation}
		\label{1224a}
		||\mathbf{e}_{\mathbf{u}^{\mathit{f}}}||_{1,h}^2\leq C(||\mathcal{A}\mathbf{e}_{\underline{\bm{\sigma}}^{\mathit{f}}}||_{0,\Omega}^2+||\underline{\bm{\sigma}}^{\mathit{f}}-\mathbf{\Pi}_h\underline{\bm{\sigma}}^{\mathit{f}}||_{0,\Omega}^2)\leq C(||\mathcal{A}^{\frac{1}{2}}\mathbf{e}_{\underline{\bm{\sigma}}^{\mathit{f}}}||_{0,\Omega}^2+||\underline{\bm{\sigma}}^{\mathit{f}}-\mathbf{\Pi}_h\underline{\bm{\sigma}}^{\mathit{f}}||_{0,\Omega}^2),
	\end{equation}
	where $C$ is independent of $\lambda_S$ and $h$. By \cite[Proposition 5.1]{JJ2012}, there exists $\underline{\bm{\tau}}_1\in\underline{\bm{\mathit{V}}}^h$ such that $(\underline{\bm{\tau}}_1,\mathbf{e}_{\underline{\bm{\rho}}^{\mathit{f}}})_{\Omega}=(\mathbf{e}_{\underline{\bm{\rho}}^{\mathit{f}}},\mathbf{e}_{\underline{\bm{\rho}}^{\mathit{f}}})_{\Omega}$ and $||\underline{\bm{\tau}}_1||_{0,\Omega}\leq C||\mathbf{e}_{\underline{\bm{\rho}}^{\mathit{f}}}||_{0,\Omega}$, where $C$ only depends on the shape regularity of the grid.
	Substituting the test function $\underline{\bm{\tau}}_1$ into (\ref{1221d_a}), we obtain that
	$||\mathbf{e}_{\underline{\bm{\rho}}^{\mathit{f}}}||_{0,\Omega}^2\leq C(||\underline{\bm{\sigma}}^{\mathit{f}}-\mathbf{\Pi}_h\underline{\bm{\sigma}}^{\mathit{f}}||_{0,\Omega}+||\mathbf{e}_{\mathbf{u}^{\mathit{f}}}||_{1,h}+||\mathcal{A}^{\frac{1}{2}}\mathbf{e}_{\underline{\bm{\sigma}}^{\mathit{f}}}||_{0,\Omega})\cdot||\mathbf{e}_{\underline{\bm{\rho}}^{\mathit{f}}}||_{0,\Omega}.$ Then, combining (\ref{1224a}), we obtain
	\begin{equation}
		\label{1225b}
		||\mathbf{e}_{\underline{\bm{\rho}}^{\mathit{f}}}||_{0,\Omega}\leq C(||\underline{\bm{\sigma}}^{\mathit{f}}-\mathbf{\Pi}_h\underline{\bm{\sigma}}^{\mathit{f}}||_{0,\Omega}+||\mathcal{A}^{\frac{1}{2}}\mathbf{e}_{\underline{\bm{\sigma}}^{\mathit{f}}}||_{0,\Omega})	
	\end{equation}
	By using (\ref{1221e}), (\ref{1225b}), and an arithmetic-geometric mean inequality, we have
	\begin{equation}
		\label{1225d}
		||\mathcal{A}^{\frac{1}{2}}\mathbf{e}_{\underline{\bm{\sigma}}^{\mathit{f}}}||_{0,\Omega}\leq C||\underline{\bm{\sigma}}^{\mathit{f}}-\mathbf{\Pi}_h\underline{\bm{\sigma}}^{\mathit{f}}||_{0,\Omega}\leq Ch^s ||\bm{\mathit{f}}||_{0,\Omega},
	\end{equation}
	where $C$ is independent of $\lambda_S$ and $h$, and the estimates (\ref{ii}), (\ref{jj}), Assumption \ref{0128b} have been used. Combining (\ref{1224a}) and (\ref{1225d}), we obtain
	\begin{equation}
		\label{1225e}
		||\mathbf{e}_{\mathbf{u}^{\mathit{f}}}||_{1,h} \leq C||\underline{\bm{\sigma}}^{\mathit{f}}-\mathbf{\Pi}_h\underline{\bm{\sigma}}^{\mathit{f}}||_{0,\Omega}\leq Ch^s ||\bm{\mathit{f}}||_{0,\Omega},
	\end{equation}
	where $C$ is independent of $\lambda_S$ and $h$. Also, by (\ref{1225b}), (\ref{1225d}), the estimates (\ref{ii})-(\ref{kk}) and the triangle inequality, we have
	\begin{equation}
		\label{1225f}
		||\mathcal{A}^{\frac{1}{2}}(\underline{\bm{\sigma}}^{\mathit{f}}-\underline{\bm{\sigma}}^{\mathit{f}}_h)||_{0,\Omega}\leq Ch^s ||\bm{\mathit{f}}||_{0,\Omega},\hspace{1em}||\underline{\bm{\rho}}^{\mathit{f}}-\underline{\bm{\rho}}^{\mathit{f}}_h||_{0,\Omega}\leq Ch^s ||\bm{\mathit{f}}||_{0,\Omega}, 
	\end{equation}
	where $C$ is independent of $\lambda_S$ and $h$.
	Now combining (\ref{1225e}) and (\ref{1225f}), we easily find that (\ref{1221l}) is well satisfied.
	Besides, 
	$||\bm{\mathit{T}}\bm{\mathit{f}}-\bm{\mathit{T}}_h\bm{\mathit{f}}||_{0,\Omega}\leq||\mathbf{u}^{\mathit{f}}-\bm{\mathit{P}}\mathbf{u}^{\mathit{f}}||_{0,\Omega}+||\mathbf{e}_{\mathbf{u}^{\mathit{f}}}||_{1,h}\leq Ch^{s+1}||\mathbf{u}^{\mathit{f}}||_{s+1,\Omega}+Ch^s||\bm{\mathit{f}}||_{0,\Omega}\leq Ch^s||\bm{\mathit{f}}||_{0,\Omega},$
	where the triangle inequality, (\ref{ll}), Assumption \ref{0128b} have been used and  $C$ is independent of $\lambda_S$ and $h$.
\end{proof}
\subsection{The spectrum of the discrete solution operator}
\label{sec: The spectrum of the discrete solution operator}
\begin{theorem}
	\label{spectrum of T_h}
	The spectrum of $\bm{\mathit{T}}_h: \bm{\mathit{L}}^2(\Omega)\to \bm{\mathit{W}}^h$ consists of $0$ and 
	$M:=dim\bm{\mathit{W}}^h$ eigenvalues, repeated accordingly to their respective multiplicities. 
	The spectrum  
	decomposes as follows: $sp(\bm{\mathit{T}}_h)=\{0\}\cup\{\mu_{hk}\}_{k=1}^{M}$. Moreover,
	
	{\rm(i)} $\{\mu_{hk}\}_{k=1}^{M}$ is $M$ nonzero eigenvalues of $\bm{\mathit{T}}_h$ and there exist $M$ corresponding eigenvectors $\{\bm{\mathit{p}}_n\}_{n=1}^{M} \in\bm{\mathit{W}}^h$, that satisfy 
	$\mu_{h1}=||\bm{\mathit{T}}_h||,\hspace{0.5em} \mu_{h1}\geq\mu_{h2}\geq\cdot\cdot\cdot\geq\mu_{hM}>0,$
	$\bm{\mathit{T}}_h\bm{\mathit{p}}_n=\mu_{hn}\bm{\mathit{p}}_n,\hspace{0.5em}  1\leq n\leq M, \hspace{0.5em}  and \hspace{0.5em} (\bm{\mathit{p}}_{\mathit{k}},\bm{\mathit{p}}_{\mathit{l}})_{\Omega}=\delta_{\mathit{k}\mathit{l}}\hspace{0.5em}  for\hspace{0.5em}  all\hspace{0.5em}  1\leq\mathit{k},\mathit{l}\leq M.$
	And the ascent of each eigenvalue is 1.
	{\rm(ii)} $\mu_h=0$ is not an eigenvalue of $\bm{\mathit{T}}_h$.
\end{theorem}
\begin{proof}
	It's enough to use \cite[Theorem 4.11-2]{PG2013} to obtain (i). For (ii), 0 is in the spectrum of $\bm{\mathit{T}}_h$. If $0\notin sp(\bm{\mathit{T}}_h)$, then $\bm{\mathit{T}}_h: \bm{\mathit{L}}^2(\Omega)\to \bm{\mathit{W}}^h\subset\bm{\mathit{L}}^2(\Omega)$ is injective. For any $\mathbf{y}\in\bm{\mathit{W}}^h$, by (i), we can express $\mathbf{y}$ as $\mathbf{y}=\sum_{i=1}^{M}y_i\bm{\mathit{p}}_i,$
	$y_i\in\mathbb{R}$ for all $1\leq i\leq M$. Let $\mathbf{x}:=\sum_{i=1}^{M}\frac{y_i}{\mu_{hi}}\bm{\mathit{p}}_i$, then $\bm{\mathit{T}}_h\mathbf{x}=\bm{\mathit{T}}_h(\sum_{i=1}^{M}\frac{y_i}{\mu_{hi}}\bm{\mathit{p}}_i)=\sum_{i=1}^{M}y_i(\frac{1}{\mu_{hi}}\bm{\mathit{T}}_h\bm{\mathit{p}}_i)=\sum_{i=1}^{M}y_i\bm{\mathit{p}}_i=\mathbf{y}$. So $\bm{\mathit{T}}_h$ is surjective. That's, $\bm{\mathit{T}}_h$ is a bijective linear operator. This is a contradiction because $\bm{\mathit{L}}^2(\Omega)$ and $\bm{\mathit{W}}^h$ is not isomorphic. However, $\mu_h=0$ is not an eigenvalue. Suppose $\exists\mu_h=0$ and $\mathbf{0}\neq\bm{\mathit{f}}\in\bm{\mathit{W}}^h$ such that $\mathbf{u}^{\mathit{f}}_h=\bm{\mathit{T}}_h\bm{\mathit{f}}=\mu_h\bm{\mathit{f}}=\mathbf{0}$. By the relation between (\ref{aa}) and (\ref{mm}), we have $(\underline{\bm{\sigma}}^{\mathit{f}}_{h},\underline{\bm{\rho}}^{\mathit{f}}_{h})=\mu_h(\underline{\bm{\sigma}}_h,\underline{\bm{\rho}}_h)=(\mathbf{0},\mathbf{0}).$ Thus in terms of (\ref{mm_b}), we have $(\bm{\mathit{f}},\bm{\omega})_{\Omega}=0$, $\forall\bm{\omega}\in\bm{\mathit{W}}^h$. That is, $\bm{\mathit{f}}=\mathbf{0}$, which is a contradiction. 	
\end{proof}
\begin{remark}
	\label{apply to other elements}
	We point out that the analyses in section \ref{sec:The discrete problem}(i.e. Theorems \ref{1220gg}, \ref{1221a} and \ref{spectrum of T_h}) can cover some other elements, for instance, the AFW element\cite{AFW2007}, the elasticity element using the matrix bubble studied by Cockburn, Gopalakrishnan, Guzm$\rm \acute{a}$n\cite{CJJ2010} and a second elasticity element using the matrix bubble\cite[Section 2]{JJ2012} . In addition, we emphasize that because we use the element given by \cite[Section 6]{JJ2012}, which removes the bubbles, our proof is much more difficult than before.
\end{remark}
\section{Spectral approximation}
\label{sec: Spectral approximation}
In this section, by virtue of the results obtained in the above sections, we will conclude that
the numerical scheme provides a correct spectral approximation. We will approximate the $L^2$-orthogonal projection of the
eigenspace of exact displacement under discrete $H^1$-norm in Theorem \ref{0308b}. Asymptotic error estimates for the eigenvalues will be also established. Then the approximation for the eigenspace of the stress tensor will be given in Theorem \ref{0308c}.

We recall some commonly used notations as follows(see \cite{Osborn1975}). For any complex number $z\in\mathbb{C}\setminus sp(\bm{\mathit{T}})$, $R_z(\bm{\mathit{T}})=(z\bm{\mathit{I}}-\bm{\mathit{T}})^{-1}$ is the resolvent operator. Let $\mu$ be an eigenvalue with finite multiplicity $m$ of $\bm{\mathit{T}}$. The spectral projection associated with $\mu$ and $\bm{\mathit{T}}$ is defined by
$\bm{\mathit{E}}:=\frac{1}{2\pi i}\int_{\gamma}R_z(\bm{\mathit{T}})dz,$
where $\gamma$ is defined in (\ref{arc_1}) below. $\bm{\mathit{E}}$ is a projection onto the space of generalized eigenvectors associated with $\mu=\frac{1}{\lambda}$ and $\bm{\mathit{T}}$. We have known that $\bm{\mathit{T}}_h\to\bm{\mathit{T}}$ in $L^2$-norm, hence for $h$ sufficiently small, $\gamma\subset\mathbb{C}\setminus sp(\bm{\mathit{T}}_h)$ and the spectral projection,
$\bm{\mathit{E}}_h:=\frac{1}{2\pi i}\int_{\gamma}R_z(\bm{\mathit{T}}_h)dz$ exists; $\bm{\mathit{E}}_h$ is the spectral projection associated with $\bm{\mathit{T}}_h$ and the eigenvalues of $\bm{\mathit{T}}_h$ which lie in $\gamma$, and is a projection onto the direct sum of the spaces of generalized eigenvectors corresponding to these eigenvalues. There are $m$ eigenvalues of $\bm{\mathit{T}}_h$ in $\gamma$, denoted as $\mu_{h1},\mu_{h2},\cdot\cdot\cdot,\mu_{hm}$ (repeated according to their respective multiplicities. We also define $\frac{1}{\lambda_{hi}}=\mu_{hi},i=1,2,\cdot\cdot\cdot,m$). $R(\bm{\mathit{E}}), R(\bm{\mathit{E}}_h)$ are generalized eigenspaces, where $R$ denotes the range. Let $\bm{\mathit{W}}(h):=\bm{\mathit{W}}^h+\bm{\mathit{H}}^1_{0,\Gamma_0}(\Omega)$. 
Given two closed subspaces $\bm{\mathit{M}}$ and  $\bm{\mathit{N}}$ of $\bm{\mathit{W}}(h)$, for $\mathbf{x}\in\bm{\mathit{W}}(h)$, we define $\delta(\mathbf{x},\bm{\mathit{M}}):=\inf\limits_{\mathbf{y}\in\bm{\mathit{M}}}||\mathbf{x}-\mathbf{y}||_{1,h}$, $\delta(\bm{\mathit{M}}, \bm{\mathit{N}})=\sup\limits_{\mathbf{y}\in\bm{\mathit{M}},||\mathbf{y}||_{1,h}=1}\delta(\mathbf{y},\bm{\mathit{N}})$ and %the gap between $\bm{\mathit{M}}$ and  $\bm{\mathit{N}}$:
$\hat{\delta}(\bm{\mathit{M}}, \bm{\mathit{N}})=\max\{\delta(\bm{\mathit{M}}, \bm{\mathit{N}}),\delta(\bm{\mathit{N}}, \bm{\mathit{M}})\}$. $\hat{\delta}(\bm{\mathit{M}}, \bm{\mathit{N}})$ is called the gap between $\bm{\mathit{M}}$ and $\bm{\mathit{N}}$.

For convenience, we divide it into three cases in some analyses below(i.e. in Lemmas \ref{0220c}, \ref{0126c}, \ref{0127b} and \ref{0127c}; Theorems \ref{0308b} and \ref{0308c}):
\begin{subequations}
	\label{0308a}
	\begin{align}
		&{\rm \textbf{Case 1}: There\hspace{0.5em} exists \hspace{0.5em} a\hspace{0.5em} constant \hspace{0.5em} M>0 \hspace{0.5em}such\hspace{0.5em} that \hspace{0.5em} \lambda_S\leq M.}\label{Case1}\\
		&{\rm \textbf{Case 2}: For \hspace{0.5em} any\hspace{0.5em} constant \hspace{0.5em} M>0, \lambda_S\geq M(i.e. \hspace{0.5em} \lambda_S\to\infty)}.\label{Case2}\\
		&{\rm \textbf{Case 3}: \lambda_S=\infty}\label{Case3}.
	\end{align}	
\end{subequations}

To recall that $\bm{\mathit{T}}$ actually depends on $\lambda_S$, we will denote it by $\bm{\mathit{T}}_{\lambda_S}$ if we need to strictly tell the difference between $\bm{\mathit{T}}_{\lambda_S}$ and $\bm{\mathit{T}}_{\infty}$(i.e.$\lambda_S=\infty$) below. Otherwise, we will use $\bm{\mathit{T}}$ broadly. We also use the same notation convention for $\bm{\mathit{T}}_h,\bm{\mathit{E}},\bm{\mathit{E}}_h,\lambda,\mu,\lambda_h$ and $\mu_h$.
We give the definition of $\gamma$ and $\gamma_{\infty}$:
\begin{subequations}
	\label{0314a}
	\begin{align}
		&{\rm \gamma:boundary \hspace{0.2em} of\hspace{0.2em} D_{\mu}:=\{z\in\mathbb{C}:|z-\mu|\leq d_{\mu},
			\hspace{0.2em} d_{\mu}=\frac{1}{2}dist(\mu,sp(\bm{\mathit{T}})\backslash\{\mu\})\}},\label{arc_1}\\
		&{\rm \gamma_{\infty}:boundary \hspace{0.2em} of \hspace{0.2em} D_{\infty}:=\{z\in\mathbb{C}:|z-\mu_{\infty}|\leq d_{\infty}, d_{\infty}=\frac{1}{2}dist(\mu_{\infty},sp(\bm{\mathit{T}}_{\infty})
			\backslash\{\mu_{\infty}\})\}}.\label{arc_2}
	\end{align}	
\end{subequations}
Here $\mu_{\infty}$ is an eigenvalue of $\bm{\mathit{T}}_{\infty}$, which is defined detailedly below. 
From (\ref{arc_1}), we have $D_{\mu}\cap sp(\bm{\mathit{T}})=\{\mu\}$, and for all $z\in\gamma$, $dist(z,sp(\bm{\mathit{T}}))=dist(\gamma,sp(\bm{\mathit{T}}))=d_{\mu}$.
\subsection{Spectral approximation}
\label{sec: Spectral approximations_begining}
Before spectral analysis, we will give a spectral characterization for incompressible elasticity and the convergence analysis for solution operators.
In the limit case $\lambda_S=\infty$,  the bilinear forms $A$ and $B$ defined as before change 
since the term where $\lambda_S$ appears vanishes. Hence the limit eigenvalue problem reads as follows: Find $\lambda_{\infty}\in\mathbb{R}$ and $\mathbf{u}_{\infty}\in\bm{\mathit{L}}^2(\Omega)$ with corresponding $(\underline{\bm{\sigma}}_{\infty},\underline{\bm{\rho}}_{\infty})\in\underline{\bm{\mathit{Y}}}\times\underline{\bm{\mathit{Q}}}$ such that
\begin{equation}
	\label{ya}
	A_{\infty}((\underline{\bm{\sigma}}_{\infty},\mathbf{u}_{\infty},\underline{\bm{\rho}}_{\infty}),(\underline{\bm{\upsilon}},\bm{\omega},\underline{\bm{\eta}}))=\lambda_{\infty} B_{\infty}(\mathbf{u}_{\infty},\bm{\omega}),\quad \forall(\underline{\bm{\upsilon}},\bm{\omega},\underline{\bm{\eta}})\in\underline{\bm{\mathit{Y}}}\times\bm{\mathit{L}}^2(\Omega)\times\underline{\bm{\mathit{Q}}}
\end{equation}
with
$A_{\infty}((\underline{\bm{\sigma}},\mathbf{u},\underline{\bm{\rho}}),(\underline{\bm{\upsilon}},\bm{\omega},\underline{\bm{\eta}}))=\frac{1}{2\mu_S}(\underline{\bm{\sigma}}^D,\underline{\bm{\upsilon}}^D)_{\Omega}+(\mathbf{u},\textbf{div}\underline{\bm{\upsilon}})_{\Omega}+(\underline{\bm{\rho}},\underline{\bm{\upsilon}})_{\Omega}+(\rho_S^{-1}\textbf{div}\underline{\bm{\sigma}},\bm{\omega})_{\Omega}$ $+(\underline{\bm{\sigma}},\underline{\bm{\eta}})_{\Omega},$
and
$B_{\infty}(\mathbf{u},\bm{\omega})=-(\mathbf{u},\bm{\omega})_{\Omega},$
for $(\underline{\bm{\sigma}},\mathbf{u},\underline{\bm{\rho}}),(\underline{\bm{\upsilon}},\bm{\omega},\underline{\bm{\eta}})\in\underline{\bm{\mathit{Y}}}\times\bm{\mathit{L}}^2(\Omega)\times\underline{\bm{\mathit{Q}}}$.

The solution operator $\bm{\mathit{T}}_{\infty}:\bm{\mathit{L}}^2(\Omega)\to\bm{\mathit{L}}^2(\Omega)$, defined for any $\bm{\mathit{f}}\in\bm{\mathit{L}}^2(\Omega)$ by
\begin{equation}
	\label{yd}
	A_{\infty}((\underline{\bm{\sigma}}_{\infty}^{\mathit{f}},\mathbf{u}_{\infty}^{\mathit{f}},\underline{\bm{\rho}}_{\infty}^{\mathit{f}}),(\underline{\bm{\upsilon}},\bm{\omega},\underline{\bm{\eta}}))= B_{\infty}(\bm{\mathit{f}},\bm{\omega}),\quad \forall(\underline{\bm{\upsilon}},\bm{\omega},\underline{\bm{\eta}})\in\underline{\bm{\mathit{Y}}}\times\bm{\mathit{L}}^2(\Omega)\times\underline{\bm{\mathit{Q}}}
\end{equation}
where $\mathbf{u}_{\infty}^{\mathit{f}}=\bm{\mathit{T}}_{\infty}\bm{\mathit{f}}$. It is easy to check that  $(\lambda_{\infty},\mathbf{u}_{\infty})\in \mathbb{R}\times\bm{\mathit{L}}^2(\Omega)$ solves ($\ref{ya}$) if and only if $(\mu_{\infty},\mathbf{u}_{\infty})$, with $\mu_{\infty}=\lambda_{\infty}^{-1}$ is an eigenpair of $\bm{\mathit{T}}_{\infty}$, i.e. if and only if $\bm{\mathit{T}}_{\infty}\mathbf{u}_{\infty}=\lambda_{\infty}^{-1}\mathbf{u}_{\infty} (\lambda_{\infty}>0).$
In this case, we have $(\underline{\bm{\sigma}}^{\mathit{f}}_{\infty}, \underline{\bm{\rho}}^{\mathit{f}}_{\infty})=\lambda_{\infty}^{-1}(\underline{\bm{\sigma}}_{\infty}, \underline{\bm{\rho}}_{\infty})$. In fact, using the same analysis and regularity assumption as before, we can obtain that $\bm{\mathit{T}}_{\infty}:\bm{\mathit{L}}^2(\Omega)\to\bm{\mathit{H}}^1(\Omega)$ is compact and has similar characterizations of the spectrum to $\bm{\mathit{T}}_{\lambda_S}$. 

Let $m_{\infty}$ be the multiplicity of $\mu_{\infty}$. From (\ref{arc_2}), we have $D_{\infty}\cap sp(\bm{\mathit{T}}_{\infty})=\{\mu_{\infty}\}$, and for all $z\in\gamma_{\infty}$, $dist(z,sp(\bm{\mathit{T}}_{\infty}))=dist(\gamma_{\infty},sp(\bm{\mathit{T}}_{\infty}))=d_{\infty}$. The spectral projection associated with $\mu_{\infty}$ and $\bm{\mathit{T}}_{\infty}$ is defined by
$\bm{\mathit{E}}_{\infty}:=\frac{1}{2\pi i}\int_{\gamma_{\infty}}R_z(\bm{\mathit{T}}_{\infty})dz.$ In fact, by Lemma \ref{0220a} below, we know that for $\lambda_S\to\infty$, $\gamma_{\infty}\subset(\mathbb{C}\backslash sp(\bm{\mathit{T}}_{\lambda_S}))$ and the spectral projection, $\bm{\mathit{E}}_{\lambda_S}:=\frac{1}{2\pi i}\int_{\gamma_{\infty}}R_z(\bm{\mathit{T}}_{\lambda_S})dz$ exists; $\bm{\mathit{E}}_{\lambda_S}$ is the spectral projection associated with $\bm{\mathit{T}}_{\lambda_S}$ and the eigenvalues of $\bm{\mathit{T}}_{\lambda_S}$ which lie in $\gamma_{\infty}$, and is a projection onto the direct sum of the spaces of generalized eigenvectors corresponding to these eigenvalues. There are $m_{\infty}$ eigenvalues of $\bm{\mathit{T}}_{\lambda_S}$ in $\gamma_{\infty}$, denoted as $\mu_{\lambda_S}^1,\mu_{\lambda_S}^2,\cdot\cdot\cdot,\mu_{\lambda_S}^{m_{\infty}}$ (repeated according to their respective multiplicities. We also define $1/\lambda_{\lambda_S}^i=\mu_{\lambda_S}^i,i=1,2,\cdot\cdot\cdot,m_{\infty}$).
\begin{lemma}
	\label{0220a}	
	There exists a constant $C>0$ independent of $\lambda_S$ such that
	$||(\bm{\mathit{T}}_{\lambda_S}-\bm{\mathit{T}}_{\infty})\bm{\mathit{f}}||_{1,\Omega}\leq\frac{C}{\lambda_S}||\bm{\mathit{f}}||_{0,\Omega},$ for all $\bm{\mathit{f}}\in\bm{\mathit{L}}^2(\Omega).$	
\end{lemma}
\begin{proof}
	Let $\bm{\mathit{f}}\in\bm{\mathit{L}}^2(\Omega)$ and let $\mathbf{u}_{\lambda_S}^{\mathit{f}}=\bm{\mathit{T}}_{\lambda_S}\bm{\mathit{f}}$ (the corresponding stress and rotation are denoted as $\underline{\bm{\sigma}}_{\lambda_S}^{\mathit{f}}$ and $\underline{\bm{\rho}}_{\lambda_S}^{\mathit{f}}$). Then in terms of (\ref{c}) and (\ref{1220a}), we have
	$
	A_{\infty}((\underline{\bm{\sigma}}_{\lambda_S}^{\mathit{f}},\mathbf{u}_{\lambda_S}^{\mathit{f}},\underline{\bm{\rho}}_{\lambda_S}^{\mathit{f}}),(\underline{\bm{\upsilon}},\bm{\omega},\underline{\bm{\eta}}))= B_{\infty}(\bm{\mathit{f}},\bm{\omega})-\frac{(tr(\underline{\bm{\sigma}}_{\lambda_S}^{\mathit{f}}),tr(\underline{\bm{\upsilon}}))_{\Omega}}{n(n\lambda_S+2\mu_S)},
	$
	for all $(\underline{\bm{\upsilon}},\bm{\omega},\underline{\bm{\eta}})\in\underline{\bm{\mathit{Y}}}\times\bm{\mathit{L}}^2(\Omega)\times\underline{\bm{\mathit{Q}}}$. Combining (\ref{yd}), we have
	$A_{\infty}((\underline{\bm{\sigma}}_{\lambda_S}^{\mathit{f}}-\underline{\bm{\sigma}}_{\infty}^{\mathit{f}},\mathbf{u}_{\lambda_S}^{\mathit{f}}-\mathbf{u}_{\infty}^{\mathit{f}},\underline{\bm{\rho}}_{\lambda_S}^{\mathit{f}}-\underline{\bm{\rho}}_{\infty}^{\mathit{f}}),(\underline{\bm{\upsilon}},\bm{\omega},\underline{\bm{\eta}}))$ $=-\frac{(tr(\underline{\bm{\sigma}}_{\lambda_S}^{\mathit{f}}),tr(\underline{\bm{\upsilon}}))_{\Omega}}{n(n\lambda_S+2\mu_S)}$.
	Then using the conventional inf-sup condition, we obtain
	\begin{equation}
		\label{yl}
		\begin{aligned}		
			&||\underline{\bm{\sigma}}_{\lambda_S}^{\mathit{f}}-\underline{\bm{\sigma}}_{\infty}^{\mathit{f}}||_{0,\Omega}\leq C\sup\limits_{\substack{(\underline{\bm{\upsilon}},\bm{\omega},\underline{\bm{\eta}})\in\\\underline{\bm{\mathit{Y}}}\times\bm{\mathit{L}}^2(\Omega)\times\underline{\bm{\mathit{Q}}}}}\frac{A_{\infty}((\underline{\bm{\sigma}}_{\lambda_S}^{\mathit{f}}-\underline{\bm{\sigma}}_{\infty}^{\mathit{f}},\mathbf{u}_{\lambda_S}^{\mathit{f}}-\mathbf{u}_{\infty}^{\mathit{f}},\underline{\bm{\rho}}_{\lambda_S}^{\mathit{f}}-\underline{\bm{\rho}}_{\infty}^{\mathit{f}}),(\underline{\bm{\upsilon}},\bm{\omega},\underline{\bm{\eta}}))}{||(\underline{\bm{\upsilon}},\bm{\omega},\underline{\bm{\eta}})||}\\
			&\leq C\sup\limits_{\substack{(\underline{\bm{\upsilon}},\bm{\omega},\underline{\bm{\eta}})\in\\\underline{\bm{\mathit{Y}}}\times\bm{\mathit{L}}^2(\Omega)\times\underline{\bm{\mathit{Q}}}}}\frac{C\cdot\frac{1}{n(n\lambda_S+2\mu_S)}\cdot||\underline{\bm{\sigma}}_{\lambda_S}^{\mathit{f}}||_{0,\Omega}\cdot||\underline{\bm{\upsilon}}||_{0,\Omega}}{||(\underline{\bm{\upsilon}},\bm{\omega},\underline{\bm{\eta}})||}\leq\frac{C}{\lambda_S}||\underline{\bm{\sigma}}_{\lambda_S}^{\mathit{f}}||_{0,\Omega}\leq\frac{C}{\lambda_S}||\bm{\mathit{f}}||_{0,\Omega}
		\end{aligned}	
	\end{equation}
	By korn's inequality, (\ref{yl}) and (\ref{1220a}), we have
	\begin{equation*}
		\label{0220b}
		\begin{aligned}		
			&||\mathbf{u}_{\lambda_S}^{\mathit{f}}-\mathbf{u}_{\infty}^{\mathit{f}}||_{1,\Omega}\leq C||\underline{\bm{\epsilon}}(\mathbf{u}_{\lambda_S}^{\mathit{f}}-\mathbf{u}_{\infty}^{\mathit{f}})||_{0,\Omega}=C||\frac{1}{2\mu_S}(\underline{\bm{\sigma}}_{\lambda_S}^{\mathit{f}})^D+\frac{1}{n(n\lambda_S+2\mu_S)}tr(\underline{\bm{\sigma}}_{\lambda_S}^{\mathit{f}})\bm{\mathit{I}}\\
			&\quad\quad-\frac{1}{2\mu_S}(\underline{\bm{\sigma}}_{\infty}^{\mathit{f}})^D||_{0,\Omega}\leq C||\underline{\bm{\sigma}}_{\lambda_S}^{\mathit{f}}-\underline{\bm{\sigma}}_{\infty}^{\mathit{f}}||_{0,\Omega}+\frac{C}{\lambda_S}||\underline{\bm{\sigma}}_{\lambda_S}^{\mathit{f}}||_{0,\Omega}\leq\frac{C}{\lambda_S}||\bm{\mathit{f}}||_{0,\Omega}
		\end{aligned}	
	\end{equation*}
	That's, $||(\bm{\mathit{T}}_{\lambda_S}-\bm{\mathit{T}}_{\infty})\bm{\mathit{f}}||_{1,\Omega}\leq\frac{C}{\lambda_S}||\bm{\mathit{f}}||_{0,\Omega}, \forall\bm{\mathit{f}}\in\bm{\mathit{L}}^2(\Omega)
	$. Then we have 
	$||\bm{\mathit{T}}_{\lambda_S}-\bm{\mathit{T}}_{\infty}||_1:= \sup\limits_{\mathbf{0}\neq\bm{\mathit{f}}\in\bm{\mathit{H}}^1_{0,\Gamma_0}(\Omega)} \frac{||(\bm{\mathit{T}}_{\lambda_S}-\bm{\mathit{T}}_{\infty})\bm{\mathit{f}}||_{1,\Omega}}{||\bm{\mathit{f}}||_{1,\Omega}}\leq C/\lambda_S.$
\end{proof}

\begin{lemma}
	\label{0220c}
	There exists a constant $C>0$ independent of $\lambda_S$ and
	$z$ such that
	\begin{equation}
		\label{0126b}
		||(z\bm{\mathit{I}}-\bm{\mathit{T}})\bm{\mathit{f}}||_{1,\Omega}\geq C dist(z,sp(\bm{\mathit{T}}))||\bm{\mathit{f}}||_{1,\Omega}\hspace{0.5em}\hspace{0.5em}\forall\bm{\mathit{f}}\in\bm{\mathit{H}}^1_{0,\Gamma_0}(\Omega),
	\end{equation}
	where $z\in\gamma$($\gamma$ is defined in (\ref{arc_1})) in \textbf{Case 1}(i.e.(\ref{Case1})); $z\in\gamma_{\infty}$($\gamma_{\infty}$ is defined in (\ref{arc_2})) in \textbf{Case 2}((\ref{Case2})) or \textbf{Case 3}((\ref{Case3})). $dist(z,sp(\bm{\mathit{T}}))$ is the distance between $z$ and the spectrum of $\bm{\mathit{T}}$ in the complex plane, which in principle depends on $\lambda_S$.
\end{lemma}
\begin{proof}
	In \textbf{Case 1}, clearly we know $\mathbb{C}\backslash sp(\bm{\mathit{T}})$ is open, hence $\forall z\in\gamma\subset\mathbb{C}\backslash sp(\bm{\mathit{T}})$, $\exists z_0\in\mathbb{C}\backslash sp(\bm{\mathit{T}})$, $\delta_0>0$ such that $z\in B(z_0,\delta_0)\subset(\mathbb{C}\backslash sp(\bm{\mathit{T}}))$. In terms of \cite[Theorem 1.13(b)]{RS2012}, we know for each $\delta\in(0,1)$, $||(z\bm{\mathit{I}}-\bm{\mathit{T}})^{-1}||_1\leq\frac{c(z_0)}{1-\delta}$. Since $0<dist(z,sp(\bm{\mathit{T}}))\leq|z|<2|z|,$ then $0<1-\frac{dist(z,sp(\bm{\mathit{T}}))}{2|z|}<1$. Thus if we set $\delta=1-\frac{dist(z,sp(\bm{\mathit{T}}))}{2|z|}$, then we have $||(z\bm{\mathit{I}}-\bm{\mathit{T}})^{-1}||_1\leq C/dist(z,sp(\bm{\mathit{T}}))$, where $C$ is independent of $\lambda_S$. Hence (\ref{0126b}) is proved. The analysis in \textbf{Case 3} is similar to \textbf{Case 1}.
	In \textbf{Case 2}, 
	For $\lambda_S$ big enough, $\forall z\in\gamma_{\infty}\subset(\mathbb{C}\backslash sp(\bm{\mathit{T}}_{\lambda_S}))$, combining \textbf{Case 3}, there exists a constant $C_1$ independent of $\lambda_S$ such that
	\begin{equation*}
		\label{0220d}
		\begin{aligned}		
			&||(z\bm{\mathit{I}}-\bm{\mathit{T}}_{\lambda_S})\bm{\mathit{f}}||_{1,\Omega}\geq||(z\bm{\mathit{I}}-\bm{\mathit{T}}_{\infty})\bm{\mathit{f}}||_{1,\Omega}-||(\bm{\mathit{T}}_{\infty}-\bm{\mathit{T}}_{\lambda_S})\bm{\mathit{f}}||_{1,\Omega}\\
			&\quad\geq C_1dist(z,sp(\bm{\mathit{T}}_{\infty}))||\bm{\mathit{f}}||_{1,\Omega}-||(\bm{\mathit{T}}_{\infty}-\bm{\mathit{T}}_{\lambda_S})\bm{\mathit{f}}||_{1,\Omega}
		\end{aligned}	
	\end{equation*}
	From Lemma \ref{0220a}, we know for $\lambda_S$ big enough, $||(\bm{\mathit{T}}_{\infty}-\bm{\mathit{T}}_{\lambda_S})\bm{\mathit{f}}||_{1,\Omega}\leq\frac{C_1}{2}dist(z,$ $sp(\bm{\mathit{T}}_{\infty}))||\bm{\mathit{f}}||_{1,\Omega}$. Then for all $z\in\gamma_{\infty}$, $||(z\bm{\mathit{I}}-\bm{\mathit{T}}_{\lambda_S})\bm{\mathit{f}}||_{1,\Omega}\geq Cdist(z,sp(\bm{\mathit{T}}_{\infty}))||\bm{\mathit{f}}||_{1,\Omega}$ $\geq C(|z-\mu_{\lambda_S}^i|-|\mu_{\lambda_S}^i-\mu_{\infty}|)||\bm{\mathit{f}}||_{1,\Omega},$
	for all $1\leq i\leq m_{\infty}$. Since by Lemma \ref{0220a} and \cite[Theorem 6]{Osborn1975}, we have $|\mu_{\lambda_S}^i-\mu_{\infty}|\to0$ as $\lambda_S\to\infty$, then for all $z\in\gamma_{\infty}$ and $\lambda_S$ big enough, we can obtain the desired assertion.
\end{proof}
%\begin{remark}
%If $F$ is a compact subset of $\mathbb{C}\backslash sp(\bm{\mathit{T}})$, we deduce from Lemma \ref{0220c} that for all $z\in F$,
%$$||(z\bm{\mathit{I}}-\bm{\mathit{T}})^{-1}||_1:=\sup\limits_{\mathbf{0}\neq\bm{\mathit{f}}\in\bm{\mathit{H}}^1_{0,\Gamma_0}(\Omega)}\frac{||(z\bm{\mathit{I}}-\bm{\mathit{T}}_h)^{-1}\bm{\mathit{f}}||_{1,\Omega}}{||\bm{\mathit{f}}||_{1,\Omega}}\leq\frac{1}{dist(z,sp(\bm{\mathit{T}}))}.$$
%\end{remark}
%Since, according to Theorem \ref{theorem3.1}, the spectrum of $\bm{\mathit{T}}$ lies in the
%disk $D:=\{z\in\mathbb{C}:|z|\leq2||\bm{\mathit{T}}||\}$, we restrict our attention to this subset of the complex plane.
\begin{lemma}
	\label{0126c}
	There exists a constant $C>0$ independent of  $h$ and $\lambda_S$ such that
	$$||(z\bm{\mathit{I}}-\bm{\mathit{T}})\bm{\mathit{f}}||_{1,h}\geq Cdist(z,sp(\bm{\mathit{T}}))|z|||\bm{\mathit{f}}||_{1,h} \hspace{0.5em}\hspace{0.5em}\forall\bm{\mathit{f}}\in\bm{\mathit{W}}(h),$$
	where $z\in\gamma$($\gamma$ is defined in (\ref{arc_1})) in \textbf{Case 1}(i.e.(\ref{Case1})); $z\in\gamma_{\infty}$($\gamma_{\infty}$ is defined in (\ref{arc_2})) in \textbf{Case 2}((\ref{Case2})) or \textbf{Case 3}((\ref{Case3})).
\end{lemma}
\begin{proof}
	We introduce $\bm{\mathit{f}}^{\star}:=\bm{\mathit{T}}\bm{\mathit{f}}\in\bm{\mathit{H}}^{1+s}_{0,\Gamma_0}(\Omega)$, and notice that $(z\bm{\mathit{I}}-\bm{\mathit{T}})\bm{\mathit{f}}^{\star}=\bm{\mathit{T}}(z\bm{\mathit{I}}-\bm{\mathit{T}})\bm{\mathit{f}}$. By virtue of Lemma \ref{0220c}, we have that
	\begin{equation*}
		\begin{aligned}	
			Cdist(z,sp(\bm{\mathit{T}}))||\bm{\mathit{f}}^{\star}||_{1,\Omega}&\leq ||(z\bm{\mathit{I}}-\bm{\mathit{T}})\bm{\mathit{f}}^{\star}||_{1,\Omega}=||\bm{\mathit{T}}(z\bm{\mathit{I}}-\bm{\mathit{T}})\bm{\mathit{f}}||_{1,\Omega}\\
			&\leq C||(z\bm{\mathit{I}}-\bm{\mathit{T}})\bm{\mathit{f}}||_{0,\Omega}\leq C||(z\bm{\mathit{I}}-\bm{\mathit{T}})\bm{\mathit{f}}||_{1,h}. 	
		\end{aligned}	
	\end{equation*}
	Finally, by the triangle inequality,
	\begin{equation*}
		\begin{aligned}	
			&||\bm{\mathit{f}}||_{1,h}\leq \frac{1}{|z|}(||(z\bm{\mathit{I}}-\bm{\mathit{T}})\bm{\mathit{f}}||_{1,h}+||\bm{\mathit{f}}^{\star}||_{1,\Omega})\leq \frac{1}{|z|}(1+\frac{C}{dist(z,sp(\bm{\mathit{T}}))})||(z\bm{\mathit{I}}-\bm{\mathit{T}})\bm{\mathit{f}}||_{1,h}\\&\quad= \frac{C+dist(z,sp(\bm{\mathit{T}}))}{|z|dist(z,sp(\bm{\mathit{T}}))}||(z\bm{\mathit{I}}-\bm{\mathit{T}})\bm{\mathit{f}}||_{1,h}\leq\frac{C}{|z|dist(z,sp(\bm{\mathit{T}}))}||(z\bm{\mathit{I}}-\bm{\mathit{T}})\bm{\mathit{f}}||_{1,h},	 	
		\end{aligned}	
	\end{equation*}
	where the boundedness of $dist(z,sp(\bm{\mathit{T}}))$ in three cases have been used.	
\end{proof}
%\begin{remark}
%\label{0127a}
%If $F$ is a compact subset of $D\backslash sp(\bm{\mathit{T}})$, we deduce from Lemma \ref{0126c} that
%there exists a constant $C>0$ independent of $h$ and $\lambda_S$ such that, for all $z\in F$,
%$$||(z\bm{\mathit{I}}-\bm{\mathit{T}})^{-1}|_{\bm{\mathit{W}}(h)}||_{1,h}:=\sup\limits_{\mathbf{0}\neq\bm{\mathit{f}}\in\bm{\mathit{W}}(h)}\frac{||(z\bm{\mathit{I}}-\bm{\mathit{T}})^{-1}\bm{\mathit{f}}||_{1,h}}{||\bm{\mathit{f}}||_{1,h}}\leq\frac{C}{dist(F,sp(\bm{\mathit{T}}))|z|}.$$	
%\end{remark}
\begin{lemma}
	\label{0127b}
	There exists $h_0>0$ such that if $h\leq h_0$,
	$$||(z\bm{\mathit{I}}-\bm{\mathit{T}}_h)\bm{\mathit{f}}||_{1,h}\geq Cdist(z,sp(\bm{\mathit{T}}))|z|||\bm{\mathit{f}}||_{1,h}\hspace{0.5em}\hspace{0.5em}\forall\bm{\mathit{f}}\in\bm{\mathit{W}}(h),$$
	where $z\in\gamma$($\gamma$ is defined in (\ref{arc_1})) in \textbf{Case 1}(i.e.(\ref{Case1})); $z\in\gamma_{\infty}$($\gamma_{\infty}$ is defined in (\ref{arc_2})) in \textbf{Case 2}((\ref{Case2})) or \textbf{Case 3}((\ref{Case3})). $C>0$ is independent of $h$ and $\lambda_S$.
\end{lemma}
\begin{proof}
	By 
	$(z\bm{\mathit{I}}-\bm{\mathit{T}}_h)\bm{\mathit{f}}=(z\bm{\mathit{I}}-\bm{\mathit{T}})\bm{\mathit{f}}+(\bm{\mathit{T}}-\bm{\mathit{P}}\bm{\mathit{T}})\bm{\mathit{f}}+(\bm{\mathit{P}}\bm{\mathit{T}}-\bm{\mathit{T}}_h)\bm{\mathit{f}}$
	and Lemma \ref{0126c},
	$$||(z\bm{\mathit{I}}-\bm{\mathit{T}}_h)\bm{\mathit{f}}||_{1,h}\geq Cdist(z,sp(\bm{\mathit{T}}))|z|||\bm{\mathit{f}}||_{1,h}-||(\bm{\mathit{T}}-\bm{\mathit{P}}\bm{\mathit{T}})\bm{\mathit{f}}||_{1,h}-||(\bm{\mathit{P}}\bm{\mathit{T}}-\bm{\mathit{T}}_h)\bm{\mathit{f}}||_{1,h}$$
	Clearly we have $
	||\bm{\mathit{T}}\bm{\mathit{f}}-\bm{\mathit{P}}(\bm{\mathit{T}}\bm{\mathit{f}})||_{1,h}\leq Ch^s||\bm{\mathit{T}}\bm{\mathit{f}}||_{1+s,\Omega}\leq Ch^s||\bm{\mathit{f}}||_{0,\Omega}\leq Ch^s||\bm{\mathit{f}}||_{1,h}$,
	and by virtue of Theorem \ref{1221a}, we know that
	\begin{equation}
		\label{0126e}
		||(\bm{\mathit{P}}\bm{\mathit{T}}-\bm{\mathit{T}}_h)\bm{\mathit{f}}||_{1,h}\leq Ch^s||\bm{\mathit{f}}||_{0,\Omega}\leq Ch^s||\bm{\mathit{f}}||_{1,h}
	\end{equation}
	Thus for $h$ small enough, $||\bm{\mathit{T}}\bm{\mathit{f}}-\bm{\mathit{P}}(\bm{\mathit{T}}\bm{\mathit{f}})||_{1,h}+||(\bm{\mathit{P}}\bm{\mathit{T}}-\bm{\mathit{T}}_h)\bm{\mathit{f}}||_{1,h}\leq\frac{C|z|dist(z,sp(\bm{\mathit{T}}))}{2} $ $||\bm{\mathit{f}}||_{1,h}$.
	Then we immediately obtain the desired assertion. 	
\end{proof}	
%\begin{remark}
%If $F$ is a compact subset of $D\backslash sp(\bm{\mathit{T}})$ and $h$ is small enough, we deduce from Lemma \ref{0127b} that $(z\bm{\mathit{I}}-\bm{\mathit{T}}_h):\bm{\mathit{W}}(h)\to\bm{\mathit{W}}(h)$ is invertible for all $z\in F$. Hence $F\subset D\backslash sp(\bm{\mathit{T}}_h)$. Consequently, for $h$ small enough, the numerical method does not introduce spurious eigenvalues. Moreover, we have that for all $z\in F$,
%$$||(z\bm{\mathit{I}}-\bm{\mathit{T}}_h)^{-1}|_{\bm{\mathit{W}}(h)}||_{1,h}:=\sup\limits_{\mathbf{0}\neq\bm{\mathit{f}}\in\bm{\mathit{W}}(h)}\frac{||(z\bm{\mathit{I}}-\bm{\mathit{T}}_h)^{-1}\bm{\mathit{f}}||_{1,h}}{||\bm{\mathit{f}}||_{1,h}}\leq\frac{C}{dist(F,sp(\bm{\mathit{T}}))|z|}.$$

%Given an isolated eigenvalue $\mu$ of $\bm{\mathit{T}}$, we define
%\begin{equation}
%	\label{Dro}
%	d_{\mu}:=\frac{1}{2}dist(\mu,sp(\bm{\mathit{T}})\backslash\{\mu\}).
%\end{equation}

%It follows that the closed disk $D_{\mu}:=\{z\in\mathbb{C}:\hspace{0.5em}|z-\mu|\leq d_{\mu}\}$ of the complex plane, with center $\mu$ and boundary $\gamma$ is such that $D_{\mu}\cap sp(\bm{\mathit{T}})=\{\mu\}$. For our case, it is clear that for all $z\in\gamma$, $dist(z,sp(\bm{\mathit{T}}))=dist(\gamma,sp(\bm{\mathit{T}}))=d_{\mu}$.
%\end{remark}

\begin{lemma}
	\label{0127c}
	There exists $h_0>0$ such that, for all $h\leq h_0$, $$||(\bm{\mathit{P}}\bm{\mathit{E}}-\bm{\mathit{E}}_h)|_{\bm{\mathit{W}}(h)}||_{1,h}\leq \frac{C}{d}||(\bm{\mathit{P}}\bm{\mathit{T}}-\bm{\mathit{T}}_h)|_{\bm{\mathit{W}}(h)}||_{1,h},$$
	with $C>0$ independent of $h$ and $\lambda_S$, where $d=d_{\mu}$($d_{\mu}$ is defined in (\ref{arc_1})) for  \textbf{Case 1}(i.e.(\ref{Case1})); $d=d_1^2/d_{\infty}$ with $d_1=dist(\gamma_{\infty},sp(\bm{\mathit{T}}_{\lambda_S}))$(here $d_{\infty},\gamma_{\infty}$ are defined in (\ref{arc_2})) for \textbf{Case 2}(i.e.(\ref{Case2})); or $d=d_{\infty}$ for \textbf{Case 3}(i.e.(\ref{Case3})). And $||(\bm{\mathit{P}}\bm{\mathit{E}}-\bm{\mathit{E}}_h)|_{\bm{\mathit{W}}(h)}||_{1,h}$ denotes the restriction of $\bm{\mathit{P}}\bm{\mathit{E}}-\bm{\mathit{E}}_h$ to $\bm{\mathit{W}}(h)$.
\end{lemma}
\begin{proof}
	We deduce from the identity 
	$\bm{\mathit{P}}((z\bm{\mathit{I}}-\bm{\mathit{T}})^{-1}-(z\bm{\mathit{I}}-\bm{\mathit{T}}_h)^{-1})=\bm{\mathit{P}}(z\bm{\mathit{I}}-\bm{\mathit{T}}_h)^{-1}(\bm{\mathit{T}}-\bm{\mathit{T}}_h)(z\bm{\mathit{I}}-\bm{\mathit{T}})^{-1}=(z\bm{\mathit{I}}-\bm{\mathit{T}}_h)^{-1}(\bm{\mathit{P}}\bm{\mathit{T}}-\bm{\mathit{T}}_h)(z\bm{\mathit{I}}-\bm{\mathit{T}})^{-1}$
	that, in \textbf{Case 1}, for any $\bm{\mathit{f}}\in\bm{\mathit{W}}(h)$
	\begin{equation}
		\label{0128a}
		\begin{aligned}	
			&||(\bm{\mathit{P}}\bm{\mathit{E}}-\bm{\mathit{E}}_h)\bm{\mathit{f}}||_{1,h}
			\leq\frac{1}{2\pi}\int_{\gamma}||(z\bm{\mathit{I}}-\bm{\mathit{T}}_h)^{-1}(\bm{\mathit{P}}\bm{\mathit{T}}-\bm{\mathit{T}}_h)(z\bm{\mathit{I}}-\bm{\mathit{T}})^{-1}\bm{\mathit{f}}||_{1,h}|dz|\\
			&\leq d_{\mu} \frac{C}{d_{\mu}|z|} ||(\bm{\mathit{P}}\bm{\mathit{T}}-\bm{\mathit{T}}_h)|_{\bm{\mathit{W}}(h)}||_{1,h}\frac{C}{d_{\mu}|z|}||\bm{\mathit{f}}||_{1,h}\leq\frac{C}{d_{\mu}}||(\bm{\mathit{P}}\bm{\mathit{T}}-\bm{\mathit{T}}_h)|_{\bm{\mathit{W}}(h)}||_{1,h}||\bm{\mathit{f}}||_{1,h},	
		\end{aligned}	
	\end{equation}
	where Lemmas \ref{0126c}, \ref{0127b} and the fact that
	for all $z\in\gamma$, $|z|\geq \mu-d_{\mu}\geq \frac{1}{2}\mu$ have been used. The analysis in \textbf{Case 3} is similar to \textbf{Case 1}. Besides, in \textbf{Case 2}, for any $\bm{\mathit{f}}\in\bm{\mathit{W}}(h)$,
	$||(\bm{\mathit{P}}\bm{\mathit{E}}_{\lambda_S}-\bm{\mathit{E}}_{\lambda_S,h})\bm{\mathit{f}}||_{1,h}
			\leq\frac{1}{2\pi}\int_{\gamma_{\infty}}||(z\bm{\mathit{I}}-\bm{\mathit{T}}_{\lambda_S,h})^{-1}(\bm{\mathit{P}}\bm{\mathit{T}}_{\lambda_S}
			-\bm{\mathit{T}}_{\lambda_S,h})(z\bm{\mathit{I}}-\bm{\mathit{T}}_{\lambda_S})^{-1}\bm{\mathit{f}}||_{1,h}|dz|\leq\frac{Cd_{\infty}}{d_1^2}||(\bm{\mathit{P}}\bm{\mathit{T}}_{\lambda_S}-\bm{\mathit{T}}_{\lambda_S,h})|_{\bm{\mathit{W}}(h)}||_{1,h}||\bm{\mathit{f}}||_{1,h}$.
\end{proof}

\begin{theorem}
	\label{0308b}
	Let $r>0$ be such that $R(\bm{\mathit{E}})\subset{\bm{\mathit{H}}}^{1+r}(\Omega)$ with the regularity assumption
	\begin{equation}
		\label{0119f}
		||\underline{\bm{\sigma}}||_{r,\Omega}+||\mathbf{u}||_{1+r,\Omega}+||\underline{\bm{\rho}}||_{r,\Omega}\leq C^{reg}||\mathbf{u}||_{0,\Omega},\hspace{1em}\forall \mathbf{u}\in R(\bm{\mathit{E}}) 
	\end{equation}
	where $C^{reg}$ is independent of $\lambda_S$ and $(\underline{\bm{\sigma}},\underline{\bm{\rho}})$ together with $(\lambda,\mathbf{u})$ is the solution to (\ref{0112a}). $R(\bm{\mathit{E}})$ is the eigenspace associated with $\mu$ in \textbf{Case 1}(i.e.(\ref{Case1})) or $\mu_{\infty}$ in \textbf{Case 3}((\ref{Case3})), or it is the direct sum of eigenspaces corresponding to $\mu_{\lambda_S}^1,\mu_{\lambda_S}^2,\cdot\cdot\cdot,\mu_{\lambda_S}^{m_{\infty}}$ contained in $\gamma_{\infty}$(i.e. (\ref{arc_2})) in \textbf{Case 2}((\ref{Case2})). Then $\exists C>0$ independent of $h$ and $\lambda_S$ such that, for $h$ small enough,
	\begin{equation} 
		\label{0116c}
		\hat{\delta}(\bm{\mathit{P}}(R(\bm{\mathit{E}})),  R(\bm{\mathit{E}}_h))\leq \frac{C}{d}h^{\rm min\{r,k+2\}},
	\end{equation}
	where $d=d_{\mu}$($d_{\mu}$ is defined in (\ref{arc_1})) for \textbf{Case 1}; $d=d_1^2/d_{\infty}$ with $d_1=dist(\gamma_{\infty},$ $sp(\bm{\mathit{T}}_{\lambda_S}))$ ($d_{\infty}$ is defined in (\ref{arc_2})) for \textbf{Case 2}; or $d=d_{\infty}$ for \textbf{Case 3}. Moreover,
	\begin{equation} 
		\label{1231b}	
		|\lambda-\lambda_{hj}|\leq Ch^{2\rm min\{r,k+1\}}\hspace{0.5em}	{\rm in\hspace{0.5em}\textbf{Case 1}},\hspace{0.5em}|\lambda_{\lambda_S}^i-\lambda_{\lambda_S,hj}^i|\leq Ch^{2\rm min\{r,k+1\}}\hspace{0.5em}	{\rm in\hspace{0.5em}\textbf{Case 2}}, 
	\end{equation}
	for all $1\leq i\leq m_{\infty}$, $1\leq j\leq m$. $m$ is the multiplicity of $\lambda$ in \textbf{Case 1} or $\lambda_{\lambda_S}^i$ in \textbf{Case 2}. And in \textbf{Case 3}, $|\lambda_{\infty}-\lambda_{\infty,hj}|\leq Ch^{2\rm min\{r,k+1\}}$, for all $1\leq j\leq m_{\infty}$.	
\end{theorem}
\begin{proof}
	For $\bm{\mathit{f}}\in R(\bm{\mathit{E}})\subset\bm{\mathit{H}}^1_{0,\Gamma_0}(\Omega)$ with $||\bm{\mathit{f}}||_{1,h}=1$, we have $||\bm{\mathit{P}}\bm{\mathit{f}}-\bm{\mathit{E}}_h\bm{\mathit{f}}||_{1,h}=||(\bm{\mathit{P}}\bm{\mathit{E}}-\bm{\mathit{E}}_h)\bm{\mathit{f}}||_{1,h}\leq||(\bm{\mathit{P}}\bm{\mathit{E}}-\bm{\mathit{E}}_h)|_{R(\bm{\mathit{E}})}||_{1,h}\leq||(\bm{\mathit{P}}\bm{\mathit{E}}-\bm{\mathit{E}}_h)|_{\bm{\mathit{W}}(h)}||_{1,h}$. Hence we have $\delta(\bm{\mathit{P}}(R(\bm{\mathit{E}})), R(\bm{\mathit{E}}_h))\leq||(\bm{\mathit{P}}\bm{\mathit{E}}-\bm{\mathit{E}}_h)|_{\bm{\mathit{W}}(h)}||_{1,h}$; For $\bm{\mathit{f}}_h\in R(\bm{\mathit{E}}_h)\subset\bm{\mathit{W}}^h$ with $||\bm{\mathit{f}}_h||_{1,h}=1$, we have $||\bm{\mathit{f}}_h-\bm{\mathit{P}}\bm{\mathit{E}}\bm{\mathit{f}}_h||_{1,h}\leq||(\bm{\mathit{P}}\bm{\mathit{E}}-\bm{\mathit{E}}_h)\bm{\mathit{f}}_h||_{1,h}\leq||(\bm{\mathit{P}}\bm{\mathit{E}}-\bm{\mathit{E}}_h)|_{\bm{\mathit{W}}^h}||_{1,h}\leq||(\bm{\mathit{P}}\bm{\mathit{E}}-\bm{\mathit{E}}_h)|_{\bm{\mathit{W}}(h)}||_{1,h}$. Then by virtue of Lemma \ref{0127c} and Theorem \ref{1221a}, we know that $||(\bm{\mathit{P}}\bm{\mathit{E}}-\bm{\mathit{E}}_h)|_{\bm{\mathit{W}}(h)}||_{1,h}$ converges to 0 as $h\to0$. Thus $\lim_{h\to0}\hat{\delta}(\bm{\mathit{P}}(R(\bm{\mathit{E}})),  R(\bm{\mathit{E}}_h))=0.$ 
	Applying \cite[Lemma 212, Lemma 213]{KT1958}, we have for $h$ small enough,
	$dim(\bm{\mathit{P}}(R(\bm{\mathit{E}})))$ $=dim(R(\bm{\mathit{E}}_h))$ and
	\begin{equation}
		\label{0127d}
		\delta(R(\bm{\mathit{E}}_h), \bm{\mathit{P}}(R(\bm{\mathit{E}}))\leq \frac{\delta( \bm{\mathit{P}}(R(\bm{\mathit{E}})),R(\bm{\mathit{E}}_h))}{1-\delta( \bm{\mathit{P}}(R(\bm{\mathit{E}})),R(\bm{\mathit{E}}_h))}
	\end{equation}
	Since $\lim_{h\to0}\hat{\delta}(\bm{\mathit{P}}(R(\bm{\mathit{E}})),  R(\bm{\mathit{E}}_h))=0,$ by (\ref{0127d}), we have $	
	\delta(R(\bm{\mathit{E}}_h), \bm{\mathit{P}}(R(\bm{\mathit{E}}))\leq C\delta( \bm{\mathit{P}}(R(\bm{\mathit{E}})),R(\bm{\mathit{E}}_h)),$
	for some constant $C$ and hence that
	\begin{equation}
		\label{0127f}	
		\hat{\delta}( \bm{\mathit{P}}(R(\bm{\mathit{E}})),R(\bm{\mathit{E}}_h))\leq (1+C)\delta( \bm{\mathit{P}}(R(\bm{\mathit{E}})),R(\bm{\mathit{E}}_h))	
	\end{equation}	
	Now for $\bm{\mathit{f}}\in R(\bm{\mathit{E}})$ with $||\bm{\mathit{f}}||_{1,h}=1$, clearly $(z\bm{\mathit{I}}-\bm{\mathit{T}})^{-1}\bm{\mathit{f}}\in R(\bm{\mathit{E}})$, similar to the proof of Lemma \ref{0127c}, we have
	\begin{equation}
		\label{0127g}	
		||\bm{\mathit{P}}\bm{\mathit{f}}-\bm{\mathit{E}}_h\bm{\mathit{f}}||_{1,h}
		=||(\bm{\mathit{P}}\bm{\mathit{E}}-\bm{\mathit{E}}_h)\bm{\mathit{f}}||_{1,h}		
		\leq\frac{C}{d}||(\bm{\mathit{P}}\bm{\mathit{T}}-\bm{\mathit{T}}_h)|_{R(\bm{\mathit{E}})}||_{1,h}		
	\end{equation}
	In terms of Theorem \ref{1221a} and the regularity assumption \ref{0119f}, we easily obtain $
	||(\bm{\mathit{P}}\bm{\mathit{T}}-\bm{\mathit{T}}_h)|_{R(\bm{\mathit{E}})}||_{1,h}\leq Ch^{\rm min\{r,k+2\}},$
	where $C$ is independent of $h$	and $\lambda_S$. Thus combining (\ref{0127f}) and (\ref{0127g}), (\ref{0116c}) is proved. By (\ref{0116c}), it's clear that  $\widetilde{\delta}(R(\bm{\mathit{E}}), R(\bm{\mathit{E}}_h))\leq Ch^{\rm min\{r,k+1\}}$, where $\widetilde{\delta}(R(\bm{\mathit{E}}), R(\bm{\mathit{E}}_h))$ is the gap between $R(\bm{\mathit{E}})$ and $R(\bm{\mathit{E}}_h)$ under $L^2$-norm. Then using similar analysis to \cite{BJJ1981}, we easily know there is a double order for eigenvalues error estimates so (\ref{1231b}) is proved.
\end{proof}

\subsection{Approximation of the eigenspaces of the stress and rotation for the eigenvalue problem}
\label{sec: Approximation of the eigenspaces of the stress and rotation}
\begin{theorem}
	\label{0308c}
	Let $\mathbf{u}$ be an eigenfunction of $\bm{\mathit{T}}$ associated with $\mu=\lambda^{-1}$ and let $\mathbf{u}_{hi}$, $i=1,2,\cdot\cdot\cdot,m$($m$ is the multiplicity of $\mu$) be the eigenfunctions corresponding to $m$ eigenvalues $\mu_{h1},\mu_{h2},\cdot\cdot\cdot,\mu_{hm}$ of $\bm{\mathit{T}}_h$, where $\mu_{hi}=\lambda_{hi}^{-1}$. Let $r>0$ be such that $R(\bm{\mathit{E}})\subset{\bm{\mathit{H}}}^{1+r}(\Omega)$ with the regularity assumption
	\begin{equation}
		\label{0119i}
		||\underline{\bm{\sigma}}||_{r,\Omega}+||\mathbf{u}||_{1+r,\Omega}+||\underline{\bm{\rho}}||_{r,\Omega}\leq C^{reg}||\mathbf{u}||_{0,\Omega},\hspace{1em}\forall \mathbf{u}\in R(\bm{\mathit{E}}) 
	\end{equation}
	where $C^{reg}$ is independent of $\lambda_S$ and $(\underline{\bm{\sigma}},\underline{\bm{\rho}})$ together with $(\lambda,\mathbf{u})$ is the solution to (\ref{0112a}). $R(\bm{\mathit{E}})$ is the eigenspace associated with $\mu$ in \textbf{Case 1}(i.e.(\ref{Case1})) or $\mu_{\infty}$ in \textbf{Case 3}((\ref{Case3})), or it is the direct sum of eigenspaces corresponding to $\mu_{\lambda_S}^1,\mu_{\lambda_S}^2,\cdot\cdot\cdot,\mu_{\lambda_S}^{m_{\infty}}$ contained in $\gamma_{\infty}$(i.e. (\ref{arc_2})) in \textbf{Case 2}((\ref{Case2})). And let $(\underline{\bm{\sigma}}_{hi}$, $\underline{\bm{\rho}}_{hi})$ together with $(\lambda_{hi},\mathbf{u}_{hi})$ be the approximation  given by (\ref{aa}). For the eigenvalue problem, we have
	\begin{equation}
		\label{stress gap}
		\widetilde{\delta}(R(\bm{\mathit{E}}_{\underline{\bm{\sigma}}}),  R(\bm{\mathit{E}}_{\underline{\bm{\sigma}}_h}))\leq Ch^{\rm min\{r,k+2\}}/d,\widetilde{\delta}(R(\bm{\mathit{E}}_{\underline{\bm{\rho}}}),  R(\bm{\mathit{E}}_{\underline{\bm{\rho}}_h}))\leq Ch^{\rm min\{r,k+2\}}/d
	\end{equation}
	where $R(\bm{\mathit{E}}_{\underline{\bm{\sigma}}})$, $R(\bm{\mathit{E}}_{\underline{\bm{\sigma}}_h})$ are the eigenspaces of the exact and numerical stresses. And
	$\widetilde{\delta}(R(\bm{\mathit{E}}_{\underline{\bm{\sigma}}}),  R(\bm{\mathit{E}}_{\underline{\bm{\sigma}}_h}))$ is the gap between $R(\bm{\mathit{E}}_{\underline{\bm{\sigma}}})$ and $R(\bm{\mathit{E}}_{\underline{\bm{\sigma}}_h})$ under $L^2$-norm. The notation $\widetilde{\delta}(R(\bm{\mathit{E}}_{\underline{\bm{\rho}}}),  R(\bm{\mathit{E}}_{\underline{\bm{\rho}}_h}))$ is defined similarly to $\widetilde{\delta}(R(\bm{\mathit{E}}_{\underline{\bm{\sigma}}}),  R(\bm{\mathit{E}}_{\underline{\bm{\sigma}}_h}))$.%$$||\underline{\bm{\sigma}}-\underline{\bm{\sigma}}_{hi}||_{0,\Omega}\leq \frac{C}{d}h^{\rm min\{r,k+2\}},||\underline{\bm{\rho}}-\underline{\bm{\rho}}_{hi}||_{0,\Omega}\leq \frac{C}{d}h^{\rm min\{r,k+2\}}$$
	%	for all $1\leq i\leq m$, where 
	$d=d_{\mu}$($d_{\mu}$ is defined in (\ref{arc_1})) for  \textbf{Case 1}; $d=d_1^2/d_{\infty}$ with $d_1=dist(\gamma_{\infty},sp(\bm{\mathit{T}}_{\lambda_S}))$ ($d_{\infty}$ is defined in (\ref{arc_2})) for \textbf{Case 2}; or $d=d_{\infty}$ for \textbf{Case 3}. Here $C$ depends on the eigenvalue $\lambda$ and the exact solution but not on the grid size $h$ and the Lam${ \rm \acute{e}}$ coefficient $\lambda_S$.	
\end{theorem}
\begin{proof}
	In terms of (\ref{0112a}) and (\ref{hh}), we can obtain that 
	\begin{subequations}
		\label{abc}
		\begin{align}		
			(\mathcal{A}(\mathbf{\Pi}_h\underline{\bm{\sigma}}),\underline{\bm{\tau}}_h)_{\Omega}+(\bm{\mathit{P}}\mathbf{u},\textbf{div}\underline{\bm{\tau}}_h)_{\Omega}+(\underline{\bm{\mathit{P}}}\underline{\bm{\rho}},\underline{\bm{\tau}}_h)_{\Omega}&=(\mathcal{A}(\mathbf{\Pi}_h\underline{\bm{\sigma}}-\underline{\bm{\sigma}}),\underline{\bm{\tau}}_h)_{\Omega} \\
			-(\rho_S^{-1}\textbf{div}(\mathbf{\Pi}_h\underline{\bm{\sigma}}),\bm{\omega}_h)_{\Omega}=-(\rho_S^{-1}\bm{\mathit{P}}\textbf{div}\underline{\bm{\sigma}},\bm{\omega}_h)_{\Omega}&=\lambda(\mathbf{u},\bm{\omega}_h)_{\Omega}=\lambda(\bm{\mathit{P}}\mathbf{u},\bm{\omega}_h)_{\Omega} \\
			(\mathbf{\Pi}_h\underline{\bm{\sigma}},\underline{\bm{\eta}}_h)_{\Omega}&=(\mathbf{\Pi}_h\underline{\bm{\sigma}}-\underline{\bm{\sigma}},\underline{\bm{\eta}}_h)_{\Omega} 
		\end{align}	
	\end{subequations}
	for all $(\underline{\bm{\tau}}_h,\bm{\omega}_h,\underline{\bm{\eta}}_h)\in\underline{\bm{\mathit{V}}}^h\times\bm{\mathit{W}}^h\times\underline{\bm{\mathit{A}}}^h$.
	Recalling (\ref{aa}), we have that 
	\begin{equation}
		\label{bcd}
		A((\underline{\bm{\sigma}}_{hi},\mathbf{u}_{hi},\underline{\bm{\rho}}_{hi}),(\underline{\bm{\tau}}_h,\bm{\omega}_h,\underline{\bm{\eta}}_h))=\lambda_{hi} B(\mathbf{u}_{hi},\bm{\omega}_h),\hspace{0.2em}\forall(\underline{\bm{\tau}}_h,\bm{\omega}_h,\underline{\bm{\eta}}_h)\in\underline{\bm{\mathit{V}}}^h\times\bm{\mathit{W}}^h\times\underline{\bm{\mathit{A}}}^h.
	\end{equation}
	Define $\mathbf{e}_{\underline{\bm{\sigma}}}=\underline{\bm{\sigma}}_
	{hi}-\mathbf{\Pi}_h\underline{\bm{\sigma}}$, $\mathbf{e}_{\mathbf{u}}=\mathbf{u}_{hi}-\bm{\mathit{P}}\mathbf{u}$, $\mathbf{e}_{\underline{\bm{\rho}}}=\underline{\bm{\rho}}_{hi}-\underline{\bm{\mathit{P}}}\underline{\bm{\rho}}$. Then by subtracting (\ref{abc}) from (\ref{bcd}), we obtain that
	\begin{equation}
		\label{def}
		\begin{aligned}		
			A((\mathbf{e}_{\underline{\bm{\sigma}}},\mathbf{e}_{\mathbf{u}},\mathbf{e}_{\underline{\bm{\rho}}}),(\underline{\bm{\tau}}_h,\bm{\omega}_h,\underline{\bm{\eta}}_h))=
			(\lambda\bm{\mathit{P}}\mathbf{u}-\lambda_{hi}\mathbf{u}_{hi},\bm{\omega}_h)_{\Omega} \\+(\mathcal{A}(\underline{\bm{\sigma}}-\mathbf{\Pi}_h\underline{\bm{\sigma}}),\underline{\bm{\tau}}_h)_{\Omega}+(\underline{\bm{\sigma}}-\mathbf{\Pi}_h\underline{\bm{\sigma}},\underline{\bm{\eta}}_h)_{\Omega}
		\end{aligned}	
	\end{equation}
	In terms of (\ref{pp}), we obtain that there exists $(\widetilde{\underline{\bm{\tau}}}_h,\widetilde{\bm{\omega}}_h,\widetilde{\underline{\bm{\eta}}}_h)\in\underline{\bm{\mathit{V}}}^h\times\bm{\mathit{W}}^h\times\underline{\bm{\mathit{A}}}^h$ such that
	$
	\frac{A((\mathbf{e}_{\underline{\bm{\sigma}}},\mathbf{e}_{\mathbf{u}},\mathbf{e}_{\underline{\bm{\rho}}}),(\widetilde{\underline{\bm{\tau}}}_h,\widetilde{\bm{\omega}}_h,\widetilde{\underline{\bm{\eta}}}_h))}{||\widetilde{\underline{\bm{\tau}}}_h||_{\underline{\bm{\mathit{H}}}(\textbf{div};\Omega)}+||\widetilde{\bm{\omega}}_h||_{0,\Omega}+||\widetilde{\underline{\bm{\eta}}}_h||_{0,\Omega}}\geq\beta(||\mathbf{e}_{\underline{\bm{\sigma}}}||_{\underline{\bm{\mathit{H}}}(\textbf{div};\Omega)}+||\mathbf{e}_{\mathbf{u}}||_{0,\Omega}+||\mathbf{e}_{\underline{\bm{\rho}}}||_{0,\Omega}),
	$
	where $\beta>0$ independent of $h$ and $\lambda_S$. Combining (\ref{def}), we can obtain 
	\begin{equation}
		\label{fgh}
		\begin{aligned}		
			&\beta(||\mathbf{e}_{\underline{\bm{\sigma}}}||_{\underline{\bm{\mathit{H}}}(\textbf{div};\Omega)}+||\mathbf{e}_{\mathbf{u}}||_{0,\Omega}+||\mathbf{e}_{\underline{\bm{\rho}}}||_{0,\Omega})\leq (||\lambda\bm{\mathit{P}}\mathbf{u}-\lambda_{hi}\mathbf{u}_{hi}||_{0,\Omega}\cdot||\widetilde{\bm{\omega}}_h||_{0,\Omega}\\
			&\quad+C||\underline{\bm{\sigma}}-\mathbf{\Pi}_h\underline{\bm{\sigma}}||_{0,\Omega}\cdot(||\widetilde{\underline{\bm{\tau}}}_h||_{0,\Omega}+||\widetilde{\underline{\bm{\eta}}}_h||_{0,\Omega}))  /(||\widetilde{\underline{\bm{\tau}}}_h||_{\underline{\bm{\mathit{H}}}(\textbf{div};\Omega)}+||\widetilde{\bm{\omega}}_h||_{0,\Omega}+||\widetilde{\underline{\bm{\eta}}}_h||_{0,\Omega})\\
		\end{aligned}	
	\end{equation}
	where Cauchy-Schwartz inequality has been used. We estimate the terms in (\ref{fgh}), 
	\begin{equation}
		\label{ghi}
		\begin{aligned}		
			&||\lambda\bm{\mathit{P}}\mathbf{u}-\lambda_{hi}\mathbf{u}_{hi}||_{0,\Omega}\leq |\lambda|||\bm{\mathit{P}}\mathbf{u}-\mathbf{u}_{hi}||_{0,\Omega}+|\lambda-\lambda_{hi}|(||\bm{\mathit{P}}\mathbf{u}-\mathbf{u}_{hi}||_{0,\Omega}+||\bm{\mathit{P}}\mathbf{u}||_{0,\Omega})\\
			&\leq |\lambda| Ch^{\rm min\{r,k+2\}}/d+Ch^{\rm min\{r,2(k+1)\}}\cdot h^{\rm min\{r,k+2\}}/d+Ch^{\rm min\{r,2(k+1)\}}||\mathbf{u}||_{0,\Omega}\\
			&\leq Ch^{\rm min\{r,k+2\}}/d,
		\end{aligned}	
	\end{equation}
	where Theorem \ref{0308b} and the triangle inequality have been used. Moreover, we have 
	\begin{equation}
		\label{hij}
		\begin{aligned}		
			||\underline{\bm{\sigma}}-\mathbf{\Pi}_h\underline{\bm{\sigma}}||_{0,\Omega}
			\leq Ch^{\rm min\{r,k+2\}}||\mathbf{u}||_{0,\Omega}
			\leq Ch^{\rm min\{r,k+2\}}.
		\end{aligned}	
	\end{equation}
	where (\ref{ii}), (\ref{jj}), and the regularity assumption (\ref{0119i}) have been used. Combining (\ref{fgh}), (\ref{ghi}), and (\ref{hij}), we can obtain that
	$\beta(||\mathbf{e}_{\underline{\bm{\sigma}}}||_{\underline{\bm{\mathit{H}}}(\textbf{div};\Omega)}+||\mathbf{e}_{\mathbf{u}}||_{0,\Omega}+||\mathbf{e}_{\underline{\bm{\rho}}}||_{0,\Omega})\leq \frac{C}{d}h^{\rm min\{r,k+2\}}.$
	Thus
	\begin{equation}
		\label{stress approximation}
		||\underline{\bm{\sigma}}_{hi}-\underline{\bm{\sigma}}||_{0,\Omega}\leq	||\mathbf{e}_{\underline{\bm{\sigma}}}||_{0,\Omega}+||\underline{\bm{\sigma}}-\mathbf{\Pi}_h\underline{\bm{\sigma}}||_{0,\Omega}\leq Ch^{\rm min\{r,k+2\}}/d,
	\end{equation}
	for all $1\leq i\leq m$, and $C$ doesn't depend on $h$ and $\lambda_S$. %and the Lam${ \rm \acute{e}}$ coefficient $\lambda_S$. 
	We recalled (\ref{0116c}) in Theorem \ref{0308b}, which approximates the $L^2$-orthogonal projection of the
	eigenspace of exact displacement in terms of the gap of subspaces. Since the stress and displacement appear in one-to-one correspondence and the error analysis above is closely dependent on (\ref{0116c}), then it's natural to express (\ref{stress approximation}) in terms of the gap of subspaces in the first inequality of (\ref{stress gap}).
	Similarly, we can also obtain the same estimate for the rotation in the second inequality of (\ref{stress gap}).
\end{proof}
\begin{remark}
	\label{apply to others}
	We point out that the spectral analyses in section \ref{sec: Spectral approximation}(i.e. Lemmas \ref{0220a}-\ref{0127c} and Theorems \ref{0308b}-\ref{0308c}) can cover all elements listed in Remark \ref{apply to other elements}.
\end{remark}
\section{Hybridization}
\label{sec: Hybridization}
In this section, we will provide hybridization of the source problem and the eigenproblem. Different from a related hybridization introduced in \cite{CJL2010} for Poisson equations, we utilize discrete $H^1$-stability of numerical displacement to prove Lemmas \ref{1217p} and \ref{1217hh}. Then these lemmas are used to prove Theorems \ref{1217j} and \ref{1217jj}. In Theorem \ref{1217jj}, we provide a good initial approximation of the eigenvalue for the nonlinear eigenproblem. The innovation can be found in Remarks \ref{1222l} and \ref{novelty remark}. For simplicity, we assume $\mathcal{A}$ a piecewise constant.
\subsection{Hybridization for the source problem}
\label{sec: Hybridization for the source problem}
We enlarge the mixed method (\ref{mm}) to include an additional unknown ${\bm{\gamma}}^{\mathit{f}}_h\in\bm{\mathit{M}}^h:=\{\bm{\mu}:\bm{\mu}|_F\in\bm{\mathit{p}}^{k+1}{(F)}$ for all  grid faces $F\in\varepsilon_h,$ and $\bm{\mu}|_{\Gamma_0}=\mathbf{0}\},$
and expand the space $\underline{\bm{\mathit{V}}}^h$ by removing $\underline{\bm{\mathit{H}}}(\textbf{div};\Omega)$-continuity constraints to obtain the space
$\widetilde{\underline{\bm{\mathit{V}}}}^h:=\{\underline{\bm{\upsilon}}\in\underline{\bm{\mathit{L}}}^2(\Omega):\underline{\bm{\upsilon}}|_K\in\underline{\bm{\mathit{V}}}(K)$ for all grid elements $K\in\mathcal{T}_h%,\hspace{0.5em}\underline{\bm{\upsilon}}\mathbf{n}=\mathbf{0}\hspace{0.5em} {\rm on\hspace{0.5em}\Gamma_1}
\}.$
%Note that $\underline{\bm{\mathit{V}}}^h=\{\underline{\bm{\upsilon}}\in\widetilde{\underline{\bm{\mathit{V}}}}^h:\left \langle \underline{\bm{\upsilon}}\mathbf{n},\bm{\mu} \right \rangle_{\partial\mathcal{T}_h}=0\hspace{0.5em}{\rm for}\hspace{0.5em}{\rm all}\hspace{0.5em} \bm{\mu}\in\bm{\mathit{M}}^h \}$.
The solution of the hybridized method, $(\underline{\bm{\sigma}}^{\mathit{f}}_h,\mathbf{u}^{\mathit{f}}_h, \underline{\bm{\rho}}^{\mathit{f}}_h,{\bm{\gamma}}^{\mathit{f}}_h)\in\widetilde{\underline{\bm{\mathit{V}}}}^h\times\bm{\mathit{W}}^h\times\underline{\bm{\mathit{A}}}^h\times\bm{\mathit{M}}^h$, satisfies
\begin{subequations}
	\label{1216a}
	\begin{align}		
		(\mathcal{A}\underline{\bm{\sigma}}^{\mathit{f}}_h,\underline{\bm{\upsilon}})_{\mathcal{T}_h}+(\mathbf{u}^{\mathit{f}}_h,\textbf{div}\underline{\bm{\upsilon}})_{\mathcal{T}_h}+(\underline{\bm{\rho}}^{\mathit{f}}_h,\underline{\bm{\upsilon}})_{\mathcal{T}_h}+\left \langle {\bm{\gamma}}^{\mathit{f}}_h, \underline{\bm{\upsilon}}\mathbf{n}\right \rangle_{\partial\mathcal{T}_h}&=0 \\
		-(\rho_S^{-1}\textbf{div}\underline{\bm{\sigma}}^{\mathit{f}}_h,\bm{\omega})_{\mathcal{T}_h}&=(\bm{\mathit{f}},\bm{\omega})_{\mathcal{T}_h} \\
		(\underline{\bm{\sigma}}^{\mathit{f}}_h,\underline{\bm{\eta}})_{\mathcal{T}_h}&=0  \\
		\left \langle \underline{\bm{\sigma}}^{\mathit{f}}_h\mathbf{n}, {\bm{\mu}}\right \rangle_{\partial\mathcal{T}_h}&=0 
	\end{align}	
\end{subequations}
for all $(\underline{\bm{\upsilon}},\bm{\omega},\underline{\bm{\eta}},\bm{\mu})\in\widetilde{\underline{\bm{\mathit{V}}}}^h\times\bm{\mathit{W}}^h\times\underline{\bm{\mathit{A}}}^h\times\bm{\mathit{M}}^h$. (\ref{1216a}) yields a ``reduced'' system. To state it, we need more notation. Define $\mathbf{A}:\widetilde{\underline{\bm{\mathit{V}}}}^h\mapsto\widetilde{\underline{\bm{\mathit{V}}}}^h$, $\mathbf{B}:\widetilde{\underline{\bm{\mathit{V}}}}^h\mapsto\bm{\mathit{W}}^h$, $\mathbf{C}:\widetilde{\underline{\bm{\mathit{V}}}}^h\mapsto\underline{\bm{\mathit{A}}}^h$, $\mathbf{D}:\widetilde{\underline{\bm{\mathit{V}}}}^h\mapsto\bm{\mathit{M}}^h$ by
$(\mathbf{A}\underline{\bm{\sigma}},\underline{\bm{\upsilon}})_{\Omega}=(\mathcal{A}\underline{\bm{\sigma}},\underline{\bm{\upsilon}})_{\mathcal{T}_h}, (\mathbf{B}\underline{\bm{\sigma}},\mathbf{u})_{\Omega}=(\mathbf{u},\nabla\cdot\underline{\bm{\sigma}})_{\mathcal{T}_h},(\mathbf{C}\underline{\bm{\sigma}},\underline{\bm{\rho}})_{\Omega}=(\underline{\bm{\sigma}},\underline{\bm{\rho}})_{\mathcal{T}_h},	\left \langle \mathbf{D}\underline{\bm{\sigma}}, {\bm{\mu}}\right \rangle_{\partial\mathcal{T}_h}=\left \langle {\bm{\mu}}, \underline{\bm{\sigma}}\mathbf{n}\right \rangle_{\partial\mathcal{T}_h}$
for all $\underline{\bm{\sigma}},\underline{\bm{\upsilon}}\in\widetilde{\underline{\bm{\mathit{V}}}}^h$, $\mathbf{u}\in\bm{\mathit{W}}^h$, $\underline{\bm{\rho}}\in\underline{\bm{\mathit{A}}}^h$, and $\bm{\mu}\in\bm{\mathit{M}}^h$. Moreover, we need local solution operators $\bm{\mathcal{Q}}_1:\bm{\mathit{M}}^h\mapsto\widetilde{\underline{\bm{\mathit{V}}}}^h$, $\bm{\mathcal{Q}}_2:\bm{\mathit{M}}^h\mapsto\bm{\mathit{W}}^h$, $\bm{\mathcal{Q}}_3:\bm{\mathit{M}}^h\mapsto\underline{\bm{\mathit{A}}}^h$, $\bm{\mathcal{Q}}_1^L:\underline{\bm{\mathit{L}}}^2(\Omega)\mapsto\widetilde{\underline{\bm{\mathit{V}}}}^h$, $\bm{\mathcal{Q}}_2^L:\bm{\mathit{L}}^2(\Omega)\mapsto\bm{\mathit{W}}^h$, $\bm{\mathcal{Q}}_3^L:\underline{\bm{\mathit{L}}}^2(\Omega)\mapsto\underline{\bm{\mathit{A}}}^h$. These operators
are defined using the solution of the following systems:
\begin{equation}
	\label{1217b} 
	\begin{pmatrix}
		\, \mathbf{A} & \mathbf{B}^t & \mathbf{C}^t \,\  \\
		\, \mathbf{B} & \mathbf{0} & \mathbf{0} \,\  \\
		\, \mathbf{C} & \mathbf{0} & \mathbf{0} 
	\end{pmatrix}
	\begin{pmatrix}
		\,\bm{\mathcal{Q}}_1\bm{\mu} \,\  \\
		\,\bm{\mathcal{Q}}_2\bm{\mu} \,\  \\
		\,\bm{\mathcal{Q}}_3\bm{\mu} \
	\end{pmatrix}
	=\begin{pmatrix}
		\,-\bm{\mathbf{D}}\bm{\mu} \,\  \\
		\,\mathbf{0} \,\  \\
		\,\mathbf{0} \
	\end{pmatrix}, 
	\begin{pmatrix}
		\, \mathbf{A} & \mathbf{B}^t & \mathbf{C}^t \,\  \\
		\, \rho_S^{-1}\mathbf{B} & \mathbf{0} & \mathbf{0} \,\  \\
		\, \mathbf{C} & \mathbf{0} & \mathbf{0} 
	\end{pmatrix}
	\begin{pmatrix}
		\,\bm{\mathcal{Q}}_1^L\bm{\mathit{f}} \,\  \\
		\,\bm{\mathcal{Q}}_2^L\bm{\mathit{f}} \,\  \\
		\,\bm{\mathcal{Q}}_3^L\bm{\mathit{f}} \
	\end{pmatrix}
	=\begin{pmatrix}
		\,\mathbf{0} \,\  \\
		\,\bm{\mathit{g}} \,\  \\
		\,\mathbf{0} \
	\end{pmatrix}
\end{equation}
where $\bm{\mathit{g}}=-\bm{\mathit{P}}\bm{\mathit{f}}$, for any $\bm{\mu}\in\bm{\mathit{M}}^h$ and $\bm{\mathit{f}}\in\bm{\mathit{L}}^2(\Omega)$. Then we have the following theorem for the source problem.
\begin{theorem}
	\label{1217d}
	The function $(\underline{\bm{\sigma}}^{\mathit{f}}_h,\mathbf{u}^{\mathit{f}}_h,\underline{\bm{\rho}}^{\mathit{f}}_h,{\bm{\gamma}}^{\mathit{f}}_h)\in\widetilde{\underline{\bm{\mathit{V}}}}^h\times\bm{\mathit{W}}^h\times\underline{\bm{\mathit{A}}}^h\times\bm{\mathit{M}}^h$ satisfy (\ref{1216a}) if and only if  ${\bm{\gamma}}^{\mathit{f}}_h$ is the unique solution of $a_h({\bm{\gamma}}^{\mathit{f}}_h,\bm{\mu})=b_h(\bm{\mu})\hspace{0.5em}\forall \bm{\mu}\in\bm{\mathit{M}}^h$, where $a_h(\bm{\mu}_1,\bm{\mu}_2)=(\mathcal{A}\bm{\mathcal{Q}}_1\bm{\mu}_1,\bm{\mathcal{Q}}_1\bm{\mu}_2)_{\mathcal{T}_h}$, $b_h(\bm{\mu})=(\bm{\mathit{f}},\bm{\mathcal{Q}}_2\bm{\mu})_{\mathcal{T}_h}$. And
	$	
	\underline{\bm{\sigma}}^{\mathit{f}}_h=\bm{\mathcal{Q}}_1{\bm{\gamma}}^{\mathit{f}}_h+\bm{\mathcal{Q}}_1^L\bm{\mathit{f}},			
	\mathbf{u}^{\mathit{f}}_h=\bm{\mathcal{Q}}_2{\bm{\gamma}}^{\mathit{f}}_h+\bm{\mathcal{Q}}_2^L\bm{\mathit{f}},
	\underline{\bm{\rho}}^{\mathit{f}}_h=\bm{\mathcal{Q}}_3{\bm{\gamma}}^{\mathit{f}}_h+\bm{\mathcal{Q}}_3^L\bm{\mathit{f}}.
	$	
\end{theorem}
\begin{proof}
	It is enough to follow the steps of \cite[Theorem 3.3]{CJJ2010}.
\end{proof}

\subsection{Hybridization for the eigenproblem}
\label{sec: Hybridization for the eigenproblem}In previous subsection, the source problem can be recovered by a reduced system,
given by Theorem \ref{1217d}. It is natural to ask such a technique for the eigenvalue problem: Find $\widetilde{\lambda}_h\in\mathbb{R}$ and $0\not\equiv\widetilde{\bm{\gamma}}_h\in\bm{\mathit{M}}^h$ satisfying
\begin{equation}
	\label{1217o}
	a_h(\widetilde{\bm{\gamma}}_h,\bm{\mu})=\widetilde{\lambda}_h(\bm{\mathcal{Q}}_2\widetilde{\bm{\gamma}}_h,\bm{\mathcal{Q}}_2\bm{\mu})\hspace{2em}\forall\bm{\mu}\in\bm{\mathit{M}}^h.	
\end{equation}
\subsubsection{\textbf{Reduction to a nonlinear eigenvalue problem}}
\label{sec: Reduction to a nonlinear eigenvalue problem}
Recall that (\ref{1216a}) represents the source problem, we turn to the interest for eigenproblem. This is to determine a nontrivial $(\underline{\bm{\sigma}}_h,\mathbf{u}_h, \underline{\bm{\rho}}_h,{\bm{\gamma}}_h)\in\widetilde{\underline{\bm{\mathit{V}}}}^h\times\bm{\mathit{W}}^h\times\underline{\bm{\mathit{A}}}^h\times\bm{\mathit{M}}^h$ and a number $\lambda_h\in\mathbb{R}$ satisfying
\begin{subequations}
	\label{1217a}
	\begin{align}		
		(\mathcal{A}\underline{\bm{\sigma}}_h,\underline{\bm{\upsilon}})_{\mathcal{T}_h}+(\mathbf{u}_h,\textbf{div}\underline{\bm{\upsilon}})_{\mathcal{T}_h}+(\underline{\bm{\rho}}_h,\underline{\bm{\upsilon}})_{\mathcal{T}_h}+\left \langle {\bm{\gamma}}_h, \underline{\bm{\upsilon}}\mathbf{n}\right \rangle_{\partial\mathcal{T}_h}&=0 \\
		-(\rho_S^{-1}\textbf{div}\underline{\bm{\sigma}}_h,\bm{\omega})_{\mathcal{T}_h}&=\lambda_h(\bm{\mathit{u}}_h,\bm{\omega})_{\mathcal{T}_h} \\
		(\underline{\bm{\sigma}}_h,\underline{\bm{\eta}})_{\mathcal{T}_h}&=0  \\
		\left \langle \underline{\bm{\sigma}}_h\mathbf{n}, {\bm{\mu}}\right \rangle_{\partial\mathcal{T}_h}&=0 
	\end{align}	
\end{subequations}
We give the reduction to the nonlinear eigenvalue problem as the following theorem.
\begin{theorem}
	\label{1217j} 
	There exists a constant $C^{\star}$, independent of h and $\mathcal{A}^{-1}$ such that any number 
	\begin{equation}
		\label{1217k} 
		\lambda_h\leq C^{\star}h^{-2}	
	\end{equation}
	satisfies
	\begin{equation}
		\label{1217l}  
		a_h({\bm{\gamma}}_h,\bm{\mu})=\lambda_h((\bm{\mathit{I}}-\lambda_h\bm{\mathcal{Q}}_2^L)^{-1}\bm{\mathcal{Q}}_2{\bm{\gamma}}_h,\bm{\mathcal{Q}}_2\bm{\mu}),\hspace{2em}\forall\bm{\mu}\in\bm{\mathit{M}}^h,
	\end{equation}
	with some nontrivial ${\bm{\gamma}}_h$ in $\bm{\mathit{M}}^h$ if and only if the number $\lambda_h$ and the functions 
	\begin{equation}
		\label{1217m}
		{\bm{\gamma}}_h,\hspace{0.5em}\mathbf{u}_h=(\bm{\mathit{I}}-\lambda_h\bm{\mathcal{Q}}_2^L)^{-1}\bm{\mathcal{Q}}_2{\bm{\gamma}}_h,\hspace{0.5em}\underline{\bm{\sigma}}_h=\bm{\mathcal{Q}}_1{\bm{\gamma}}_h+\lambda_h\bm{\mathcal{Q}}_1^L\mathbf{u}_h\hspace{0.5em}{\rm and}\hspace{0.5em}\underline{\bm{\rho}}_h=\bm{\mathcal{Q}}_3{\bm{\gamma}}_h+\lambda_h\bm{\mathcal{Q}}_3^L\mathbf{u}_h
	\end{equation}
	together satisfy (\ref{1217a}). Above, the operator $\bm{\mathit{I}}$ denotes the identity on $\bm{\mathit{W}}^h$ and the inverse in (\ref{1217l}) exists whenever (\ref{1217k}) holds.
\end{theorem}
Next, we provide some estimates for the local solution operators $\bm{\mathcal{Q}}_2^L$, $\bm{\mathcal{Q}}_2$, which is important to the proof of Theorems \ref{1217j} and \ref{1217jj}. %For any polynomial $\bm{\mu}$, we denote $||\bm{\mu}||_{L^2(\partial K)}:=||\bm{\mu}||_{0,\partial K}$.
\begin{lemma}
	\label{1217p}
	Let K be any grid element and let $\bm{\mathit{f}},\bm{\mathit{g}}\in\bm{\mathit{L}}^2(K)$. Then 
	\begin{equation}
		\label{1217q}
		(\bm{\mathcal{Q}}_2^L\bm{\mathit{f}},\bm{\mathit{g}})_K=(\rho_S^{-1}\mathcal{A}\bm{\mathcal{Q}}_1^L\bm{\mathit{f}},\bm{\mathcal{Q}}_1^L\bm{\mathit{g}})_K,
	\end{equation}
	\begin{equation}
		\label{1217r}
		||\bm{\mathcal{Q}}_2^L\bm{\mathit{f}}||_{0,K}\leq C_1^Kh_K^2||\bm{\mathit{f}}||_{0,K},
	\end{equation}
	where $C_1^K$ denotes a constant independent of $h_K$ and $\mathcal{A}^{-1}$ in $K$.
\end{lemma}
\begin{proof}
	Recall from (\ref{1217b}) that the local solution operators $\bm{\mathcal{Q}}_1^L\bm{\mathit{f}}$, $\bm{\mathcal{Q}}_2^L\bm{\mathit{f}}$ and $\bm{\mathcal{Q}}_3^L\bm{\mathit{f}}$ satisfy
	\begin{subequations}
		\label{1217t}
		\begin{align}		
			(\mathcal{A}\bm{\mathcal{Q}}_1^L\bm{\mathit{f}},\underline{\bm{\upsilon}})_K+(\bm{\mathcal{Q}}_2^L\bm{\mathit{f}},\textbf{div}\underline{\bm{\upsilon}})_K+(\bm{\mathcal{Q}}_3^L\bm{\mathit{f}},\underline{\bm{\upsilon}})_K&=0 \label{1217t_a}\\
			-(\rho_S^{-1}\textbf{div}\bm{\mathcal{Q}}_1^L\bm{\mathit{f}},\bm{\omega})_K&=(\bm{\mathit{f}},\bm{\omega})_K \\
			(\bm{\mathcal{Q}}_1^L\bm{\mathit{f}},\underline{\bm{\eta}})_K&=0  
		\end{align}	
	\end{subequations}
	for all $(\underline{\bm{\upsilon}},\bm{\omega},\underline{\bm{\eta}})\in\widetilde{\underline{\bm{\mathit{V}}}}^h\times\bm{\mathit{W}}^h\times\underline{\bm{\mathit{A}}}^h$. The proof of (\ref{1217q}) follows immediately from (\ref{1217t}):			
	$	(\rho_S\bm{\mathit{f}},\bm{\mathcal{Q}}_2^L\bm{\mathit{g}})_K=-(\textbf{div}\bm{\mathcal{Q}}_1^L\bm{\mathit{f}},\bm{\mathcal{Q}}_2^L\bm{\mathit{g}})_K=(\mathcal{A}\bm{\mathcal{Q}}_1^L\bm{\mathit{g}},\bm{\mathcal{Q}}_1^L\bm{\mathit{f}})_K+(\bm{\mathcal{Q}}_3^L\bm{\mathit{g}},\bm{\mathcal{Q}}_1^L\bm{\mathit{g}})_K=(\mathcal{A}\bm{\mathcal{Q}}_1^L\bm{\mathit{g}},\bm{\mathcal{Q}}_1^L\bm{\mathit{f}})_K
	$
	To prove (\ref{1217r}), first note that $\mathbf{B}^t\bm{\mathcal{Q}}_2^L\bm{\mathit{f}}=-\mathbf{A}\bm{\mathcal{Q}}_1^L\bm{\mathit{f}}-\mathbf{C}^t\bm{\mathcal{Q}}_3^L\bm{\mathit{f}}$. Going through the same steps as \cite[Lemma 3.1]{CJL2010}, we have
	\begin{equation}
		\label{1217v}
		||\bm{\mathcal{Q}}_2^L\bm{\mathit{f}}||_{0,K}^2\leq\frac{4}{9}h_K^2||\mathbf{B}^t\bm{\mathcal{Q}}_2^L\bm{\mathit{f}}||_{0,K}^2\leq\frac{4}{9}h_K^2(||\mathbf{A}\bm{\mathcal{Q}}_1^L\bm{\mathit{f}}||_{0,K}^2+||\mathbf{C}^t\bm{\mathcal{Q}}_3^L\bm{\mathit{f}}||_{0,K}^2)	
	\end{equation}
	By (\ref{1217q}), (\ref{1217t}) and the boundedness of $\mathcal{A}$, we have
	\begin{equation}
		\label{1220b}
		\begin{aligned}
			||\mathbf{A}\bm{\mathcal{Q}}_1^L\bm{\mathit{f}}||_{0,K}^2\leq %C(\frac{1}{2\mu_S}((\bm{\mathcal{Q}}_1^L\bm{\mathit{f}})^D,(\bm{\mathcal{Q}}_1^L\bm{\mathit{f}})^D)_K+\frac{1}{n(n\lambda_S+2\mu_S)}\int_K|tr(\bm{\mathcal{Q}}_1^L\bm{\mathit{f}})|^2)\\
			C(\mathcal{A}\bm{\mathcal{Q}}_1^L\bm{\mathit{f}},\bm{\mathcal{Q}}_1^L\bm{\mathit{f}})_K  =C(\rho_S\bm{\mathit{f}},\bm{\mathcal{Q}}_2^L\bm{\mathit{f}})_K\leq C||\bm{\mathcal{Q}}_2^L\bm{\mathit{f}}||_{0,K}\cdot||\bm{\mathit{f}}||_{0,K},
		\end{aligned}	
	\end{equation}
	where $C$ is independent of $\mathcal{A}^{-1}$ and $h$. Next, we estimate $||\mathbf{C}^t\bm{\mathcal{Q}}_3^L\bm{\mathit{f}}||_{0,K}$. We denote the local discrete $H^1$ norm of $\bm{\mathcal{Q}}_2^L\bm{\mathit{f}}$ on the element $K$ as
	$||\bm{\mathcal{Q}}_2^L\bm{\mathit{f}}||_{1,h}^2:=||\nabla\bm{\mathcal{Q}}_2^L\bm{\mathit{f}}||_{0,K}^2+\frac{1}{h_K}||\bm{\mathcal{Q}}_2^L\bm{\mathit{f}}||_{0,\partial K}^2.$
	Then by going through the similar proof to Theorem \ref{1220gg}, we easily obtain 
	\begin{equation}
		\label{1217x}
		||\bm{\mathcal{Q}}_2^L\bm{\mathit{f}}||_{1,h}\leq C||\mathcal{A}\bm{\mathcal{Q}}_1^L\bm{\mathit{f}}||_{0,K},	
	\end{equation}
	where $C$ is a constant independent of $\mathcal{A}^{-1}$ and $h$.
	Hence, we have
	\begin{equation}
		\label{1217y}
		\begin{aligned}
			&||\mathbf{C}^t\bm{\mathcal{Q}}_3^L\bm{\mathit{f}}||_{0,K}^2=(\bm{\mathcal{Q}}_3^L\bm{\mathit{f}},\bm{\mathcal{Q}}_3^L\bm{\mathit{f}})_K=(\nabla\bm{\mathcal{Q}}_2^L\bm{\mathit{f}},\bm{\mathcal{Q}}_3^L\bm{\mathit{f}})_K-\left \langle (\bm{\mathcal{Q}}_3^L\bm{\mathit{f}})\mathbf{n}, \bm{\mathcal{Q}}_2^L\bm{\mathit{f}}\right \rangle_{\partial K}\\
			&
			\leq||\bm{\mathcal{Q}}_3^L\bm{\mathit{f}}||_{0,K}(||\nabla\bm{\mathcal{Q}}_2^L\bm{\mathit{f}}||_{0,K}+C h_K^{-\frac{1}{2}}||\bm{\mathcal{Q}}_2^L\bm{\mathit{f}}||_{0,\partial K})\leq
			C||\bm{\mathcal{Q}}_3^L\bm{\mathit{f}}||_{0,K}||\mathcal{A}\bm{\mathcal{Q}}_1^L\bm{\mathit{f}}||_{0,K}
		\end{aligned}	
	\end{equation}
	where (\ref{1217t_a}), Cauchy-Schwartz, discrete trace inequality (for instance, see \cite{DA2012}) and (\ref{1217x}) have been used. Then in terms of (\ref{1217y}) and (\ref{1220b}) we can obtain
	\begin{equation}
		\label{1217z}
		\begin{aligned}
			||\mathbf{C}^t\bm{\mathcal{Q}}_3^L\bm{\mathit{f}}||_{0,K}^2=||\bm{\mathcal{Q}}_3^L\bm{\mathit{f}}||_{0,K}^2&\leq C||\mathcal{A}\bm{\mathcal{Q}}_1^L\bm{\mathit{f}}||_{0,K}^2\leq C||\bm{\mathcal{Q}}_2^L\bm{\mathit{f}}||_{0,K}\cdot||\bm{\mathit{f}}||_{0,K}.
		\end{aligned}	
	\end{equation}
	where $C$ does not depend on $\mathcal{A}^{-1}$ and $h$. Then in terms of (\ref{1220b}), (\ref{1217z}) and (\ref{1217v}), we have
	$||\bm{\mathcal{Q}}_2^L\bm{\mathit{f}}||_{0,K}\leq \frac{4}{9}Ch_K^2||\bm{\mathit{f}}||_{0,K}.$
	Let $C_1^K=\frac{4}{9}C$ from above inequality, then we obtain (\ref{1217r}). 
\end{proof}
\begin{remark}
	\label{1222l}
	Note that in Lemma \ref{1217p}, by our analysis(mainly (\ref{1217x}), (\ref{1217y})), we obtain the control relationship between $\bm{\mathcal{Q}}_3^L\bm{\mathit{f}}$ and $\bm{\mathcal{Q}}_1^L\bm{\mathit{f}}$. This result closely depends on the discrete $H^1$-stability of numerical displacement. Utilizing this relationship, we have $||\mathbf{C}^t\bm{\mathcal{Q}}_3^L\bm{\mathit{f}}||_{0,K}^2\leq C||\mathcal{A}\bm{\mathcal{Q}}_1^L\bm{\mathit{f}}||_{0,K}^2\leq C||\bm{\mathcal{Q}}_2^L\bm{\mathit{f}}||_{0,K}\cdot||\bm{\mathit{f}}||_{0,K}.$ Hence we can eliminate a $||\bm{\mathcal{Q}}_2^L\bm{\mathit{f}}||_{0,K}$	in (\ref{1217v}) to obtain an $O(h^2)$ estimate, i.e.(\ref{1217r}). If we use the analysis in \cite{CJL2010}, we can not obtain the approximation of the rotation because the rotation will be eliminated during the energy argument; We can only obtain an $O(h)$ estimate. Indeed, by traditional inf-sup condition, we have $||\bm{\mathcal{Q}}_3^L\bm{\mathit{f}}||_{0,K}\leq C||\bm{\mathit{f}}||_{0,K}$, by which we can only deduce $||\bm{\mathcal{Q}}_2^L\bm{\mathit{f}}||_{0,K}\leq Ch_K||\bm{\mathit{f}}||_{0,K}$.
\end{remark}
\begin{proof}[proof of Theorem \ref{1217j}]
	Using Lemma \ref{1217p}, it's enough to follow \cite[Theorem 3.1]{CJL2010}. 
\end{proof}

\subsubsection{\textbf{The perturbed eigenvalue problem}}
\label{sec: The perturbed eigenvalue problem} This subsection is devoted to comparing
the mixed eigenvalues $\lambda_h$ with the eigenvalues $\widetilde{\lambda}_h$ of (\ref{1217o}). We will show that the easily computable $\widetilde{\lambda}_h$ can be used as initial guesses in various algorithms to
compute $\lambda_h$.
\begin{theorem}
	\label{1217jj}
	Suppose $\lambda_h$ is an eigenvalue of (\ref{1217a}) satisfying (\ref{1217k}). Then there is an $h_0>0$ and a constant C (independent of $h$ and $\mathcal{A}^{-1}$) such that for all $h<h_0$, there is an eigenvalue $\widetilde{\lambda}_h$ of (\ref{1217o}) satisfying
	$|\lambda_h-\widetilde{\lambda}_h|\leq C\lambda_h^2\widetilde{\lambda}_hh^2.$
\end{theorem}
Before Theorem \ref{1217jj}, we give an important lemma below. 
\begin{lemma}
	\label{1217hh}
	Let $\bm{\mu}\in\bm{\mathit{M}}^h$. There exists a constant $C_2$ independent of $h$ and $\mathcal{A}^{-1}$ such that $||\bm{\mathcal{Q}}_2\bm{\mu}||_{0,\Omega}^2\leq C_2\sum_{K\in\mathcal{T}_h}h_K||\bm{\mu}||_{0,\partial K}^2$.
\end{lemma}
\begin{proof}
	Recall from (\ref{1217b}) that the local solution operators $\bm{\mathcal{Q}}_1\bm{\mathit{f}}$, $\bm{\mathcal{Q}}_2\bm{\mathit{f}}$, and $\bm{\mathcal{Q}}_3\bm{\mathit{f}}$ satisfy	
	\begin{subequations}
		\label{1217kk}
		\begin{align}		
			(\mathcal{A}\bm{\mathcal{Q}}_1\bm{\mu},\underline{\bm{\upsilon}})_K+(\bm{\mathcal{Q}}_2\bm{\mu},\textbf{div}\underline{\bm{\upsilon}})_K+(\bm{\mathcal{Q}}_3\bm{\mu},\underline{\bm{\upsilon}})_K&=-\left \langle \bm{\mu}, \underline{\bm{\upsilon}}\mathbf{n}\right \rangle_{\partial K} \\
			(\rho_S^{-1}\textbf{div}(\bm{\mathcal{Q}}_1\bm{\mu}),\bm{\omega})_K&=0 \\
			(\bm{\mathcal{Q}}_1\bm{\mu},\underline{\bm{\eta}})_K&=0
		\end{align}	
	\end{subequations}
	for all $(\underline{\bm{\upsilon}},\bm{\omega},\underline{\bm{\eta}})\in\widetilde{\underline{\bm{\mathit{V}}}}^h\times\bm{\mathit{W}}^h\times\underline{\bm{\mathit{A}}}^h$. Clearly, we also have $\mathbf{A}\bm{\mathcal{Q}}_1\bm{\mu}+\mathbf{B}^t\bm{\mathcal{Q}}_2\bm{\mu}+\mathbf{C}^t\bm{\mathcal{Q}}_3\bm{\mu}=-\mathbf{D}^t\bm{\mu}.$
	Going through the same steps of \cite[Lemma 3.1]{CJL2010}, we have
	\begin{equation}
		\label{1217ll}
		||\bm{\mathcal{Q}}_2\bm{\mu}||_{0,K}\leq\frac{2}{3}h_K||\mathbf{B}^t\bm{\mathcal{Q}}_2\bm{\mu}||_{0,K}\leq\frac{2}{3}h_K(||\mathbf{A}\bm{\mathcal{Q}}_1\bm{\mu}||_{0,K}+||\mathbf{C}^t\bm{\mathcal{Q}}_3\bm{\mu}||_{0,K}+||\mathbf{D}^t\bm{\mu}||_{0,K})	
	\end{equation}
	Firstly, we estimate $||\mathbf{A}\bm{\mathcal{Q}}_1\bm{\mu}||_{0,K}$. Recall that we have assumed $\mathcal{A}$ a constant on $K$, then we choose the test function $\underline{\bm{\upsilon}}=\mathcal{A}\bm{\mathcal{Q}}_1\bm{\mu},\bm{\omega}=\bm{\mathcal{Q}}_2\bm{\mu},\underline{\bm{\eta}}=\bm{\mathcal{Q}}_3\bm{\mu}$, by (\ref{1217kk}), we have
	$||\mathbf{A}\bm{\mathcal{Q}}_1\bm{\mu}||_{0,K}^2=(\mathcal{A}\bm{\mathcal{Q}}_1\bm{\mu},\mathcal{A}\bm{\mathcal{Q}}_1\bm{\mu})_K=-\left \langle \bm{\mu}, (\mathcal{A}\bm{\mathcal{Q}}_1\bm{\mu})\mathbf{n}\right \rangle_{\partial K}\leq Ch_K^{-\frac{1}{2}}||\mathcal{A}\bm{\mathcal{Q}}_1\bm{\mu}||_{0,K}\cdot||\bm{\mu}||_{0,\partial K}$.
	Therefore, there exists a constant $C$ independent of $\mathcal{A}^{-1}$ and $h$ such that
	\begin{equation}
		\label{1230b}
		||\mathbf{A}\bm{\mathcal{Q}}_1\bm{\mu}||_{0,K}\leq Ch_K^{-\frac{1}{2}}||\bm{\mu}||_{0,\partial K}.
	\end{equation}
Denote the local discrete $H^1$ norm of $\bm{\mathcal{Q}}_2\bm{\mu}$ on the element $K$ as
$||\bm{\mathcal{Q}}_2\bm{\mu}||_{1,h}^2:=||\nabla\bm{\mathcal{Q}}_2\bm{\mu}||_{0,K}^2+\frac{1}{h_K}||\bm{\mathcal{Q}}_2\bm{\mu}||_{0,\partial K}^2.$ Using the similar analysis to (\ref{1217x})-(\ref{1217z})(mainly by Theorem \ref{1220gg}), we have
\begin{equation}
	\label{1227o}
	||\mathbf{C}^t\bm{\mathcal{Q}}_3\bm{\mu}||_{0,K}=||\bm{\mathcal{Q}}_3\bm{\mu}||_{0,K}\leq Ch_K^{-\frac{1}{2}}||\bm{\mu}||_{\partial K},	
\end{equation} 
where $C$ does not depend on $\mathcal{A}^{-1}$ and $h$.
For $||\mathbf{D}^t\bm{\mu}||_{0,K}$, by the definition of $\mathbf{D}$, we have
$||\mathbf{D}^t\bm{\mu}||_{0,K}^2=\left\langle\bm{\mu},(\mathbf{D}^t\bm{\mu})\mathbf{n}\right\rangle_{\partial K}\leq Ch_K^{-\frac{1}{2}}||\mathbf{D}^t\bm{\mu}||_{0,K}||\bm{\mu}||_{0,\partial K}$.
Then combining (\ref{1230b}), (\ref{1227o}) and (\ref{1217ll}), we have
$||\bm{\mathcal{Q}}_2\bm{\mu}||_{0,K}\leq\frac{2}{3}Ch_K^{\frac{1}{2}}||\bm{\mu}||_{0,\partial K}$.
Squaring and summing the estimate over all elements, we obtain the desired assertion.
\end{proof}

\begin{proof}[proof of Theorem \ref{1217jj}]
Here we only emphasize that we mainly utilize Lemmas \ref{1217p} and \ref{1217hh}. Then it is enough to follow the steps of the proof of \cite[Theorem 3.2]{CJL2010}.
\end{proof}

\begin{remark}
\label{novelty remark}
Although some descriptions of this section(such as the statements of Theorems \ref{1217d}, \ref{1217j}, \ref{1217jj}) are very similar to those in \cite{CJL2010}, the main proofs(i.e. the proofs of Lemmas \ref{1217p}, \ref{1217hh} and Theorems \ref{1217j}, \ref{1217jj}) and ideas are quite different from \cite{CJL2010}. In detail, during the  proofs of Lemmas \ref{1217p} and \ref{1217hh}, inequalities (\ref{1217x})-(\ref{1217z}) and (\ref{1227o}) closely depend on Theorem \ref{1220gg}(i.e. discrete $H^1$-stability of numerical displacement). Then lemmas \ref{1217p} and \ref{1217hh} are used to prove Theorems \ref{1217j} and \ref{1217jj} to obtain an $O(h^2)$ initial approximation of the eigenvalue for the nonlinear eigenproblem.
\end{remark}
\section{Postprocessing}
\label{sec: Postprocessing}
In this section, we will discuss an accuracy-enhancing postprocessing technique for the source problem in Theorem \ref{1222b} and the eigenproblem in Theorem \ref{1222c}.
Let us first define the local postprocessing operator $\mathcal{L}(\underline{\bm{\sigma}}^{\mathit{f}}_h,\mathbf{u}^{\mathit{f}}_h, \underline{\bm{\rho}}^{\mathit{f}}_h)$. Given a pair of functions $\underline{\bm{\sigma}}^{\mathit{f}}_h,\mathbf{u}^{\mathit{f}}_h, \underline{\bm{\rho}}^{\mathit{f}}_h$, the operator gives a function $\mathbf{u}_h^{\star}\equiv\mathcal{L}(\underline{\bm{\sigma}}^{\mathit{f}}_h,\mathbf{u}^{\mathit{f}}_h, \underline{\bm{\rho}}^{\mathit{f}}_h)$ in $\bm{\mathit{p}}^{k+2}(K)$ defined element by element as follows:
\begin{subequations}
	\label{1222a}
	\begin{align}		
		(\nabla\mathbf{u}_h^{\star},\nabla\mathbf{w})_K&=(\mathcal{A}\underline{\bm{\sigma}}^{\mathit{f}}_h+\underline{\bm{\rho}}^{\mathit{f}}_h,\nabla\mathbf{w})_K,\hspace{2em} \forall\mathbf{w}\in\bm{\mathit{p}}^{k+2}_{\perp,K}(K) \\
		(\mathbf{u}_h^{\star},\mathbf{v})_K&=(\mathbf{u}_h^{\mathit{f}},\mathbf{v})_K\hspace{2em} \forall\mathbf{v}\in\bm{\mathit{p}}^k(K)
	\end{align}	
\end{subequations}
for all elements $K\in\mathcal{T}_h$. Here $\bm{\mathit{p}}^{k+2}_{\perp,K}(K)$ denotes the $\bm{\mathit{L}}^2(K)$-orthogonal complement of $\bm{\mathit{p}}^k(K)$ in $\bm{\mathit{p}}^{k+2}(K)$. The following theorem is the analysis of postprocessing for the source problem. We omit its proof as it proceeds along the same lines as a proof in \cite{R1988}.
\begin{theorem}
	\label{1222b}
	Suppose $\mathbf{u}^{\mathit{f}}$ is in $\bm{\mathit{H}}^{k+3}(\Omega)$, then
	$(\sum_{K\in\mathcal{T}_h}||\nabla(\mathbf{u}^{\mathit{f}}-\mathbf{u}_h^{\star})||_{0,K}^2)^{\frac{1}{2}}$ $\leq$ $ Ch^{k+2}$ $(|\mathbf{u}^{\mathit{f}}|_{k+3,\Omega}+|\underline{\bm{\sigma}}^{\mathit{f}}|_{k+2,\Omega}+|\underline{\bm{\rho}}^{\mathit{f}}|_{k+2,\Omega})+C||\bm{\mathit{P}}\mathbf{u}^{\mathit{f}}-\mathbf{u}_h^{\mathit{f}}||_{1,h}.$
	Besides, if we estimate $||\bm{\mathit{P}}\mathbf{u}^{\mathit{f}}-\mathbf{u}_h^{\mathit{f}}||_{0,\Omega}$ by a duality argument under proper regularity assumption(see \cite[Theorem 5.1]{CJJ2010}), then
	$||\mathbf{u}^{\mathit{f}}-\mathbf{u}_h^{\star}||_{0,\Omega}\leq Ch^{k+3}(|\mathbf{u}^{\mathit{f}}|_{k+3,\Omega}+|\underline{\bm{\sigma}}^{\mathit{f}}|_{k+2,\Omega}+|\underline{\bm{\rho}}^{\mathit{f}}|_{k+2,\Omega}).$
\end{theorem}

The post-processed eigenfunctions are obtained by first computing a mixed eigenfunction $\mathbf{u}_h$ with corresponding ($\underline{\bm{\sigma}}_h,\underline{\bm{\rho}}_h$) and then applying $\mathcal{L}$ to this pair:
$\hat{\mathbf{u}}_h=\mathcal{L}(\underline{\bm{\sigma}}_h,\mathbf{u}_h,\underline{\bm{\rho}}_h).$
We recall that if $m$ is the multiplicity of $\lambda$, then there are $m$ linearly independent eigenfunctions $\mathbf{u}_{hi},i=1,2,\cdot\cdot\cdot,m.$ of $\bm{\mathit{T}}_h$, each corresponding to the eigenvalue $\lambda_{hi}$. Then the post-processed eigenspace is defined by 
$R(\hat{\bm{\mathit{E}}}_h)={\rm span}\{\hat{\mathbf{u}}_{h1},\hat{\mathbf{u}}_{h2},\cdot\cdot\cdot,\hat{\mathbf{u}}_{hm}\}$, where $\hat{\mathbf{u}}_{hi}=\mathcal{L}(\underline{\bm{\sigma}}_{hi},\mathbf{u}_{hi},\underline{\bm{\rho}}_{hi}).$
The following theorem shows that the postprocessed eigenfunctions converge at a higher rate for sufficiently smooth eigenfunctions.  Here for $\mathbf{x}\in R(\bm{\mathit{E}})$, we define $\delta_1(\mathbf{x},R(\hat{\bm{\mathit{E}}}_h)):=\inf\limits_{\mathbf{y}\in R(\hat{\bm{\mathit{E}}}_h)}(\sum_{K\in\mathcal{T}_h}||\nabla(\mathbf{x}-\mathbf{y})||_{0,K}^2)^{\frac{1}{2}}$, $\delta_1(R(\bm{\mathit{E}}),R(\hat{\bm{\mathit{E}}}_h))=\sup\limits_{{\mathbf{y}\in R(\bm{\mathit{E}}),||\mathbf{y}||_{1,h}=1}}\delta_1(\mathbf{y},R(\hat{\bm{\mathit{E}}}_h))$, and $\hat{\delta}_1(R(\bm{\mathit{E}})$, $R(\hat{\bm{\mathit{E}}}_h))$$=\max\{\delta_1(R(\bm{\mathit{E}}),R(\hat{\bm{\mathit{E}}}_h)),\delta_1(R(\hat{\bm{\mathit{E}}}_h),R(\bm{\mathit{E}}))\}$. And recall that $\widetilde{\delta}(\cdot,\cdot)$ is the gap under $L^2$-norm.
\begin{theorem}
	\label{1222c}
	Suppose $R(\bm{\mathit{E}})\subset\bm{\mathit{H}}^{k+3}(\Omega)$ and $s_{\lambda}$ is the largest positive number such that
	\begin{equation}
		\label{1222d}
		||\underline{\bm{\sigma}}^{\mathit{f}}||_{s_{\lambda},\Omega}+||\mathbf{u}^{\mathit{f}}||_{1+s_{\lambda},\Omega}+||\underline{\bm{\rho}}^{\mathit{f}}||_{s_{\lambda},\Omega}\leq C^{reg}||\bm{\mathit{f}}||_{0,\Omega}
	\end{equation}
	holds for all $\bm{\mathit{f}}\in R(\bm{\mathit{E}})$, where $C^{reg}$ is independent of $\lambda_S$. If $s_{\lambda}\geq k+2$, then there are positive constants $h_0$ and $C$, depending on $\lambda$, such that for all $h<h_0$, the post-processed eigenspace satisfies
	$
		\hat{\delta}_1(R(\bm{\mathit{E}}),R(\hat{\bm{\mathit{E}}}_h))\leq Ch^{k+2},\widetilde{\delta}(R(\bm{\mathit{E}}),R(\hat{\bm{\mathit{E}}}_h))\leq Ch^{k+3}
	$
	where $C$ doesn't depend on $h$ and $\lambda_S$.
\end{theorem}

Before going further, we give Lemma \ref{1222e} below(we omit its proof since it's enough to follow \cite[Lemma 4.1]{CJL2010}). Define $\hat{\bm{\mathit{T}}}_h:\bm{\mathit{L}}^2(\Omega)\mapsto\bm{\mathit{L}}^2(\Omega)$ by
$\hat{\bm{\mathit{T}}}_h\bm{\mathit{f}}=\mathcal{L}(\underline{\bm{\sigma}}^{\mathit{f}}_h,{\bm{\mathit{T}}}_h\bm{\mathit{f}},\underline{\bm{\rho}}^{\mathit{f}}_h).$
\begin{lemma}
	\label{1222e}
	The nonzero eigenvalues of $\hat{\bm{\mathit{T}}}_h$ coincide with the nonzero eigenvalues of ${\bm{\mathit{T}}}_h$. Furthermore, if $\mathbf{u}_h$ is an eigenfunction of ${\bm{\mathit{T}}}_h$ such that 
	${\bm{\mathit{T}}}_h\mathbf{u}_h=\beta\mathbf{u}_h,$
	for some $\beta>0$, then 
	$\label{1222g}
	\hat{\bm{\mathit{T}}}_h\hat{\mathbf{u}}_h=\beta\hat{\mathbf{u}}_h,$
	where $\hat{\mathbf{u}}_h=\mathcal{L}(\underline{\bm{\sigma}}_h,\mathbf{u}_h,\underline{\bm{\rho}}_h)$ and $(\underline{\bm{\sigma}}_h,\underline{\bm{\rho}}_h)$ together with ${\mathbf{u}}_h$ is the  solution to the corresponding mixed method. The multiplicity of $\beta$, as an eigenvalue of ${\bm{\mathit{T}}}_h$ or $\hat{\bm{\mathit{T}}}_h$, is the same.
\end{lemma}

\begin{proof}[proof of Theorem \ref{1222c}]
	By Lemma \ref{1222e}, the postprocessed functions are eigenfunctions of the operators $\hat{\bm{\mathit{T}}}_h$. Hence going through the similar proof to Theorem \ref{0308b}, we have
	\begin{equation}
		\label{1222h}
		\hat{\delta}_1(R(\bm{\mathit{E}}),R(\hat{\bm{\mathit{E}}}_h))\leq C\sup\limits_{\mathbf{0}\neq\bm{\mathit{f}}\in R(\bm{\mathit{E}})}(\sum_{K\in\mathcal{T}_h}||\nabla({\bm{\mathit{T}}}\bm{\mathit{f}}-\hat{\bm{\mathit{T}}}_h\bm{\mathit{f}})||_{0,K}^2)^{\frac{1}{2}}/||\bm{\mathit{f}}||_{1,\Omega}.	
	\end{equation}
	\begin{equation}
		\label{0102a}
		\widetilde{\delta}(R(\bm{\mathit{E}}),R(\hat{\bm{\mathit{E}}}_h))\leq C\sup\limits_{\mathbf{0}\neq\bm{\mathit{f}}\in R(\bm{\mathit{E}})}||{\bm{\mathit{T}}}\bm{\mathit{f}}-\hat{\bm{\mathit{T}}}_h\bm{\mathit{f}}||_{0,\Omega}/||\bm{\mathit{f}}||_{0,\Omega}.	
	\end{equation}
	By Theorem \ref{1222b}, the regularity assumption (\ref{1222d}), and Theorem \ref{1221a}, we easily get
	$
	(\sum_{K\in\mathcal{T}_h}$ $||\nabla(\mathbf{u}^{\mathit{f}}-\hat{\mathbf{u}}^{\mathit{f}}_h)||_{0,K}^2)^{\frac{1}{2}}\le Ch^{k+2}(|\mathbf{u}^{\mathit{f}}|_{k+3,\Omega}+|\underline{\bm{\sigma}}^{\mathit{f}}|_{k+2,\Omega}+|\underline{\bm{\rho}}^{\mathit{f}}|_{k+2,\Omega})+C||\bm{\mathit{P}}\mathbf{u}^{\mathit{f}}-\mathbf{u}_h^{\mathit{f}}||_{1,h}\leq C h^{k+2}||\bm{\mathit{f}}||_{1,\Omega},
	$
	$
	||\mathbf{u}^{\mathit{f}}-\hat{\mathbf{u}}^{\mathit{f}}_h||_{0,\Omega}\leq Ch^{k+3}(|\mathbf{u}^{\mathit{f}}|_{k+3,\Omega}+|\underline{\bm{\sigma}}^{\mathit{f}}|_{k+2,\Omega}+|\underline{\bm{\rho}}^{\mathit{f}}|_{k+2,\Omega})\leq C h^{k+3}||\bm{\mathit{f}}||_{0,\Omega}.	
	$
	Thus $(\sum_{K\in\mathcal{T}_h}||\nabla({\bm{\mathit{T}}}\bm{\mathit{f}}-\hat{\bm{\mathit{T}}}_h\bm{\mathit{f}})||_{0,K}^2)^{\frac{1}{2}}=(\sum_{K\in\mathcal{T}_h}||\nabla(\mathbf{u}^{\mathit{f}}-\hat{\mathbf{u}}^{\mathit{f}}_h)||_{0,K}^2)^{\frac{1}{2}}\leq C h^{k+2}||\bm{\mathit{f}}||_{1,\Omega},||{\bm{\mathit{T}}}\bm{\mathit{f}}-\hat{\bm{\mathit{T}}}_h\bm{\mathit{f}}||_{0,\Omega}=||\mathbf{u}^{\mathit{f}}-\hat{\mathbf{u}}^{\mathit{f}}_h||_{0,\Omega}\leq C h^{k+3}||\bm{\mathit{f}}||_{0,\Omega}$ for all $\bm{\mathit{f}}\in R(\bm{\mathit{E}})$. Using (\ref{1222h}) and (\ref{0102a}), we directly obtain the desired assertion.
\end{proof}
\iffalse
\begin{remark}
	Though some descriptions of this section(for instance, the statement of Theorem \ref{1222c})  are similar to those in \cite{CJL2010}, the main idea is quite different from \cite{CJL2010}. In Theorem \ref{1222c}, combining discrete $H^1$-stability of numerical displacement and the analysis in section 5(especially, Theorem \ref{0308b}), we obtain an $O(h^{k+2})$ approximation to the eigenspace of exact displacement under the norm $(\sum_{K\in\mathcal{T}_h}||\nabla(\cdot)||_{0,K}^2)^{\frac{1}{2}}$, with proper regularity assumption, while only an $O(h^{k})$ approximation can be obtained if we use the analysis in \cite{CJL2010}.	
\end{remark}
\fi
\section{Numerical results}
\label{sec: Numerical results}
In this section, we provide the results of a couple of numerical tests carried out with the method based on our element proposed in section \ref{sec:The discrete problem} and with the analogue based on the conventional Arnold-Falk-Winther element, which confirm the theoretical results proved above.

In this paper, we aim to approximate numerically the stress, displacement and rotation, by piecewise 
	$(k+1)$, $k$, and $(k+1)$-th degree polynomial functions ($k\geq 1$) on special grids for linear elasticity eigenvalue problem. For this purpose, one way is to introduce a large linear system of the type $\mathbf{A}\mathbf{x}=\lambda_h\mathbf{B}\mathbf{x}$(i.e. (\ref{aa})) in section \ref{sec:The discrete problem}; The other way is to provide a condensed nonlinear, but smaller eigenproblem (\ref{1217l}) via hybridization of (\ref{1217a}) in section \ref{sec: Hybridization}. With respect to the first way (described in section \ref{sec:The discrete problem}), we solve (\ref{aa}) by utilizing the Matlab command $eigs(\mathbf{A}\backslash\mathbf{B})$. We complete the matrix assembly under the package of iFEM \cite{chenl2008} and the discrete eigenvalue problem is solved in
	MATLAB 2020a on a DELL Precision 3630 Tower with 32G memory. With respect to the second way (described in section \ref{sec: Hybridization}), we solve the nonlinear eigenproblem (\ref{1217l}) by the Newton's method based on an accurate initial approximation. Algorithmic strategies of solving the nonlinear eigenvalue problem due to hybridization proposed in section \ref{sec: Hybridization} are presented in section \ref{sec: Algorithmic strategies} below.

The material constants have been chosen $\rho_S=1$ and Young Modulus $E=1$. Let the Poisson ratio $\nu$ take different values in $(0,1/2]$. We recall that the Lam$\rm \acute{e}$ coefficients of a material are defined in terms of the
Young’s modulus $E$ and Poisson's ratio $\nu$ as
	$\lambda_S:=\frac{E\nu}{(1+\nu)(1-2\nu)}$, $\mu_S:=\frac{E}{2(1+\nu)}$.
We compute the eigenvalues and eigenfunctions considering different polynomial
degrees in the unitary square and the classic L-shaped domain. We also report some numerical results in the unitary cube. We present in the following tables an estimate of the order of convergence and more
accurate values of the vibration frequencies extrapolated from the computed ones by
means of a least-squares fitting of the model
$\omega_{hi}\approx\omega_i+C_ih^{\alpha_i}.$
This fitting has been done for each vibration mode separately. $\omega_{hi}:=\sqrt{\lambda_{hi}}$ is the computed vibration frequencies.
The fitted parameters $\omega_i$ and $\alpha_i$ are the extrapolated vibration frequency and the estimated order of
convergence, respectively.

\subsection{Unitary square} \label{sec: Unitary square} We consider an elastic body occupying the domain $\Omega:= (0,1)^2$, fixed at its top($\Gamma_0$)  and free at the rest of the boundary($\Gamma_1$). We have used Hseih-Clough-Toucher grids as shown in (a) of Figure \ref{HCT}. The refinement
parameter N used to label each grid is the number of elements on each edge.  With the boundary conditions considered in our model problem, it turns out that the values of the regularity exponents $s_0$ in Assumption \ref{0128b} are $0.6797,0.5999,0.5946$ when $\nu=0.35,0.49,0.50$, respectively.(see \cite{SDR2019} and the references therein).
\begin{figure}[htbp]
	\centering
	\subfigure[]{
		\begin{minipage}[t]{0.3\linewidth}
			\centering
			\includegraphics[width=1.5in]{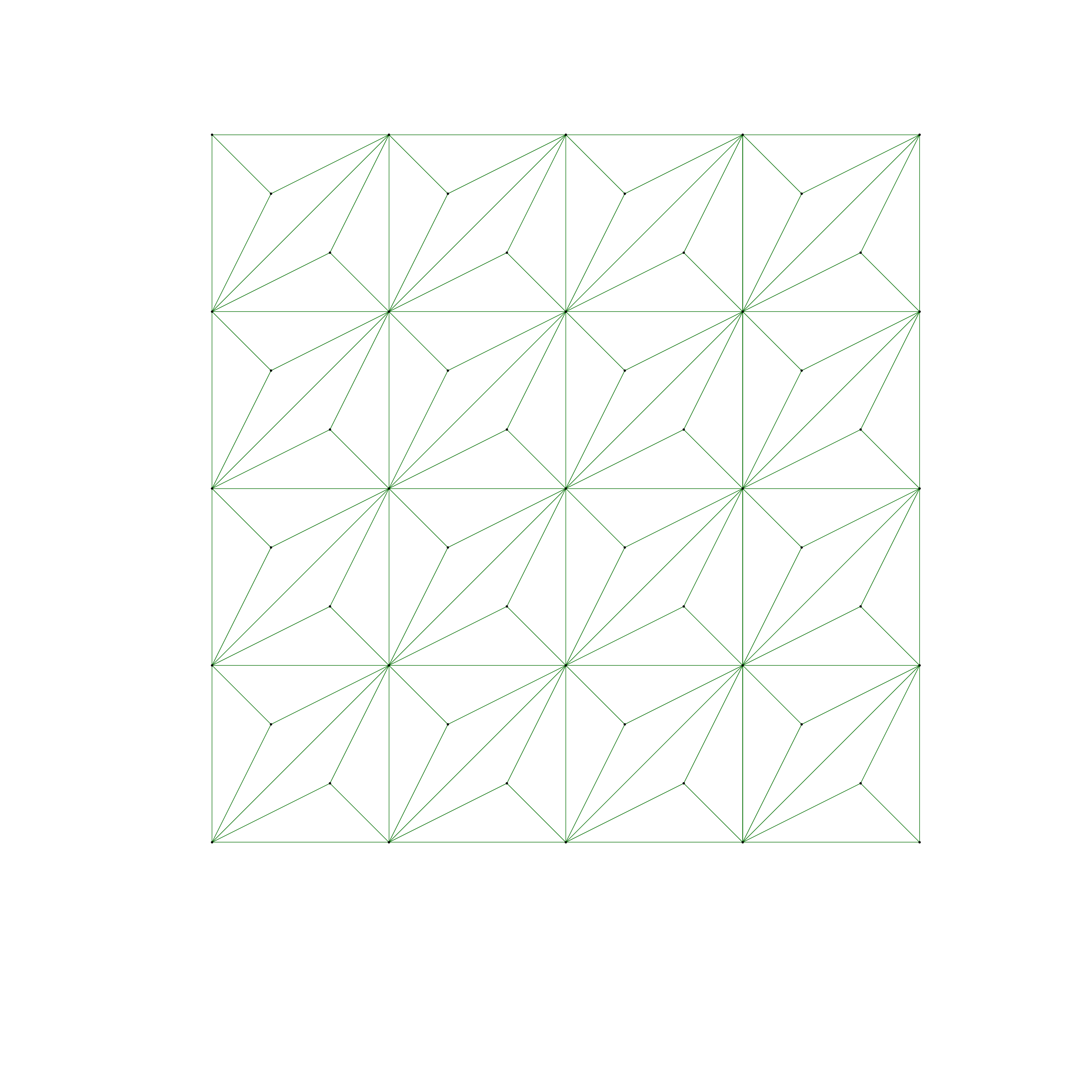}
			%\caption{fig2}
		\end{minipage}%
	} 
    \subfigure[]{
    	\begin{minipage}[t]{0.3\linewidth}
    		\centering
    		\includegraphics[width=1.5in]{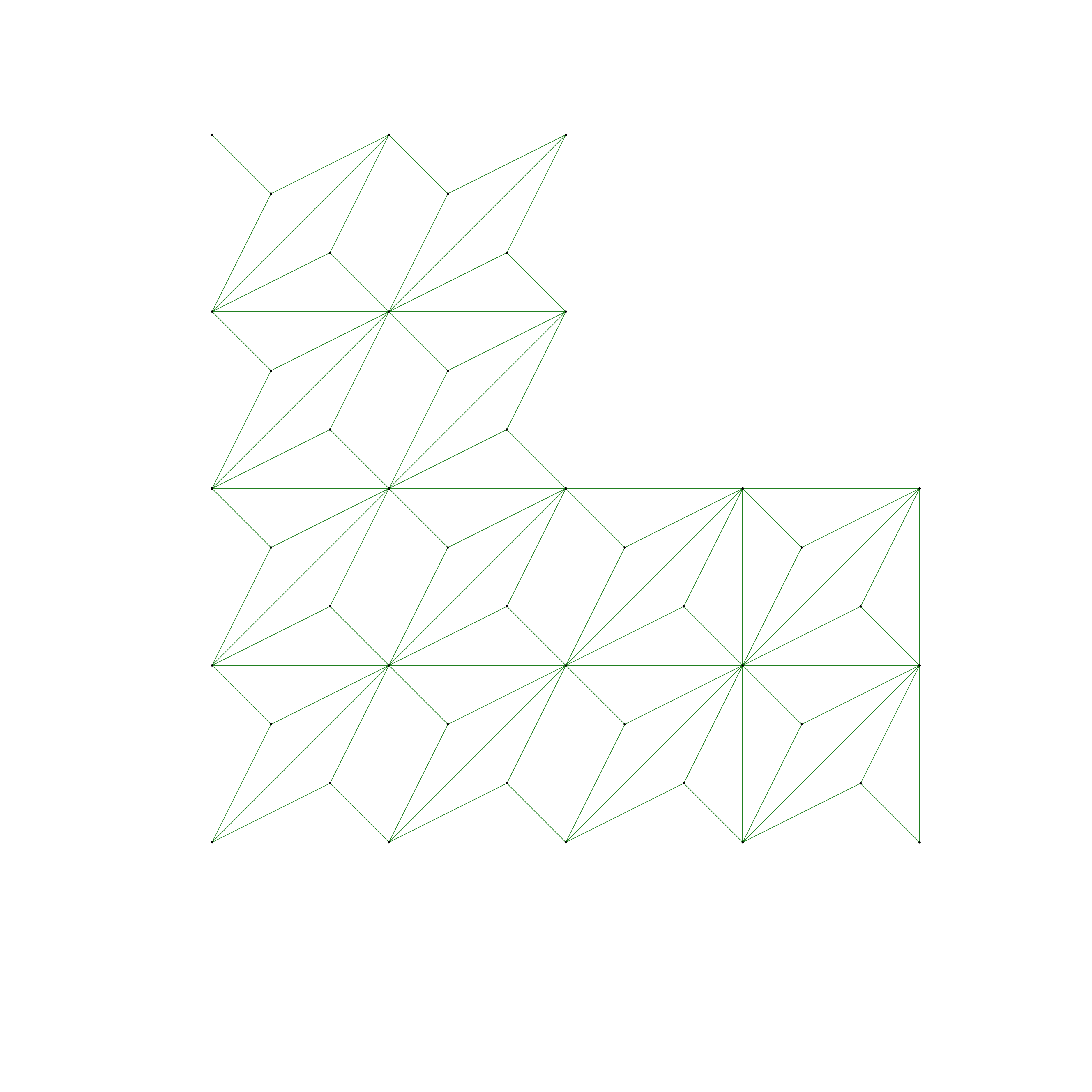}
    		%\caption{fig2}
    	\end{minipage}%
    }%
	\subfigure[]{
		\begin{minipage}[t]{0.3\linewidth}
			\centering
			\includegraphics[width=1.5in]{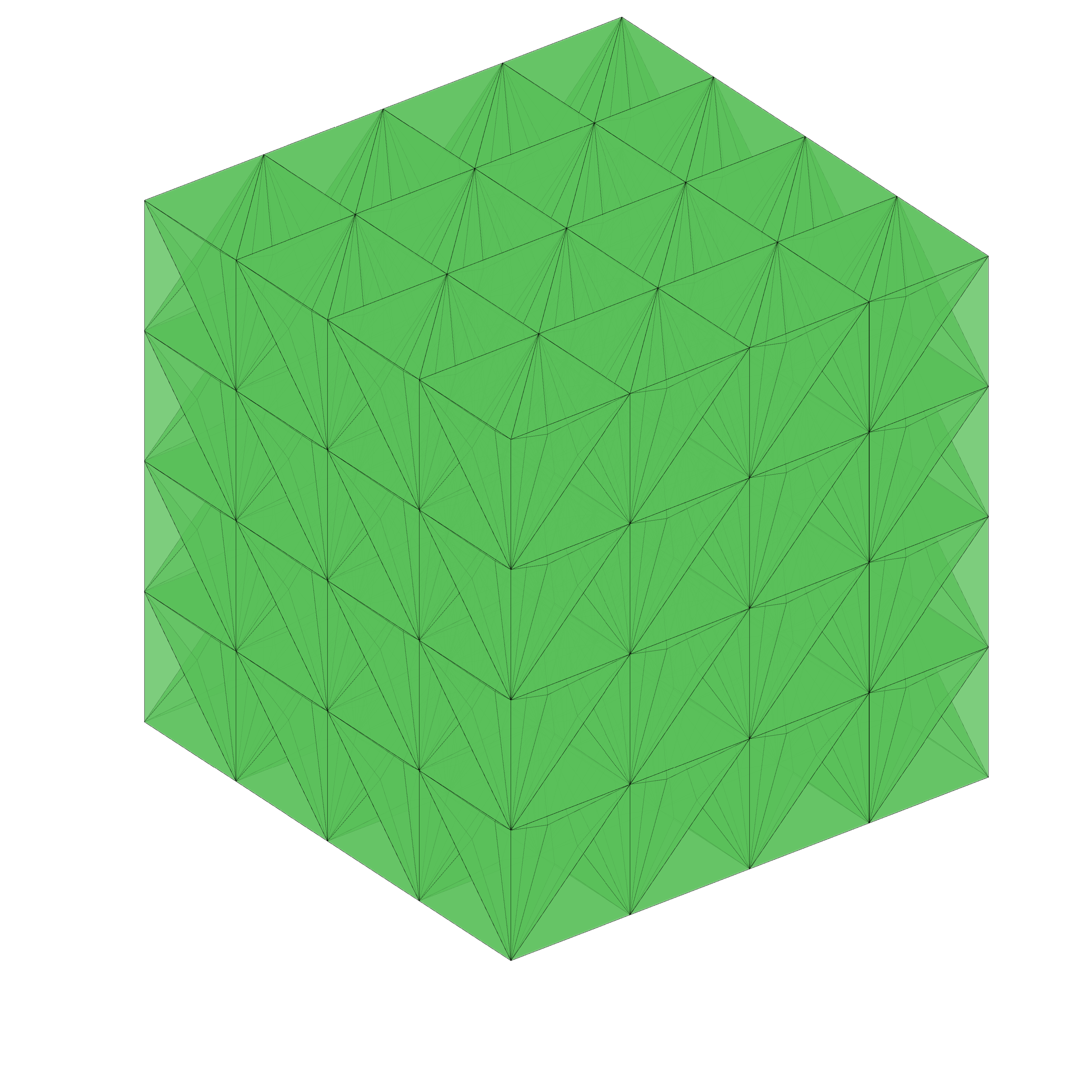}
			%\caption{fig2}
		\end{minipage}%
	}%
	\centering
	\caption{Hseih-Clough-Toucher grids in different domains}
	\label{HCT}
\end{figure}

We report in Tables \ref{comparison} and \ref{comparison_k=2}(corresponding to polynomial degrees k=1,2, respectively) the first two vibration frequencies computed on HCT grids for Poisson ratios $\nu=0.35,0.49$ and 0.50 of AFW element and ours defined in section \ref{sec:The discrete problem}. Comparing with the values of regularity exponents $s_0$ given above, we observe that there is a double order of convergence for the vibration frequencies on both AFW element and ours. That is, in all cases we have $\alpha\simeq2min\{r,k+1\}\approx2s_0$, which corresponds to the best possible order convergence for this problem. We point out that the method is clearly locking-free.

In the following test, we apply the method to a problem with smooth eigenfunctions. So the rate of convergence becomes $\alpha\simeq2min\{r,k+1\}=2(k+1)$. For this purpose, we consider a homogeneous Dirichlet condition on the whole boundary. We present in Tables \ref{smooth_1} and \ref{smooth_2} the lowest vibration frequencies computed using different Poisson ratios $\nu=0.35$ and 0.50 (for $\nu=0.49$ the results are similar) on HCT grids of AFW element and ours when $k=1$. %In particular, and for simplicity, we consider $\nu=0.35$ (for other values of $\nu$ the results are similar) when $k=2$. 
In this case, it can be clearly seen that, when using degree $k$, the order of convergence
is $2(k+1)$ as the theory predicts.
We present in Figure \ref{magnitude} plots of the first, third and fourth eigenfunctions of the spectral
problem obtained
with our method for $\nu=0.35,k=1$ and $N = 20$ in the unit square. The colors represent the magnitude of the
displacement $\mathbf{u}$ of the elastic structure.

For completeness, next we compare the 
	performance of the two schemes (AFW and Ours), where the AFW scheme 
	is used on a more conventional mesh (for instance, Left panel of \cite[Figure 1.]{XYY2020}) and the number of degrees of freedom of 
	the two methods are similar to each other. More precisely, we consider the AFW scheme on the conventional grids \cite[Figure 1.(Left panel)]{XYY2020} and our scheme (i.e. (\ref{aa})) on HCT grids. Let NE, NT denote the number of all grid edges and all elements in the unit square, respectively. A simple computation reveals the degrees of freedom of AFW scheme on the conventional grids \cite[Figure 1.(Left panel)]{XYY2020} is 6NE+15NT and the degrees of freedom of ours on HCT grids is 6NE+18NT when $k=1$. Therefore, inspired by \cite[Fig.4, Fig.7]{JA2018}, to make sure the number of degrees of freedom of 
	the two methods are similar to each other, we just need to let the number of all grid edges and all elements in the unit square be similar. %Inspired by \cite[Fig.4, Fig.7]{JA2018}, we present Hseih-Clough-Toucher grids for our scheme with NE=2055, NT=1350 and the conventional grids (\cite[Figure 1.(Left panel)]{XYY2020}) for AFW scheme with NE=2241, NT=1458 in the unit square in Figure \ref{HCT and conventional grids}.
	Then we also apply the two methods to a problem with smooth eigenfunctions, as what have been done in Tables \ref{smooth_1} and \ref{smooth_2}. %For this purpose, we consider a homogeneous Dirichlet condition on the whole boundary as before.
	We present in Tables \ref{degrees of freedom of our scheme on HCT grids} and \ref{degrees of freedom of AFW on conventional grids} the lowest vibration frequencies computed on HCT grids for our scheme and on conventional grids (\cite[Figure 1.(Left panel)]{XYY2020}) for AFW scheme, respectively. The two schemes have similar degrees of freedom at each refinement level. For simplicity, we consider $\nu=0.35$ (for
	other values of $\nu$ the results are similar) and $k=1$. In this case, it can be seen that the order of convergence of both schemes
	is $2(k+1)$ as the theory predicts.
\begin{table}
	\footnotesize
	\centering
	\caption{Computed lowest vibration frequencies and convergence order of AFW element and ours on HCT grids for k=1 in the unitary square.}
	\label{comparison}
	\begin{tabular}{|c|ccccccc|}
		\hline
		&$\nu$  & N=5 & N=10 & N=15 & N=20 & $\alpha$ & Extra \\
		\hline
		\multirow{6}*{AFW} & 0.35 &  0.6803984	&0.6806595	&0.6807324	&0.6807654	&1.30	&0.68084
		\\
		%\hline
		& &1.6987527	&1.6991041	&1.6992042	&1.6992482	&1.29	&1.69935
		\\
		%\hline
		& 0.49 &0.6987399	&0.6991534	&0.6992891	&0.6993557	&1.16	&0.69952
		\\
		%\hline
		& &1.8356588	&1.8365383	&1.8368018	&1.8369218	&1.20	&1.83722
		\\
		%\hline
		& 0.5 &0.7007181	&0.7011933	&0.7013345	&0.7014042	&1.17	&0.70157
		\\
		%\hline
		& &1.8469392	&1.8478613	&1.8481387	&1.8482654	&1.19	&1.84858
		\\
		\hline
		\multirow{6}*{Ours} & 0.35 & 0.6796493	&0.680362	&0.6805600	&0.6806486	&1.31	&0.68084
		\\
		%\hline
		& &1.6978382	&1.6987575	&1.6990058	&1.6991143	&1.37	&1.69934
		\\
		%\hline
		& 0.49 &0.6969628	&0.6983811	&0.6988158	&0.6990212	&1.21 &0.69951
		\\
		%\hline
		& &1.8329877	&1.8353825	&1.8360905	&1.8364175	&1.21	&1.83721
		\\
		%\hline
		& 0.50 &0.6988048	&0.7003704	&0.7008283	&0.7010455	&1.21	&0.70156
		\\
		%\hline 
		& &1.8440948	&1.8466189	&1.8473703	&1.8477188	&1.20	&1.84857\\
		\hline	
	\end{tabular}
	%\label{tbl:table1}
\end{table}

\begin{table}
	\footnotesize
	\centering
	\caption{Computed lowest vibration frequencies and convergence order of AFW element and ours on HCT grids  for k=2 in the unitary square.}
	\label{comparison_k=2}
	\begin{tabular}{|c|ccccccc|}
		\hline
		&$\nu$ & N=5 & N=10 & N=15 & N=20 & $\alpha$ & Extra \\
		\hline
		\multirow{6}*{AFW} & 0.35 &0.6805874	&0.6808241	&0.6809177	&0.6809213	&1.31	&0.68100
		\\
		%\hline
		& &1.6989894	&1.6992701	&1.6993784	&1.6993826	&1.34	&1.69947
		\\
		%\hline
		& 0.49 &0.6992833	&0.6999822	&0.7002879	&0.7003012	&1.17	&0.70059
		\\
		%\hline
		& & 1.8365431	&1.8376107	&1.8380733	&1.8380935	&1.18	&1.83853
		\\
		%\hline
		& 0.5 &0.7013035	&0.7020556	&0.7023866	&0.7024011	&1.16	&0.70272
		\\
		%\hline
		& &1.8478974	&1.8490555	&1.8495613	&1.8495835	&1.17	&1.85007
		\\
		\hline
		\multirow{6}*{Ours} & 0.35 & 0.6798383	&0.6800750	&0.6801686	&0.6801704	&1.33	&0.68025
		\\
		%\hline
		& &1.6980749	&1.6983555	&1.6984639	&1.6984660	&1.36	&1.69855
		\\
		%\hline
		& 0.49 &0.6975062	&0.6982050	&0.6985108	&0.6985174	&1.20	&0.69880
		\\
		%\hline
		& &1.8338720	&1.8349396	&1.8354022	&1.8354123	&1.21	&1.83583
		\\
		%\hline
		& 0.50 &0.6993902	&0.7001423	&0.7004732	&0.7004805	&1.19	&0.70079
		\\
		%\hline 
		& &1.8450530	&1.8462111	&1.8467170	&1.8467280	&1.20	&1.84719\\
		\hline	
	\end{tabular}
	%\label{tbl:table1}
\end{table}

\begin{table}
	\footnotesize
	\centering
	\caption{Computed lowest vibration frequencies and convergence order of AFW element on HCT grids for k=1 with smooth eigenfunctions in the unitary square.}
	\label{smooth_1}
	\begin{tabular}{|c|cccccc|}
		\hline
		$\nu$& N=5 & N=10 & N=15 & N=20 & $\alpha$ & Extrapolated \\
		\hline
		\multirow{6}*{0.35}&4.1938369	&4.1931527	&4.1931125	&4.1931056	&3.86	&4.19310\\
		&4.1955110	&4.1932731	&4.1931373	&4.1931136	&3.81	&4.19310\\
		&4.3774386	&4.3725322	&4.3722447	&4.3721954	&3.86 	&4.37217\\
		&5.9433488	&5.9338900	&5.9332882	&5.9331824	&3.74	&5.93312\\
		&6.1723016	&6.1559710	&6.1549785	&6.1548066	&3.81	&6.15471\\
		&6.1818336	&6.1566702	&6.1551214	&6.1548523	&3.79	&6.15471\\	
		\hline
		\multirow{6}*{0.50}&4.1825398	&4.1774873	&4.1771847	&4.1771324	&3.83	&4.17710\\
		&5.5571180	&5.5426042	&5.5417184	&5.5415644	&3.80	&5.54148\\
		&5.5659647	&5.5432475	&5.5418492	&5.5416061	&3.79	&5.54147\\
		&6.5823754	&6.5408166	&6.5380437	&6.5375519	&3.67	&6.53726\\
		&7.2318856	&7.1723532	&7.1686115	&7.1679559	&3.76	&7.16759\\
		&7.5335711	&7.4669774	&7.4627546	&7.4620126	&3.74	&7.46160
		\\
		\hline
	\end{tabular}
\end{table}

\begin{table}
	\footnotesize
	\centering
	\caption{Computed lowest vibration frequencies and convergence order of Our element on HCT grids for k=1 with smooth eigenfunctions in the unitary square.}
	\label{smooth_2}
	\begin{tabular}{|c|cccccc|}
		\hline
		$\nu$& N=5 & N=10 & N=15 & N=20 & $\alpha$ & Extrapolated \\
		\hline
		\multirow{6}*{0.35}&4.1933998	&4.1931164	&4.1931029	&4.1931016	&4.24	&4.19310\\
		&4.1942253	&4.1931663	&4.1931101	&4.1931028	&4.05	&4.19310\\
		&4.3754374	&4.3724032	&4.3722193	&4.3721874	&3.81	&4.37217\\
		&5.9382535	&5.9335228	&5.9332037	&5.9331513	&3.66	&5.93312\\
		&6.1657333	&6.1555383	&6.1548926	&6.1547793	&3.74	&6.15472\\
		&6.1719278	&6.1560049	&6.1549882	&6.1548098	&3.73	&6.15471\\
		\hline
		\multirow{6}*{0.50}&4.1803763	&4.1773465	&4.1771572	&4.1771238	&3.76	&4.17711\\
		&5.5514705	&5.5422260	&5.5416434	&5.5415407	&3.75	&5.54148\\
		&5.5574321	&5.5426655	&5.5417335	&5.5415695	&3.75	&5.54148\\
		&6.5643373	&6.5395799	&6.5377965	&6.5374736	&3.55	&6.53726\\
		&7.2115804	&7.1708805	&7.1683155	&7.1678619	&3.75	&7.16761\\
		&7.5085037	&7.4652099	&7.4623984	&7.4618995	&3.71	&7.46161
		\\
		\hline
	\end{tabular}
\end{table}
\iffalse
\begin{table}
	\centering
	\caption{Computed lowest vibration frequencies and convergence order of AFW element and ours on HCT grids  for k=2, $\nu=0.35$ with smooth eigenfunctions in the unitary square.}
	\label{smooth_3}
	\begin{tabular}{|c|cccccc|}
		\hline
		& N=5 & N=10 & N=15 & N=20 & $\alpha$ & Extrapolated \\
		\hline
		\multirow{6}*{AFW}&4.19216377	&4.19214638	&4.19214616	&4.192146069	&5.94	&4.192146097\\
		&4.199269599	&4.199169487	&4.199167764	&4.199167797	&5.83	&4.199167686\\
		&4.371197197	&4.371097086	&4.371095375	&4.371095395	&5.83	&4.37109529\\
		&5.941675684	&5.94165829	&5.941657984	&5.941657984	&5.78	&5.941657966\\
		&6.166060201	&6.165960112	&6.165958401	&6.165958401	&5.82	&6.165958305\\
		&6.17559215	&6.17549205	&6.17549035	&6.17549035	&5.83	&6.175490255
		\\
		\hline
		\multirow{6}*{Ours}&4.191726717	&4.191709328	&4.191709017	&4.19170898	&5.66	&4.191708977\\
		&4.197983882	&4.197883779	&4.197882091	&4.197881928	&5.76	&4.197881901\\
		&4.369196005	&4.369095913	&4.369094205	&4.369094051	&5.75	&4.369094018\\
		&5.936580424	&5.936563029	&5.936562723	&5.936562722	&5.78	&5.936562705\\
		&6.159491855	&6.159391746	&6.159390055	&6.159390054	&5.84	&6.15938996\\
		&6.16568639	&6.16558629	&6.16558459	&6.165584589	&5.83	&6.165584494
		\\
		\hline
	\end{tabular}
\end{table}
\fi
\begin{figure}[htbp]
	\centering
	\subfigure{
		\begin{minipage}[t]{0.3\linewidth}
			\centering
			\includegraphics[width=1.2in]{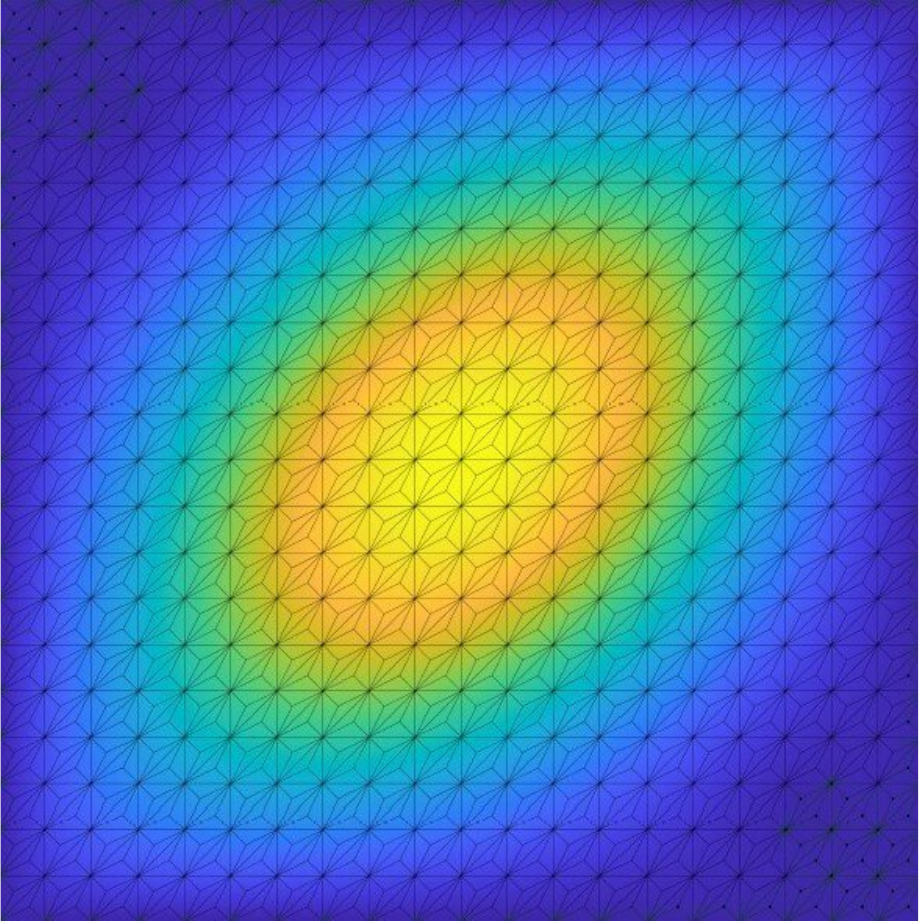}
			%\caption{fig1}
		\end{minipage}%
	}%
	\subfigure{
		\begin{minipage}[t]{0.3\linewidth}
			\centering
			\includegraphics[width=1.2in]{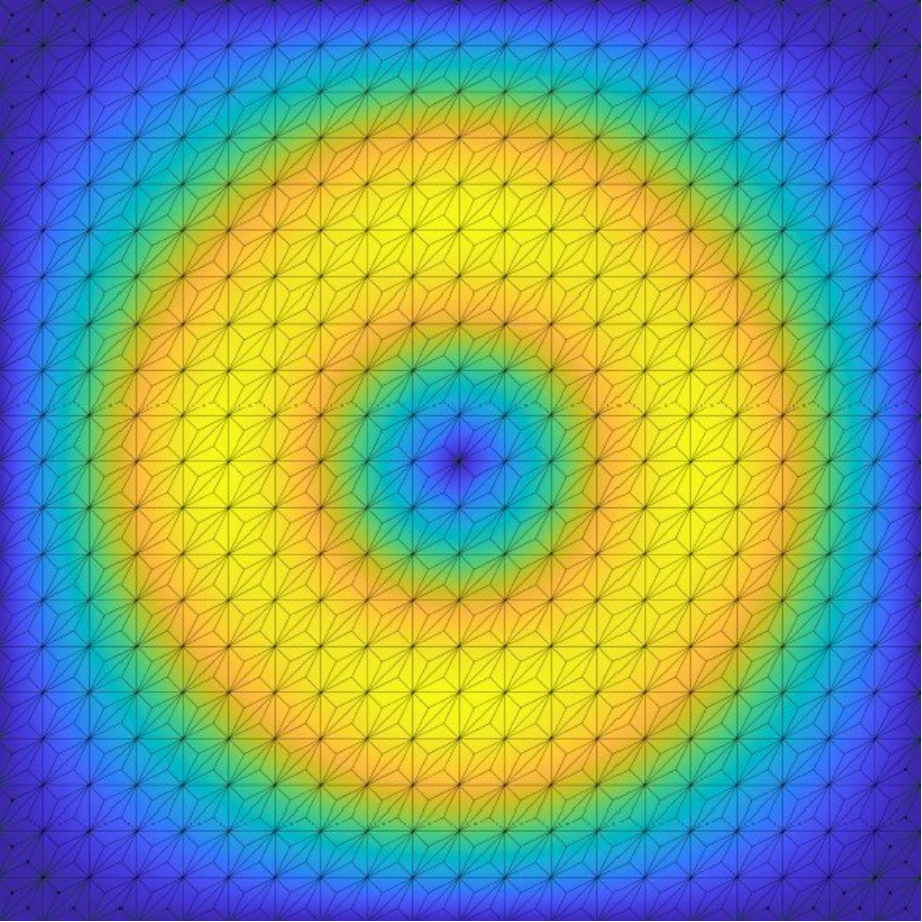}
			%\caption{fig2}
		\end{minipage}%
	}%
	\subfigure{
		\begin{minipage}[t]{0.3\linewidth}
			\centering
			\includegraphics[width=1.2in]{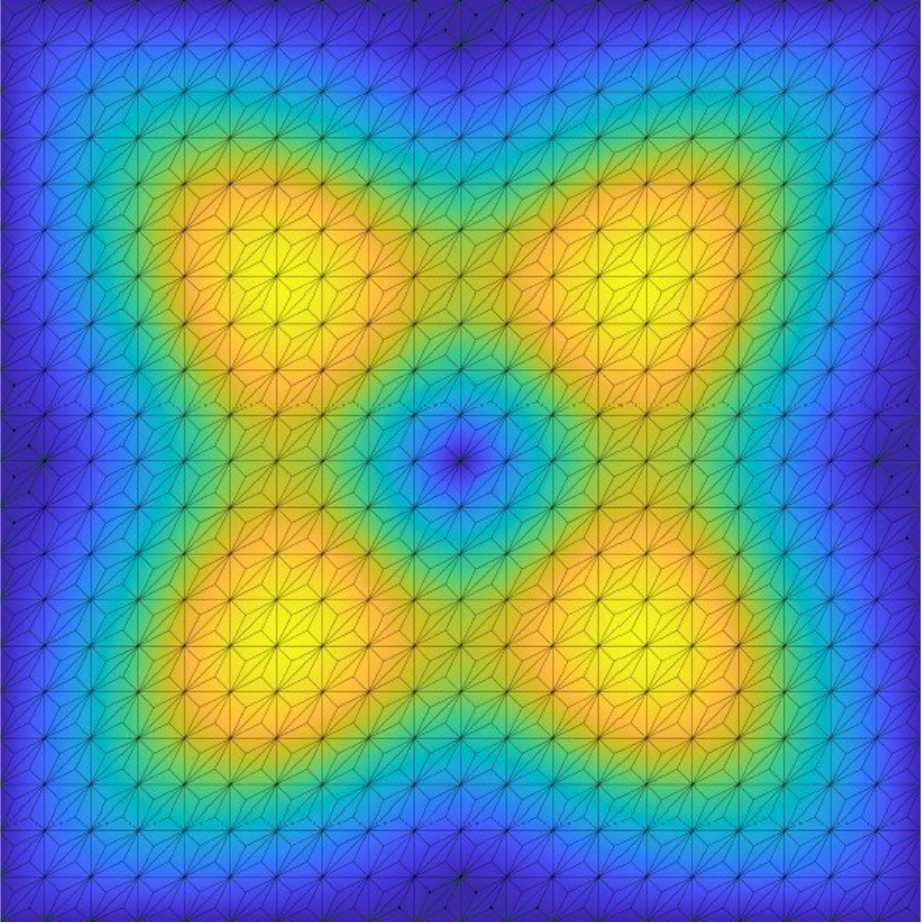}
			%\caption{fig2}
		\end{minipage}
	}%
	\centering
	\caption{Eigenfunctions corresponding to the first (upper left), third (middle) and fourth (right) computed eigenvalues with $\nu=0.35$, $N=20$ and $k=1$.}
	\label{magnitude}
\end{figure}

\begin{table}[!htbp]
	\footnotesize
	\centering
	\caption{Computed lowest vibration frequencies and convergence order of our scheme on HCT grids for k=1, $\nu=0.35$ with smooth eigenfunctions in different degrees of freedom in the unitary square.}
	\label{degrees of freedom of our scheme on HCT grids}
	\begin{tabular}{|c|c|c|c|c|c|}
		\hline		
		\multicolumn{4}{|c|}{degrees of freedom of our scheme on HCT grids} &\multirow{2}*{$\alpha$}&\multirow{2}*{Extrapolated}\\
		\cline{1-4}
		\thead{4110\\(NE=235,\\NT=150)}&\thead{16320\\(NE=920,\\NT=600)}&\thead{36630\\(NE=2055,\\NT=1350)}&\thead{65040\\(NE=3640,\\NT=2400)}&&\\
		\hline
		4.193399817	&4.193116368	&4.193102866	&4.193101643	&4.24	&4.193100462\\
		4.194225282	&4.193166286	&4.193110074	&4.193102843	&4.05	&4.193097902\\
		4.375437405	&4.372403186	&4.372219332	&4.372187356	&3.81	&4.372170305\\
		5.938253523	&5.933522844	&5.933203689	&5.93315131	&3.66	&5.933115993\\
		6.165733255	&6.155538342	&6.154892607	&6.154779336	&3.74	&6.15471561\\
		6.17192779	&6.156004919	&6.154988186	&6.15480978	&3.73	&6.154708253\\
		\hline
	\end{tabular}
\end{table}

\begin{table}[!htbp]
	\footnotesize
	\centering
	\caption{Computed lowest vibration frequencies and convergence order of AFW scheme on conventional grids(i.e.\cite[Figure 1.(Left panel)]{XYY2020}) for k=1, $\nu=0.35$ with smooth eigenfunctions in different degrees of freedom in the unitary square.}
	\label{degrees of freedom of AFW on conventional grids}
	\begin{tabular}{|c|c|c|c|c|c|}
		\hline		
		\multicolumn{4}{|c|}{degrees of freedom of AFW scheme on conventional grids} &\multirow{2}*{$\alpha$}&\multirow{2}*{Extrapolated}\\
		\cline{1-4}
		\thead{3996\\(NE=261,\\NT=162)}&\thead{15768\\(NE=1008,\\NT=648)}&\thead{35316\\(NE=2241,\\NT=1458)}&\thead{66156\\(NE=4181,\\NT=2738)}&&\\
		\hline
		4.193298538	&4.193116168	&4.19310539	&4.193103309	&3.84	&4.19310247
		\\
		4.193748133	&4.193146848	&4.193111811	&4.193105241	&3.86	&4.193102528
		\\
		4.373476812	&4.372255511	&4.372188748	&4.372176946	&3.97	&4.372172121
		\\
		5.935976791	&5.93332532	&5.933172035	&5.933143966	&3.88	&5.933132079
		\\
		6.159325043	&6.155024477	&6.154785237	&6.154742745	&3.94	&6.154725001
		\\
		6.161835403	&6.155188189	&6.154817845	&6.154752032	&3.94	&6.154724504
		\\
		\hline
	\end{tabular}
\end{table}
\subsection{L-shaped domain}\label{sec: L-shaped domain}
In this subsection we consider a non-convex domain that we
will call the L-shaped domain which is defined by $\Omega_L=(0,1)^2\backslash[1/2,1]^2$. We also fixed at its top($\Gamma_0$) and free at the rest of the boundary($\Gamma_1$). The grids have been shown in (b) of Figure \ref{HCT}.
\iffalse
\begin{figure}[htbp]
	\centering
	\subfigure[N=2]{
		\begin{minipage}[t]{0.3\linewidth}
			\centering
			\includegraphics[width=2in]{./figures/L_N=2.jpg}
			%\caption{fig1}
		\end{minipage}%
	}%
	\subfigure[N=4]{
		\begin{minipage}[t]{0.3\linewidth}
			\centering
			\includegraphics[width=2in]{./figures/L_N=4.jpg}
			%\caption{fig2}
		\end{minipage}%
	}%
	\subfigure[N=6]{
		\begin{minipage}[t]{0.3\linewidth}
			\centering
			\includegraphics[width=2in]{./figures/L_N=6.jpg}
			%\caption{fig2}
		\end{minipage}
	}%
	\centering
	\caption{Hseih-Clough-Toucher Grids in the L-shaped domain}
	\label{L_HCT}
\end{figure}
\fi
The eigenfunctions of this problem may present singularities due the reentrant
angles of the domain. In this case, the theoretical
order of convergence satisfies $2r\geq1.08$(see \cite{DFG2021} and the references therein). In Tables \ref{L_k=1}, we report the
first five vibration frequencies obtained with AFW element and ours, the respective order of
convergence and extrapolated values for $\nu=0.35$ (for
other values of $\nu$ the results are similar) and
$k=1$.

We observe from Table \ref{L_k=1} that both AFW element and ours provide a double order
of convergence for the vibration frequencies. Namely, in all cases we have $\alpha\simeq2min\{r,k+1\}$, which corresponds to the the best possible order of convergence for this problem.

\begin{table}
	\footnotesize
	\centering
	\caption{Computed lowest vibration frequencies and convergence order of AFW element on HCT grids for k=1, $\nu=0.35$ in the L-shaped domain.}
	\label{L_k=1}
	\begin{tabular}{|c|cccccc|}
		\hline
		$\nu$& N=8 & N=10 & N=14 & N=20 & $\alpha$ & Extrapolated \\
		\hline
		\multirow{6}*{AFW}&0.2848491	&0.2939103	&0.2976808	&0.3054181	&1.11	&0.31566\\
		&0.7589797	&0.7598993	&0.7609046	&0.7616218	&1.15	&0.76303\\
		&1.5525942	&1.5526532	&1.5527219	&1.5527684	&1.13	&1.55287\\
		&2.3258082	&2.3262582	&2.3267595	&2.3271234	&1.09	&2.32789\\
		&2.9934912	&2.9941597	&2.9959462	&2.9963987	&1.11	&2.99827
		\\	
		\hline
		\multirow{6}*{Ours}&0.2844861	&0.2846356	&0.2848004	&0.2849184	&1.13	&0.28516\\
		&0.7562699	&0.7577808	&0.7594440	&0.7606374	&1.13	&0.76303\\
		&1.5520395	&1.5522592	&1.5524828	&1.5526288	&1.41	&1.55285\\
		&2.3232704	&2.3244319	&2.3256113	&2.3264030	&1.38	&2.32764\\
		&2.9908387	&2.9921828	&2.9936368	&2.9946699	&1.18	&2.99664\\
		\hline
	\end{tabular}
\end{table}

For completeness, we compare the 
	performance of the two schemes (AFW and Ours), where the AFW scheme 
	is used on a more conventional mesh (i.e. \cite[Figure 1.(Middle panel)]{XYY2020}) and the number of degrees of freedom of 
	the two methods are similar to each other. %As predicted in previous subsection, to make sure the number of degrees of freedom of the two methods are similar to each other, we just need to let the number of all grid edges and all elements in the L-shaped domain be similar. %We present Hseih-Clough-Toucher grids for our scheme with NE=1351, NT=882 and the conventional grids (\cite[Figure 1.(Middle panel)]{XYY2020}) for AFW scheme with NE=1573, NT=1014 in the L-shaped domain in Figure \ref{HCT and conventional L shape grids}.
	We report in Tables \ref{degrees of freedom of our scheme on HCT grids_L} and \ref{degrees of freedom of AFW on conventional grids_L} the first two lowest vibration frequencies of both schemes for $\nu=0.35$ (for
	other values of $\nu$ the results are similar) and $k=1$. In this case, it can be seen that $\alpha\simeq2min\{r,k+1\}$, which corresponds to the best possible order of convergence for this problem. 

We end this subsection presenting plots of the second, third, and fourth eigenfunctions obtained
with our method in the L-shaped domain in Figure \ref{L_magnitude}. In particular, we show the eigenfunctions computed with $\nu=0.35,k=1$ and $N = 20$.
\begin{figure}[htbp]
	\footnotesize
	\centering
	\subfigure{
		\begin{minipage}[t]{0.3\linewidth}
			\centering
			\includegraphics[width=1.2in]{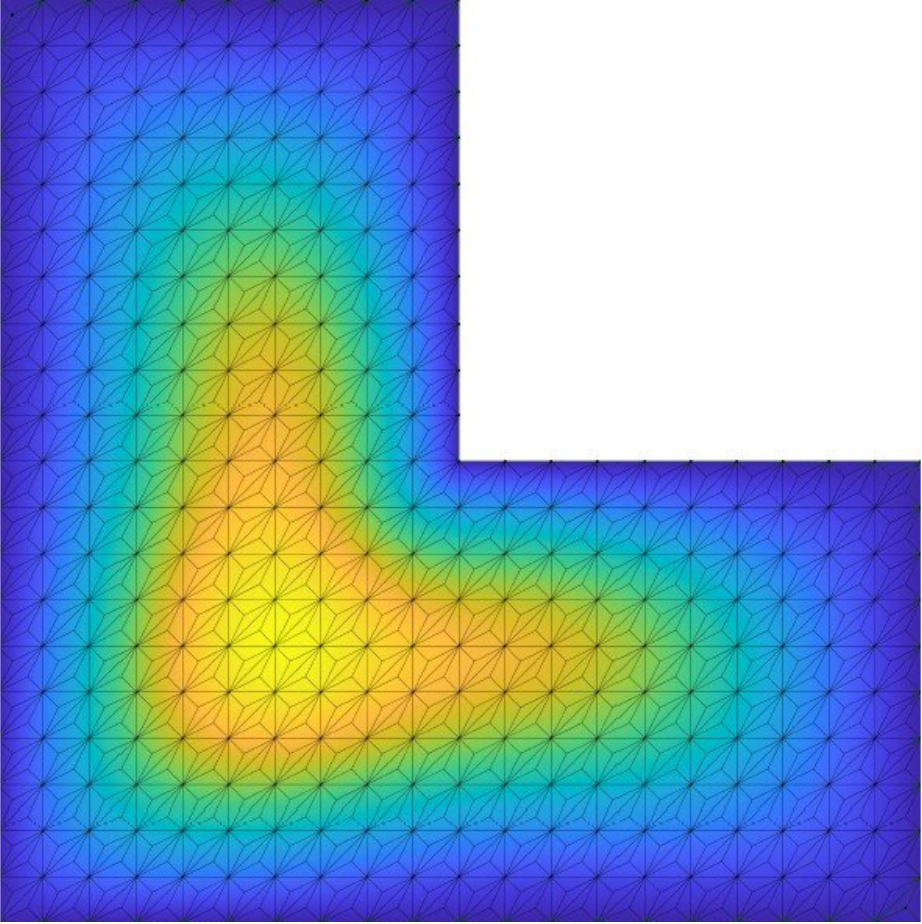}
			%\caption{fig1}
		\end{minipage}%
	}%
	\subfigure{
		\begin{minipage}[t]{0.3\linewidth}
			\centering
			\includegraphics[width=1.2in]{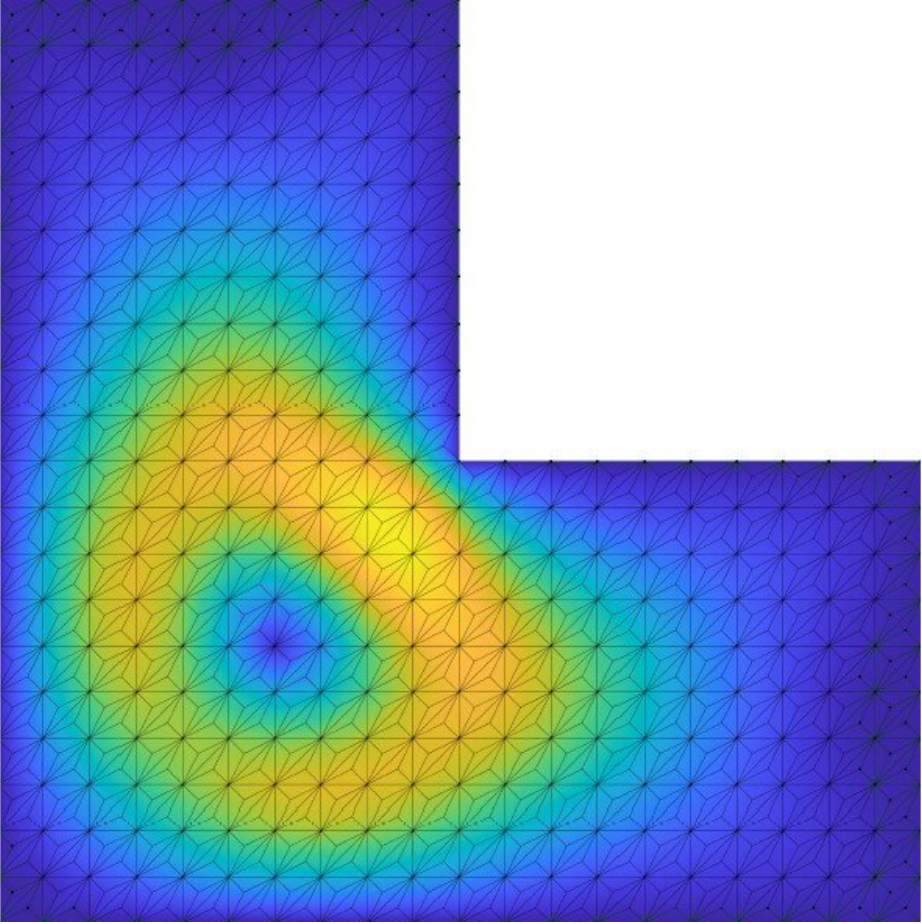}
			%\caption{fig2}
		\end{minipage}%
	}%
	\subfigure{
		\begin{minipage}[t]{0.3\linewidth}
			\centering
			\includegraphics[width=1.2in]{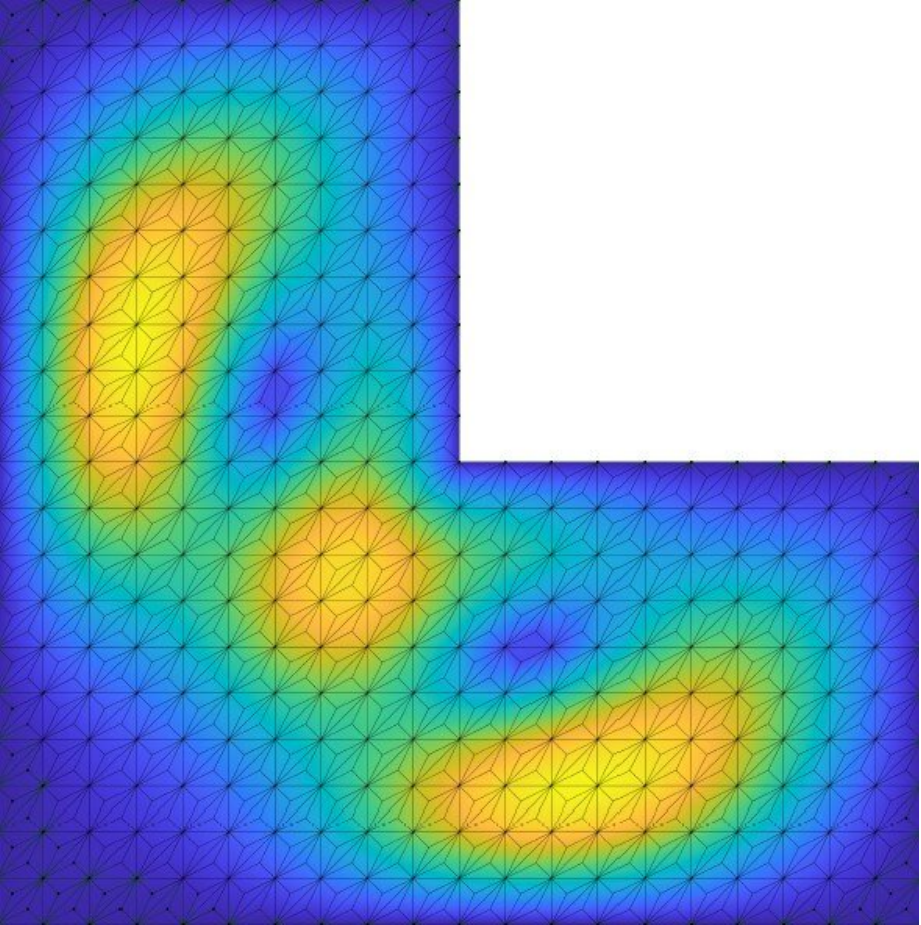}
			%\caption{fig2}
		\end{minipage}
	}%
	\centering
	\caption{Eigenfunctions corresponding to the second (upper left), third (middle) and fourth (right) computed eigenvalues with $\nu=0.35$, $N=20$ and $k=1$.}
	\label{L_magnitude}
\end{figure}

\begin{table}[!htbp]
	\footnotesize
	\centering
	\caption{Computed lowest vibration frequencies and convergence order of our scheme on HCT grids for k=1, $\nu=0.35$ in different degrees of freedom in the L-shaped domain.}
	\label{degrees of freedom of our scheme on HCT grids_L}
	\begin{tabular}{|c|c|c|c|c|c|}
		\hline		
		\multicolumn{4}{|c|}{degrees of freedom of our scheme on HCT grids} &\multirow{2}*{$\alpha$}&\multirow{2}*{Extrapolated}\\
		\cline{1-4}
		\thead{7872\\(NE=448,\\NT=288)}&\thead{12270\\(NE=695,\\NT=450)}&\thead{23982\\(NE=1351,\\NT=882)}&\thead{48840\\(NE=2740,\\NT=1800)}&&\\
		\hline
		0.284486101	&0.284635552	&0.284800373	&0.284918353	&1.13	&0.285155458\\
		0.756269944	&0.757780789	&0.759443985	&0.760637362	&1.13	&0.763034178\\
		\hline
	\end{tabular}
\end{table}

\begin{table}[!htbp]
	\footnotesize
	\centering
	\caption{Computed lowest vibration frequencies and convergence order of AFW scheme on conventional grids(i.e.\cite[Figure 1.(Middle panel)]{XYY2020}) for k=1, $\nu=0.35$ in different degrees of freedom in the L-shaped domain.}
	\label{degrees of freedom of AFW on conventional grids_L}
	\begin{tabular}{|c|c|c|c|c|c|}
		\hline		
		\multicolumn{4}{|c|}{degrees of freedom of AFW scheme on conventional grids} &\multirow{2}*{$\alpha$}&\multirow{2}*{Extrapolated}\\
		\cline{1-4}
		\thead{7224\\(NE=469,\\NT=294)}&\thead{11880\\(NE=765,\\NT=486)}&\thead{24648\\(NE=1573,\\NT=1014)}&\thead{47088\\(NE=2988,\\NT=1944)}&&\\
		\hline
		0.284920877	&0.284975711	&0.285033671	&0.285068764	&1.08	&0.285152955
		\\
		0.76038081	&0.761043364	&0.761730094	&0.762141723	&1.12	&0.763073773
		\\
		\hline
	\end{tabular}
\end{table}

\subsection{Unitary cube}\label{sec: Unitary cube}
In the following test, we consider the unitary cube 
$\Omega=(0,1)^3$ and the lowest order Arnold-Falk-Winther
finite element spaces (i.e. $\underline{\bm{\mathit{p}}}^{1}(K)$, $\bm{\mathit{p}}^0(K),\underline{\bm{\mathit{A}}}^{0}(K)$. Notice that all analyses in section \ref{sec:The discrete problem} and \ref{sec: Spectral approximation} can cover this case, see Remarks \ref{apply to other elements} and \ref{apply to others} for more details). The grids have been shown in (c) of Figure \ref{HCT}.
\iffalse
\begin{figure}[htbp]
	\centering
	\subfigure[N=2]{
		\begin{minipage}[t]{0.3\linewidth}
			\centering
			\includegraphics[width=2in]{./figures/C_N=2.jpg}
			%\caption{fig1}
		\end{minipage}%
	}%
	\subfigure[N=4]{
		\begin{minipage}[t]{0.3\linewidth}
			\centering
			\includegraphics[width=2in]{./figures/C_N=4.jpg}
			%\caption{fig2}
		\end{minipage}%
	}%
	\subfigure[N=6]{
		\begin{minipage}[t]{0.3\linewidth}
			\centering
			\includegraphics[width=2in]{./figures/C_N=6.jpg}
			%\caption{fig2}
		\end{minipage}
	}%
	\centering
	\caption{Hseih-Clough-Toucher Grids in the unit cube}
	\label{cube}
\end{figure}
\fi
In Table \ref{3D}, we report the computed lowest vibration frequencies from $N=2$ to $N=5$ for different Poisson ratios. We cannot compute more refinement of the grids to make the convergence order stable because of the computer memory limitation. 
\begin{table}
	\footnotesize
	\centering
	\caption{Computed lowest vibration frequencies of the lowest order AFW element in the unitary cube.}
	\label{3D}
	\begin{tabular}{|c|cccc|}
		\hline
		$\nu$ & N=2 & N=3 & N=4 & N=5 \\
		\hline
		\multirow{6}*{0.35}&4.966970743	&4.710164711	&4.608767817	&4.558358596\\
		&5.088532532	&4.784707885	&4.657825328	&4.592680299\\
		&5.088532532	&4.784707885	&4.657825328	&4.592680299\\
		&5.717989249	&5.337325264	&5.106323175	&4.990902401\\
		&5.868937316	&5.446359377	&5.175519032	&5.037597398\\
		\hline
		\multirow{6}*{0.50}&5.510608838	&5.098093559	&4.878228891	&4.767062214\\
		&5.646994176	&5.206697614	&4.946591774	&4.812860194\\
		&5.646994176	&5.206697614	&4.946591774	&4.812860194\\
		&6.565845408	&6.425771429	&6.100930923	&5.917027855\\
		&7.561862492	&6.569398314	&6.275577054	&6.079412616\\
		\hline
	\end{tabular}
\end{table}
\subsection{Some details of the nonlinear eigenvalue problem}\label{sec: Some details of the nonlinear eigenvalue problem} In section \ref{sec: Hybridization}, we introduce a hybridization to reduce the mixed method to a condensed nonlinear eigenproblem. Now we compare global degree freedom with different methods and provide algorithmic strategies for solution of the nonlinear eigenproblem (\ref{1217l}).

\subsubsection{Comparison of global degree freedom with different methods}\label{sec: Comparison of global degree freedom with different methods} In this subsection, we give the
numerical comparison of global degree freedom for the element of \cite{SDR2013,SDR2019}, our element described in section \ref{sec:The discrete problem} and the condensed nonlinear eigenproblem in Tables \ref{global degree freedom 2D} and \ref{global degree freedom 3D}. Observing Tables \ref{global degree freedom 2D} and \ref{global degree freedom 3D}, we find though the global degree freedom of our element described in section \ref{sec:The discrete problem} is larger than that in \cite{SDR2013,SDR2019}, the degree freedom of the condensed nonlinear eigenvalue problem after hybridization is much smaller than
the other two methods.
\begin{table}
	\footnotesize
	\centering
	\caption{Global degree freedom of trigonal finite element spaces with different methods.}
	\label{global degree freedom 2D}
	\begin{tabular}{|c|c|c|c|c|c|}
		\hline
		Method	&Formula  &k=1  &k=2  &k=3  &k=4    \\
		\hline
		\thead{ \cite{SDR2013,SDR2019}($\underline{\bm{\mathit{p}}}^{k+1}(K), \underline{\bm{\mathit{A}}}^{k}(K)$)}&\thead{$2(k+3)(k+2)$+$(k+2)(k+1)/2$}  &27  &46  &70  &99    \\
		\hline
		\thead{Our element($\underline{\bm{\mathit{p}}}^{k+1}(K)$,\\$\bm{\mathit{p}}^k(K),\underline{\bm{\mathit{A}}}^{k+1}(K)$)}	&\thead{$2(k+3)(k+2)+(k+2)(k+1)$\\+$(k+3)(k+2)/2$}  &36  &62  &95  &135    \\
		\hline
		\thead{Nonlinear eigenproblem \\($\bm{\mathit{p}}^k(F),F\in\partial K$)}	&$3(k+2)$  &9  &  12& 15 & 18   \\
		\hline
	\end{tabular}
\end{table}
\begin{table}
	\footnotesize
	\centering
	\caption{Global degree freedom of tetrahedral finite element spaces with different methods.  Here $p_k$ denotes $dim(\underline{\bm{\mathit{p}}}^{k+1}(K))=3(k+4)(k+3)(k+2)/2.$}
	\label{global degree freedom 3D}
	\begin{tabular}{|c|c|c|c|c|c|}
		\hline
		Method	&Formula  &k=1  &k=2  &k=3  &k=4    \\
		\hline
		\thead{ \cite{SDR2013,SDR2019}($\underline{\bm{\mathit{p}}}^{k+1}(K), \underline{\bm{\mathit{A}}}^{k}(K)$)}&\thead{$p_k+(k+3)(k+2)(k+1)/2$}  &102  &210  &375  &609    \\
		\hline
		\thead{Our element($\underline{\bm{\mathit{p}}}^{k+1}(K)$,\\$\bm{\mathit{p}}^k(K),\underline{\bm{\mathit{A}}}^{k+1}(K)$)}	&\thead{$p_k+(k+3)(k+2)(k+1)/2$\\+$(k+4)(k+3)(k+2)/2$}  &132  &270  &480  &777    \\
		\hline
		\thead{Nonlinear eigenproblem \\($\bm{\mathit{p}}^k(F),F\in\partial K$)}	&$2(k+3)(k+2)$  &24  &40 &60 &84   \\
		\hline
	\end{tabular}
\end{table}

\subsubsection{Algorithmic strategies}\label{sec: Algorithmic strategies} In this subsection, we discuss an algorithm for solution of the nonlinear eigenproblem (\ref{1217l}) based on an accurate initial approximation. We obtain the accurate initial approximation by solving a standard eigenproblem, namely, the perturbed problem analyzed in section \ref{sec: The perturbed eigenvalue problem}. Therefore, Newton's method is well suited for solving (\ref{1217l}), as discussed below. Similar algorithmic strategies have been shown in \cite{CJL2010}.

Hybridization of (\ref{1217a}) gives rise to the nonlinear eigenproblem (\ref{1217l}). We will recast this as a problem of finding the
	zero of a differentiable function and apply the Newton iteration. Define the operator $\bm{\mathit{F}}:\bm{\mathit{M}}^h\mapsto\bm{\mathit{M}}^h$ ($\bm{\mathit{M}}^h$ is defined in section \ref{sec: Hybridization}) by
	$\left \langle \bm{\mathit{F}}\bm{\mu}, \bm{\gamma}\right \rangle=a_h(\bm{\mu}, \bm{\gamma}),$ for all $\bm{\mu}, \bm{\gamma}\in\bm{\mathit{M}}^h,$
	where $\left \langle \cdot, \cdot\right \rangle$ denotes the $\bm{\mathit{L}}^2(\varepsilon_h)$-innerproduct. Also, define $\bm{\mathit{M}}(\lambda)$ to be the operator-valued function of $\lambda$ given by (see section \ref{sec: Hybridization} for the definitions of $\bm{\mathcal{Q}}_2,\bm{\mathcal{Q}}_2^L$)
	\begin{equation}
		\label{1220k}
		\left \langle \bm{\mathit{M}}(\lambda)\bm{\mu}, \bm{\gamma}\right \rangle=((\bm{\mathit{I}}-\lambda\bm{\mathcal{Q}}_2^L)^{-1}\bm{\mathcal{Q}}_2\bm{\mu},\bm{\mathcal{Q}}_2\bm{\gamma}).
	\end{equation}
	The nonlinear eigenproblem then takes the following form: Find $\bm{\beta}\in\bm{\mathit{M}}^h$ and $\lambda>0$ satisfying
	\begin{equation}
		\label{1220i}
		\bm{\mathit{H}}(\bm{\beta},\lambda)\equiv
		\begin{pmatrix}
			\, \bm{\mathit{F}}\bm{\beta}-\lambda\bm{\mathit{M}}(\lambda)\bm{\beta}   \,\  \\
			\, \left \langle \bm{\beta}, \bm{\beta}\right \rangle-1 
		\end{pmatrix}=\mathbf{0}.
	\end{equation}
	The first equation of above system is the same as (\ref{1217l}), while the second is a normalization condition. We apply Newton's method to solve (\ref{1220i}). Calculating the Fr$\rm \acute{e}$chet derivative of $\bm{\mathit{H}}$ at an arbitrary $(\bm{\beta},\lambda)$ and writing down the Newton iteration, we find that the next iteration $(\bm{\beta}',\lambda')$ is defined by
	\begin{equation}
		\label{1220j}
		\begin{pmatrix}
			\, (\bm{\mathit{F}}-\lambda\bm{\mathit{M}}(\lambda))(\bm{\beta}'-\bm{\beta})-(\lambda'-\lambda)\bm{\mathit{N}}(\lambda)\bm{\beta}   \,\  \\
			\, 2\left \langle \bm{\beta}, \bm{\beta}'-\bm{\beta}\right \rangle 
		\end{pmatrix}=-\bm{\mathit{H}}(\bm{\beta},\lambda),
	\end{equation}
	where $\bm{\mathit{N}}(\lambda)=\bm{\mathit{M}}(\lambda)+\lambda\frac{d\bm{\mathit{M}}(\lambda)}{d\lambda}.$  We also have
	$\left \langle \bm{\mathit{N}}(\lambda)\bm{\mu}, \bm{\gamma}\right \rangle=((\bm{\mathit{I}}-\lambda\bm{\mathcal{Q}}_2^L)^{-2}\bm{\mathcal{Q}}_2\bm{\mu},\bm{\mathcal{Q}}_2\bm{\gamma}). $
	Combining (\ref{1220i}), (\ref{1220j}) can be rewritten as
	\begin{subequations}
		\label{1220m}
		\begin{align}		
			(\bm{\mathit{F}}-\lambda\bm{\mathit{M}}(\lambda))\bm{\beta}'&=(\lambda'-\lambda)\bm{\mathit{N}}(\lambda)\bm{\beta}, \label{1220m_a}\\
			\left \langle \bm{\beta}', \bm{\beta}\right \rangle&=1 \label{1220m_b}
		\end{align}	
	\end{subequations}
	Observe that (\ref{1220m_a}) implies that $\bm{\beta}'$ depends linearly on $\lambda'-\lambda$. Hence we can decouple the above system and rearrange the computations, as stated in Algorithm \ref{alg:A}. Step 1 of Algorithm \ref{alg:A} gives good initial approximations, as already established
	in Theorem \ref{1217jj}. The value of $\delta_{\lambda}$, equaling the difference of successive eigenvalue
	iterations, is determined by (\ref{1220m_b}).
	\begin{algorithm}
		\caption{Newton's method for solving the nonlinear eigenproblem}  
		\label{alg:A}
		To solve for a nonlinear eigenvalue and eigenfunction satisfying (\ref{1217l}), proceed as follows:
		\begin{itemize}
			\item[1.] First obtain an initial approximation $\bm{\beta}_0$ and $\lambda_0$ by solving the linear eigenproblem
			$\bm{\mathit{F}}\bm{\beta}_0=\lambda_0\bm{\mathit{M}}(0)\bm{\beta}_0.$
			\item[2.] For $n=0,1,2,...,$ until convergence, perform the steps as follows:
			\begin{itemize}
				\item[(a)] Compute $\hat{\bm{\beta}}$ by solving the linear system
				\begin{equation}
					\label{1220n}
					(\bm{\mathit{F}}-\lambda_n\bm{\mathit{M}}(\lambda_n))\hat{\bm{\beta}}=\bm{\mathit{N}}(\lambda_n)\bm{\beta}_n.	
				\end{equation}
				\item[(b)] Set $\delta_{\lambda}=1/\left \langle \hat{\bm{\beta}}, \bm{\beta}_n\right \rangle$.
				\item[(c)] Update the eigenvalue: $\lambda_{n+1}=\lambda_n+\delta_{\lambda}.$
				\item[(d)] Update the nonlinear eigenfunction: $\bm{\beta}_{n+1}=\delta_{\lambda}\hat{\bm{\beta}}.$
			\end{itemize}
		\end{itemize} 
\end{algorithm}

\subsection{Computational costs of the compared methods}\label{sec: Computational costs of the compared methods} Recall that we use MATLAB 2020a to compute on a DELL Precision 3630 Tower with 32G memory. We compare the computational costs of our element with that of the conventional AFW element for $k=1$ and different Poisson's ratio $\nu$ on HCT grids in the unitary square and L-shaped domain in Tables \ref{Computational cost of unitary square} and \ref{Computational cost of L shaped domain}. Observing Tables \ref{Computational cost of unitary square} and \ref{Computational cost of L shaped domain}, we find that the computational costs of both methods are similar but the computational cost of our element is a little larger than that of AFW element. In fact, this is to be expected since the global degree freedom of our element described in section \ref{sec:The discrete problem}($\underline{\bm{\mathit{p}}}^{k+1},\bm{\mathit{p}}^k,\underline{\bm{\mathit{A}}}^{k+1}$) is larger than that of AFW element($\underline{\bm{\mathit{p}}}^{k+1},\bm{\mathit{p}}^k,\underline{\bm{\mathit{A}}}^{k}$).
\begin{table}
	\footnotesize
	\centering
	\caption{Computational cost(times($s$)) of  AFW element and ours for $k=1$ and different Poisson's ratio in the unitary square}
	\label{Computational cost of unitary square}
	\begin{tabular}{|c|ccccc|}
		\hline
		$\nu$ &Method &N=5  &N=10  &N=15  &N=20  \\
		\hline
		$0.35$&AFW  &2.346391
		&68.712579
		&405.645936
		&2167.479581
		\\
		&Ours  &2.936615
		&86.011346
		&518.568172
		&2527.466308
		\\
		\hline
		$0.50$&AFW  &4.237313
		&74.888692
		&475.002422
		&4371.372905
		\\
		&Ours  &5.289374
		&92.785181
		&621.320673
		&5788.094644
		\\
		\hline
	\end{tabular}
\end{table}
\begin{table}[H]
	\footnotesize
	\centering
	\caption{Computational cost(times($s$)) of AFW element and ours for $k=1$ and different Poisson's ratio in the L-shaped domain}
	\label{Computational cost of L shaped domain}
	\begin{tabular}{|c|ccccc|}
		\hline
		$\nu$ & Method &N=8  &N=10  &N=14  &N=20  \\
		\hline
		$0.35$&AFW  &10.394203
		&33.732006
		&151.975342
		&723.244933
		\\
		&Ours  &13.279231
		&41.996868
		&197.698534
		&903.910458
		\\
		\hline
		$0.50$&AFW  &15.50252
		&39.449416
		&179.962158
		&874.177556
		\\
		&Ours  &19.691046
		&50.040814
		&228.446985
		&1167.807067
		\\
		\hline
	\end{tabular}
\end{table}

\bibliographystyle{siamplain}
\bibliography{references2}

\end{document}